\documentclass[a4paper, 11pt, reqno]{amsart}
\usepackage[T1]{fontenc}
\usepackage[utf8]{inputenc}
\usepackage[english]{babel}
\usepackage{hyperref}
\usepackage{xcolor}
\usepackage{soul}
\theoremstyle{definition}
\newtheorem{defin}{Definition}[section]
\newtheorem{ex}[defin]{Example}
\newtheorem{exx}[defin]{Examples}
\theoremstyle{plain}
\newtheorem{theo}[defin]{Theorem}
\newtheorem{lemma}[defin]{Lemma}
\newtheorem{obs}[defin]{Remark}
\newtheorem{prop}[defin]{Proposition}
\newtheorem{cor}[defin]{Corollary}

\newtheorem*{theo-intro}{Theorem}
\newtheorem{theorem}{Theorem}

\renewenvironment{abstract}
{\par\noindent\textbf{\abstractname.}\ \ignorespaces}
{\par\medskip}

\title[Packing in locally CAT$(\kappa)$ spaces]{Packing and doubling in metric spaces \\ with curvature bounded above}
\author{Nicola Cavallucci, Andrea Sambusetti}
\address{Nicola Cavallucci and Andrea Sambusetti, Dipartimento di Matematica “Guido
	Castelnuovo”, SAPIENZA Universita di Roma, Piazzale Aldo Moro 5, I-00185 `
	Roma.}
\email{cavallucci@mat.uniroma1.it, sambuset@mat.uniroma1.it}
\date{}

\begin{document}
\maketitle

\footnotesize
\begin{abstract}
We study locally compact, locally geodesically complete, locally CAT$(\kappa)$ spaces (GCBA$^\kappa$-spaces). 
We prove a Croke-type local volume estimate  only depending on  the dimension of these spaces.
We show that a local doubling condition, with respect to the natural measure, implies pure-dimensionality.
Then we consider GCBA$^\kappa$-spaces satisfying a  uniform packing  condition at some fixed scale $r_0$  or a doubling condition at arbitrarily small scale, and prove several compactness results  with respect to pointed Gromov-Hausdorff convergence.
Finally, as a particular case, we study  convergence and stability of   $M^\kappa$-complexes with bounded geometry. 
\end{abstract}
\normalsize

\tableofcontents
\pagebreak

\section{Introduction}
\label{sec-introduction}
\enlargethispage{3mm}
Metric spaces  with curvature bounded from above
% ({\em CBA spaces}, for short)
are currently one of the main topics in metric geometry. 
They have been studied from various points of view during the last decades. In general these metric spaces can be very wild and the local geometry can be difficult to understand. Under basic additional assumptions (local compactness and local geodesic completeness) it is possible to control much better the local and asymptotic properties of these   spaces, as proved by Kleiner  and Lytchak-Nagano. 
In particular, under these assumptions,   the topological  dimension coincides with the Hausdorff dimension,
the local dimension can be detected from the tangent cones
%  and the maximal dimension of some embedded open, Euclidean ball,
and there exist a decomposition of $X$ in $k$-dimensional subspaces $X^k$ (containing dense open subsets locally bilipschitz equivalent to $\mathbb R^k$ and admitting a regular Riemannian metric), and  a canonical measure $\mu_X$,  coinciding with the restriction  of the $k$-dimensional  Hausdorff measure on each  $X^k$, which is positive and finite on any open relatively compact subset (cp. the foundational works  \cite{Kl99} , \cite{LN19}, \cite{LN-finale-18}).
Following  \cite{LN19}  we will call for short  {\em GCBA-spaces}
%\footnote{In \cite{LN-finale-18}, the authors assume that GCBA-spaces are locally compact and separable.
%However, separability is used in \cite{LN-finale-18} only for
% {\em global} statements (cp. the end of section 5.1 therein), which we won't need. On the other hand,  local compactness will follow in all our statements from the packing condition, as we will explain in Sec.\ref{sec-packingcovering}.}
the locally geodesically complete, locally compact and separable metric spaces satisfying some curvature upper bound, i.e. which are locally CAT$(\kappa)$ for some $\kappa$. When we want to emphasize in a statement the role of $\kappa$ we will write GCBA$^\kappa$.
%{\color{red} -- quindi cambiare ovunque in GCBA}

GCBA-spaces   arise in a natural way as generalizations and limits of Riemannian manifolds with sectional curvature bounded from above.
However many geometric results   in the Riemannian setting such as convergence and finiteness theorems, Margulis' lemma etc. also require lower bounds  on the  curvature.  
For instance, by Bishop-Gromov's Theorem,  a lower bound on the Ricci curvature implies a bound of the complexity of the manifold as a metric space:  namely a lower bound $\text{Ric}_X \geq -a^2$ implies a uniform estimate of the  {\em packing function} at any fixed scale $r_0$.  

We recall that  a metric space $(X,d)$ satisfies the  {\em $P_0$-packing condition at scale $r_0>0$} if all balls of radius $3r_0$ contain  at most $P_0$ points that are $2r_0$-separated  from each other (this can be equivalently expressed in terms of  coverings with balls, see Sec.\ref{sec-packingcovering}).  
Also, we will  say that a  metric-measured space $(X,d,\mu)$ satisfies a {\em $D_0$-doubling condition  up to scale $r_0>0$}  if for any $0<r\leq r_0$ and for any $x\in X$ it holds $$\frac{\mu(B(x,2r))}{\mu(B(x,r))}\leq D_0.$$
% see Sec. \S\ref{doubling}. 
% $X$ satisfies a $C_0$-covering condition at scale $r_0$ if any ball of radius $3r_0$ can be covered by at most $C_0$ balls of radius $r_0$.
%Other examples are Cheeger-Colding's and  Kapovich-Wilking's generalized Margulis' lemma for Riemannian manifolds  with Ricci curvature bounded from below, or their virtual nilpotence theorems for $n$-dimensional manifolds satisfying  $Ric_X \cdot diam(X) \geq -\epsilon(n)$. 
From a  metric-geometry perspective  the original interest in studying  metric spaces satisfying a packing condition at {\em arbitrarily small} scales is Gromov's famous Precompactness Theorem \cite{Gr81}. Another major outcome involving packing is Gromov's celebrated result on groups with polynomial growth, as extended by    Breuillard-Green-Tao \cite{BGT11} (cp. also the previous results   \cite{Kle2} and  \cite{ST}), which  shows that  a uniform bound of the packing or doubling constant for $X$ at {\em arbitrarily large} scale (or even at fixed, sufficiently large scale with respect to the diameter) yields an even stronger limitation on the complexity of the fundamental group of $X$,  that is  almost-nilpotency.   
We will see soon (cp. Theorem \ref{intro-doub} below) that, for GCBA-spaces,  a doubling condition for the canonical measure $\mu_X$ at {\em arbitrarily small} scales also has interesting consequences on the local structure of  \nolinebreak  $X$. 
\vspace{1mm}

The first key-result of the paper is a Croke-type local volume estimate for  GCBA-spaces of dimension bounded above for balls of radius smaller than the {\em almost-convexity radius}:
 \enlargethispage{6mm}

\begin{theorem}[Theorem \ref{boundbelowvolume}] ${}$ 
	\label{intro-croke}
	
	\noindent For any complete \textup{GCBA} space $X$ of dimension $\leq n_0$ and   any ball of radius $r <\min\lbrace\rho_{\textup{ac}} (X), 1\rbrace$ it holds:
	\vspace{-4mm}
	
	\begin{equation}
		\label{lowerestimate}	
		\mu_X(B(x,r))\geq c_{n_0} \cdot r^{n_0}
	\end{equation}
	
	\vspace{-1mm}
	\noindent where $c_{n_0}$ is a constant depending  only on the dimension $n_0$. 
\end{theorem}

%\noindent Notice that the (mild) regularity  given by the CBA condition is necessary for such an uniform estimate: it is easy to produce complete, locally geodesic,  locally geodesically complete metric spaces where (\ref{lowerestimate}) fails to hold,  {\color{red} see Example \ref{exdatrovare} da trovare.}
%\vspace{2mm}

%\noindent The almost convexity radius $\rho_{ac} (X)$ of a geodesic space $X$ is defined as follows (details can be found in Section \ref{sec-preliminaries}). Given a geodesic $[x,y]$ between $x,y\in X$ and $t\in[0,1]$ we denote by $y_t$ the point along the geodesic at distance $td(x,y)$ from $x$. Then, the almost convexity radius of $X$ is the supremum of the radii $r$ such that for any $x\in X$, any two points $y,z\in B(x,r)$ and any two geodesics $[x,y], [x,z]$ it holds 
%$$d(y_t,z_t)\leq 2td(y,z).$$ 

The almost-convexity radius $\rho_{\text{ac}} (x)$ of a geodesic space $X$ at a point $x$ is defined as  the supremum of the radii $r$ such that for any $y,z\in B(x,r)$    and any $t\in[0,1]$ it holds: 
\vspace{-5mm}

$$d(y_t,z_t)\leq 2  t \cdot d(y,z),$$

\noindent where $y_t, z_t$ denote points along  geodesics $[x,y]$ and  $[x,z]$  at distance $td(x,y)$ and $td(x,z)$ respectively from $x$. The almost-convexity radius of $X$ is correspondingly defined as   $\rho_{\text{ac}} (X) = \inf_{x \in X} \rho_{\text{ac}}(x)$.
%\small
%\begin{equation}\label{p-busemann}
%d(x_{\frac12}, y_{\frac12}) \leq M_p ( d(x_0,y_0) , d (x_1, y_1))= \sqrt[p]{ \frac{d(x_0,y_0)^p + d (x_1, y_1)^p}{2}}
%\end{equation}
% \normalsize 
%\noindent When $p=1$, condition  (\ref{p-busemann}) is precisely the nonpositive curvature condition of Busemann.  
It is not difficult to show that  every GCBA-space $X$ always has positive almost-convexity radius at every point: namely if $X$ is locally CAT$(\kappa)$ and $x\in X$ then $\rho_{\text{ac}}(x)$ is always greater than or equal to the CAT$(\kappa)$-{\em radius} $\rho_{\text{cat}} (x)$ (see Section \ref{sec-preliminaries} for  all details and the relation  with the contraction and the logarithmic maps).
However the almost-convexity radius is a more flexible geometric invariant than the CAT$(\kappa)$-radius, much alike the injectivity radius for Riemannian maniolds, 
since a space $X$ might have a large curvature $\kappa$ concentrated in a very small region around $x$, so that it may happen that $\rho_{\text{ac}} (x) $ is much larger than the  CAT$(\kappa)$-radius  at $x$.

%\gg \rho_{\text{cat}} (x)$.  {\color{red} non sappiamo se è vera questa, magari mettiamo $>$ invece di $\gg$. Oppure la scriviamo puntualmente invece che globalmente.} 
We stress the fact  that no explicit upper bound on the curvature is assumed for the estimate (\ref{lowerestimate}); the condition GCBA is only needed to ensure sufficient regularity of the space (and the existence of a natural measure to compute volumes).
\vspace{1mm}

% \enlargethispage{6mm}

For all subsequent results we will consider, as standing assumption, GCBA-spaces with a uniform upper bound on the packing 
%or doubling 
constant at some   fixed  scale $r_0$  smaller than the almost-convexity radius, or a doubling condition up to an arbitrary small scale. These classes of metric spaces are large enough to contain many interesting examples besides Riemannian manifolds and small enough to be, as we will see,  {\em compact}  in the Gromov-Hausdorff sense. 

\noindent Notice that for Riemannian manifolds a local doubling  or a packing condition at some scale $r_0>0$ are much weaker assumptions than a lower bound of the Ricci curvature (see \cite{BCGS}, Sec.3.3, for different examples and a comparison of Ricci, packing and doubling conditions).
However there are a lot of non-manifolds examples in these classes of metric spaces. The simplest ones are simplicial complexes with locally constant curvature (also called {\em $M^\kappa$-complexes}, cp. \cite{BH09}) and "bounded geometry" in an appropriate sense: they will be 
%presented and 
studied in detail in Section \ref{sec-simplicial}.
\noindent Other interesting classes of spaces satisfying a uniform packing condition at fixed scale are the class of  Gromov-hyperbolic spaces with {\em bounded entropy}, admitting a cocompact group of isometries (as shown in  \cite{BCGS}, \cite{BCGS2}),  or the class of (universal coverings of)  compact, non-positively curved manifolds with bounded entropy admitting acylindrical splittings   (see   \cite{CerS20}). See also   \cite{CavS20bis} for applications in the non-cocompact case.

\vspace{2mm}
In Sections \ref{sec-volume} and \ref{sec-packingcovering} we will see how the packing or covering conditions and the upper bound on the curvature interact. While in geodesic metric spaces it is always possible to extend the packing condition at some scale to bigger scales (cp. Lemma \ref{packingbigscales}), it is not possible in general to extend the packing condition uniformly to smaller scales, see Example \ref{example-packing}.  Another key-result  of the paper is that this extension is possible when the metric space has curvature bounded from above and is locally geodesically complete. \linebreak  In particular the local geometry at scales smaller than $r_0$ is controlled by the packing condition:

\begin{theorem}[Extract from Theorem \ref{char_pack}] \label{intro-pack}${}$ \\
	Let $X$ be a complete, geodesic, \textup{GCBA}-space 
	%with curvature  smaller than $\kapp$ and 
	with  
	almost-convexity radius $\rho_{\textup{ac}}(X) \geq \rho_0 > 0$. 
	\begin{itemize}
	\setlength\itemindent{-2mm}
		\item[(a)] there exist $P_0$ and $r_0 \leq \rho_0/3$ 
		%		{\color{red} non basta $\rho_0$?} {\color{blue} no perchè per definizione andiamo a vedere le palle di raggio $3r_0$, ed è importante che la mappa di contrazione sia già 2-Lipschitz a questa scala} 
		such that $X$ satisfies the $P_0$-packing condition at scale $r_0$;	
		\item[(b)] there exist $n_0$ and $V_0, R_0 >0$ such that $X$ has dimension $\leq n_0$ and $\mu_X(B(x,R_0))\leq V_0$ for all $x \in X$;
		\item[(c)] there   exist two functions $c(r), C(r)$ such that for any $x\in X$ and for any $0<r<\rho_0$:
		\vspace{-1mm}
		$$0< c(r) \leq \mu_X(B(x,r)) \leq C(r) < +\infty.$$		
	\end{itemize}
\end{theorem}
%\vspace{-8mm} 

% {\color{red} Non sono sicuro che  \cite{Na02} mostri esattamente la condizione che ho dato sopra per quadrilateri (e che usiamo nella dimostrazione);  per questo ho  aggiunto \cite{Bac}, in cui mi pare si dimostri che la condizione sopra, per mid-points di quadrilateri, \`e equivalente a quella per midpoints di triangoli. \\
% Serve sicuramente ricordare da qualche parte l'equivalenza di queste condizioni, e della condizione con $x_t= t\cdot x_0 + (1-t)x_1$ invece che con $x_{\frac12}$: per questo ho suddiviso la sezione \ref{sec-preliminaries}in sottosezioni (una delle quali da riempire con questa roba)}

\noindent For Riemannian manifolds of dimension $n$ the measure $\mu_X$ coincides with the $n$-dimensional Hausdorff measure, so (b) corresponds simply to a uniform upper bound on the Riemannian volume of balls of some fixed radius $R_0$, a condition that it is sometimes easier  to verify than the bounded packing. \linebreak
The proof of Theorem \ref{char_pack} is essentially based on   universal estimates from below and from above of the volume of small balls of $X$ in terms of dimension and of the packing constants.   
We will prove these estimates in Section \ref{sec-volume}.
We want to point out that, while the estimate (\ref{lowerestimate}) and  Theorem \ref{char_pack} are new, many of the ideas behind these results are already implicitely present  in \cite{LN19}. 

%  \enlargethispage{5mm}

\vspace{2mm}
In  Section \ref{sec-doubling}, we investigate the relation between the local doubling condition\footnote{Beware that the {\em doubling constant} which is used in \cite{LN19} is a different notion, which is purely metric and does not depend on the measure.}  
with respect to the natural measure $\mu_X$ and the local structure of GCBA-spaces. It is easy to show that a local doubling condition implies the packing. 
% The measure $\mu_X$ is said to be $D_0$-doubling up to scale $t_0$ if for any $0<t<t_0$ and for any $x\in X$ it holds $\frac{\mu_X(B(x,2t))}{\mu_X(B(x,t))}\leq D_0$.
However it turns out that the doubling property is much stronger and characterizes GCBA-spaces which are {\em purely dimensional spaces}, i.e. those whose points  have all the same dimension. Indeed we prove:

\begin{theorem}[Extract from Corollary \ref{doubling} \& Theorem \ref{pure-dimensional}] 
	\label{intro-doub}${}$ \\
	Let $X$ be a complete, geodesic, \textup{GCBA}-space
	%with $2$-convexity radius $\rho_{\textup{conv}}(X)  \geq \rho_0>0$. \linebreak
	%	 locally compact, locally CAT$(\kappa)$, locally geodesically complete, length metric space $X$ 
	with almost-convexity radius $\rho_{\textup{ac}}(X)\geq \rho_0 > 0$. 
	The following conditions are equivalent:
	\begin{itemize}
\setlength\itemindent{-4mm}
		\item[(a)] there exists $D_0>0$ such that the natural measure $\mu_X$ is $D_0$-doubling up to some scale $r_0 > 0$;
		\item[(b)] $X$ is purely $n$-dimensional for some $n$  and there exist constants  $P_0$ and $ r_0 \leq  \rho_0/3$ such that $X$  satisfies the $P_0$-packing condition at scale $r_0$.
	\end{itemize}
\end{theorem}

%{\color{red} --  la compattezza locale segue dalla completezza + packing o doubling;
%
%-- anche qui, la specifica  curvatura  $< \kappa$ non \`e necessaria, basta CBA;
%}

%\vspace{2mm}

The families of spaces with uniformly bounded diameter, satisfying a packing condition for some universal function $P=P(r)$ and  all $0<r \leq r_0$, are classically called {\em uniformly compact}; actually one can always extract from them convergent subsequences for the Gromov-Hausdorff distance \linebreak (see \cite{Gr81}).
Moreover it is classical that an upper bound  on the curvature is stable under Gromov-Hausdorff convergence, provided that the corresponding CAT$(\kappa)$-radius is uniformly bounded below. 
%In Section {\color{red} \ref{subsec-buseman,cat}}, we will prove that the same is true also under a lower bound of the  $2$-convexity radius.\linebreak  Namely, we prove that  on any $p$-Busemann space, if  the $\epsilon$-quasigeodesics are {\em locally} close to geodesics (as happens for locally CAT$(\kappa)$-spaces),  then the same is true globally;  in particular, {\em all}  geodesics  on a  limit space $X$ of $p$-Busemann, locally CAT(k)-spaces $X_n$ are always limits  of geodesics in $X_n$  (see  
%{\color{red} Proposition \ref{prop-excess} and Propositions \ref{ultralimitCAT}, \ref{ultralimit-stability})
%\\--dobbiamo decidere in che forma dare questi nuovi enunciati: per esempio  come li ho scritti sopra (eviterei i midpoints negli enunciati,  e invece farei riferimento alle $\epsilon$-quasigeodesics);\\
%--inoltre sono  da risistemare nelle varie sezioni: probabilmente  fuori dall'ap\-pendice per valorizzarli, visto che sono nuovi rispetto a BH
%}
Starting from the  results proved above  it is possible to decline Gromov's Precompactness Theorem for GCBA-spaces  as follows.
Consider the classes \footnote{Mnemonically we write before the semicolon the parameters which are relative to the packing condition or to the condition on the natural measure $\mu_X$} 
\vspace{-3mm}

$$ \text{GCBA}_{\text{pack}}^\kappa (P_0,r_0; \rho_0),  \hspace{5mm}
\text{GCBA}_{\text{vol}}^\kappa (V_0, R_0; \rho_0,  n_0 ) $$

%{\narrowstyle
\noindent   of  complete, geodesic, GCBA-spaces with curvature $\leq \kappa$, almost-convexity radius  $\rho_{\text{ac}}(X) \geq \rho_0 >0$ and satisfying, respectively,  condition (a) or (b) of Theorem \ref{intro-pack}.
%}
\\
%\normalstyle
%$GCBA^\kappa_\mu (\rho_0, n_0, D_0}$
%the corresponding subsets made of those space which have diameter bounded by $D_0$  
Let also denote by  
\vspace{-4mm}

$$\text{GCBA}^\kappa_{\text{vol}} (V_0; \rho_0,   n_0^{=} )$$

\noindent  the   class of complete,  geodesic, GCBA-spaces with curvature $\leq \kappa$,    total measure $\mu_X (X) \leq V_0$,   almost-convexity radius  $\rho_{\text{ac}}(X) \! \geq \! \rho_0 \!>\! 0$ and dimension \linebreak 
{\em precisely equal to}  $n_0$.
%subset of $GCBA_\mu^\kappa (V_0, R_0 ; \rho_0, n_0)$   
Then:
\begin{theorem}[Theorem \ref{compactness-packing}, Corollary \ref{compactness-unpointed} \& \ref{compactness-pointed}]  
	\label{intro-compactness} ${}$
	
\begin{itemize}
\setlength\itemindent{-4mm}
\item[(a)] The classes	$\textup{GCBA}_{\textup{pack}}^\kappa (P_0,r_0; \rho_0)$ and $\textup{GCBA}_{\textup{vol}}^\kappa (V_0, r_0 ; \rho_0, n_0)$ are compact with respect to the pointed Gromov-Hausdorff convergence;  
\item[(b)]  the class
	% $GCBA^\kappa_\mu (\rho_0, n_0, D_0)$ and 
	$\textup{GCBA}^\kappa_{\textup{vol}} (V_0; \rho_0, n_0^{=})$ is compact with respect to the Gromov-Hausdorff convergence and contains only finitely many homotopy types.
\end{itemize}
\end{theorem} 

%\noindent {\color{blue}\st{We will also see that a uniform packing  at some scale $r_0$ is also a } {\em necessary} \st{condition for compactness (see Theorem} \ref{precompactness}\st{ for the precise statement)}.}

\vspace{1mm}
As our spaces are locally CAT$(\kappa)$  with CAT$(\kappa)$-radius uniformly bounded below (see  inequality (\ref{ac-radius,cat-radius}) in Sec. \ref{sec-preliminaries}),  it is not surprising that the limit space is again locally CAT$(\kappa)$. 
Less trivially,   as a part of the proof of  the compactness, we  need to show that   the conditions on the measure, on the almost-convexity radius and on the dimension 
%$\rho_{\text{ac}}(X)  \geq  \rho_0$    and   $\mu_X (B(x,r_0)) \leq V_0$  
are stable under Gromov-Hausdorff  limits.\linebreak
So let us highlight the following results which are  consequence  of the  estimates in Theorems A and B and are part  of the compactness theorem. They will be proved in Section \ref{sec-compactness}:

\begin{theorem}[Proposition \ref{prop-semicontinuity} \& Proposition \ref{cor-dim}]${}$\\
	%\hspace{2mm} 
	Let $(X_n,x_n)$ be   \textup{GCBA}$^\kappa$-spaces   converging    to $(X,x)$ with respect to the  pointed Gromov-Hausdorff topology. Then:
	
	\begin{itemize}
 
 	\item[(a)] $\rho_{\textup{ac}} (X) \geq \limsup_{n \rightarrow \infty} \rho_{\textup{ac}} (X_n)$;
		\item[(b)]  if  $\rho_{\textup{ac}} (X_n) \geq \rho_0>0$ for all $n$ then  $\textup{dim}(X) \leq \lim_{n\to +\infty} \textup{dim}(X_n)  $ 
		and the equality 
		%$\textup{dim}(X) = \lim_{n\to +\infty} \textup{dim}(X_n) $ 
		holds if and only if the distance from $x_n$ to the maximal dimensional subspace $X_n^{\textup{max}}$ of $X_n$  stays uniformly bounded when \linebreak 
		$n\rightarrow \infty$.
	\end{itemize}
	%{\narrowstyle
	\vspace{-2mm}
	\noindent (The second assertion refines  Lemma 2.1 of \cite{Nag18}, holding for \textup{CAT}$(\kappa)$-spaces).
	%}
	
\end{theorem}

\noindent Therefore $\textup{GCBA}$ spaces with curvature uniformly   bounded from above and almost convexity radius uniformly bounded below can collapse only if the maximal dimensional subspaces go  to infinity.  We will see such an  example in Section \ref{sec-compactness}.\pagebreak
\vspace{2mm}

On the other hand the lower-semicontinuity of the natural measure of balls and of the total volume will follow from \cite{LN19}, where it is proved that if $(X_n)_{n\geq 0}$ is a sequence of GCBA-spaces converging to $X$ then the natural measures $\mu_{X_n}$ converge weakly to the natural measure $\mu_X$ (see Lemma \ref{volume-convergence} and the proof of Corollary \ref{volume-convergence-uniform} for details). 
We will   see  in Section \ref{sec-doubling} that,  under the stronger assumptions that the natural measure is doubling up to some arbitrarily  small scale, the volume of balls is actually {\em continuous} (cp.  Corollary \ref{volume-convergence-uniform}).

Once proved that the bound on the total volume  is stable under Gromov-Hausdorff convergence and that this implies the uniform boundedness of the spaces in our class,   
%{\color{red} Questo l'avevo scritto x situare il nostro lavoro nell'ambito di teoremi classici, quando avevamo l'ipotesi $\rho_{cat}(X) \geq \rho_0$;   bisogna pensarci un attimo per capire se vale ancora $LGC(r)=r$ per $r\leq \rho_0$, mi pare di s\`i}.\linebreak
the homotopy finiteness stated in (ii) is a particular case of Petersen's finiteness theorem  \cite{Pet90}; actually,  as the CAT$(\kappa)$-radius is uniformly bounded below,   these spaces have a common local geometric contractibility function LGC$(r)=r$ for $r\leq \rho_0$. 

%\noindent These are precisely the points where the lower estimate (\ref{lowerestimate}) is crucial and the control of the packing function at smaller scales (Theorem \ref{intro-pack}) will be used. 
\vspace{2mm}

% \enlargethispage{5mm}
It is not difficult (see Section \ref{sec-compactness}) to check that also the doubling property is stable under pointed Gromov-Hausdorff convergence and so is the property of being pure dimensional. 
Namely let us also consider the classes (with the same conventions as before)
\vspace{-3mm}

$$ \text{GCBA}_{\text{doub}}^\kappa(D_0   , r_0   ;   \rho_0) \hspace{5mm}   
\text{GCBA}_{\text{vol}}^\kappa(V_0   ;   \rho_0  ,  n_0^{^{\text{pure}}})  $$
of  complete, geodesic, GCBA-spaces $X$ with curvature $\leq \kappa$,   almost-convexity radius    $\rho_{\text{ac}}(X) \geq \rho_0 >0$ and which are, respectively, either  $D_0$-doubling  up to scale $r_0$ or purely $n_0$-dimensional with total measure  $\mu_X(X) \leq V_0$. \linebreak
We then deduce the following additional compactness results:
%{\color{red} non starei a scrivere {\em tutte} le possibili variazioni nell'intro, per questo  ne ho messo solo due, una per spazi non compatti e una per spazi compatti}

\begin{theorem}[Extract from Corollaries \ref{compactness-unpointed} \& \ref{compactness-pointed}] ${}$
	\label{intro-compactness2} 
	
	\noindent   The classes  $\textup{GCBA}_{\textup{doub}}^\kappa(D_0,r_0;\rho_0) $,   $ \textup{GCBA}_{\textup{vol}}^\kappa (V_0; \rho_0, n_0^{^{\textup{pure}}})$ 
	%{\narrowstyle 
	are  compact with respect  to  pointed and unpointed Gromov-Hausdorff convergence  respectively. \linebreak
	%}
	Moreover  $ \textup{GCBA}_{\textup{vol}}^\kappa (V_0; \rho_0, n_0^{^{\textup{pure}}})$ contains only finitely many homotopy types.
	
\end{theorem}

\enlargethispage{5mm}
\noindent  The proof of these and  other compactness and stability results is presented in Section \ref{sec-compactness}.

\vspace{2mm}
Finally in Section \ref{sec-simplicial} we specialize our results to study the convergence and stability of   $M^\kappa$-complexes with bounded geometry. 
We will first establish some basic relations relating the injectivity radius to the size and valency of the complexes. 
Recall that the {\em valency} of a $M^\kappa$-complex $X$ is the maximum number of simplices having a same vertex in common and the {\em size} of the simplices of $X$ is defined as the smallest radius $R > 0$ such that any simplex contains a ball of radius $\frac{1}{R}$ and is contained in a ball of radius $R$;
we refer to  Sec. \ref{subsec-basics} for further definitions and details. Then we prove:

\begin{theorem}[Proposition \ref{complex-curvature} , Sec.\ref{sec-simplicial}] ${}$\\
	Let $X$ be a $M^\kappa$-complex	 whose simplices have size bounded by $R$, with valency at most $N$ and no free faces. Then the following conditions are equivalent:
	\begin{itemize}\setlength\itemindent{-4mm}
		\item[(a)] $X$ is a complete $\textup{GCBA}$-space with curvature $\leq \kappa$;
		\item[(b)] $X$ satisfies the link condition  at all vertices;
		\item[(c)] $X$ is locally uniquely geodesic;
		\item[(d)] $X$ has positive injectivity radius;
		\item[(e)] $X$ has injectivity radius 
		%$\rho_{\textup{inj}}(X) \geq \iota_0$
		$\geq \iota_0$, for some  $\iota_0$ depending only on $R$ and $N$.
	\end{itemize}
\end{theorem}

%{\color{red} 
%\noindent 
%-- l'enunciato (e) \`e da verificare, no? intanto l'ho dato per buono.\\
%-- ma la completezza viene dalla condizione sulla size?\\ 
%%-- la condizione sulla size (al posto delle shapes) la capisco per a,b,c,d; ma $N_0$ \`e necessario per le prime 4?
%}
%{\color{blue} la completezza segue sicuramente dalla size + valenza limitata: ogni successione di Cauchy sta dentro una palla di raggio piccolo a piacere e tale palla interseca un numero finito di simplessi. A questo punto è facile avere la convergenza della successione. Se togli la size è falso: realizza il segmento $[0,1)$ come concatenazione di segmentini sempre più corti. Il risultato non è completo anche se la valenza è limitata. Credo in realtà che la completezza segua semplicemente dalla size, anche con valenza infinita. Infatti una successione di Cauchy o sta tutta interna a una faccia e abbiamo finito, oppure oscilla intorno al bordo di un simplesso. Il fatto che la size sia limitata garantisce che definitivamente la successione oscilli sempre intorno alla stessa faccia o allo stesso vertice perchè c'è una distanza prefissata tra i vertici di uno stesso simplesso. Quindi dovrebbe essere facile trovare il limite su quella faccia o in quel vertice.}

\noindent The   equivalence  of the first four conditions is well-known for $M^\kappa$-complexes with {\em finite shape} (that is whose geometric simplices, up to isometry,  vary in a finite set), see \cite{BH09}. The last condition is new and we will use it to exhibit other examples of compact families of GCBA-spaces. Namely, let  
$$M^\kappa   (R_0,  N_0),   \hspace{5mm}   M^\kappa  (R_0; V_0, n_0) $$
be the class of $M^\kappa$-complexes $X$ without free faces,  with  {\em positive} injectivity radius (but nor a-priori uniformly bounded below), simplices of size bounded by $R_0$  and, respectively, valency bounded by $N_0$ or total volume bounded by $V_0$ and dim$(X) \leq n_0$.  
It is immediate to check that, for suitable \linebreak $N_0=N_0(R_0,V_0, n_0)$, the class $M^\kappa (R_0; V_0, n_0)$ is a subclass of $M^\kappa   (R_0,  N_0)$, made of {\em compact}  $M^\kappa$-complexes, namely with a uniformly bounded number of simplices (cp. proof of Theorem \ref{finiteness-complex}); hence it contains only  finitely many $M^\kappa$-complexes  {\em up to   simplicial homeomorphism}.
On the other hand we prove: 

% {\color{red}  ho cercato  di usare delle notazioni che non spiazzino il lettore rispetto alle precedenti (per quanto possibile): \\
% --  usiamo $M^k$, con $\kappa$ in alto come prima;\\
% --  inoltre i parametri $N_0$  e $R_0$   ci danno il packing, quindi li ho messi come per le altre famiglie prima del punto e virgola;  \\
% -- nelle dimostrazioni magari  la dimensione chiamiamola  $n_0$.  }

\begin{theorem}[Extract from Theorem \ref{simplicial-compactness} \& Corollary \ref{finiteness-complex}, Sec.\ref{sec-simplicial}] \label{intro-complexes}
	${}$\\
	The classes $ M^\kappa  (R_0, N_0)$ and $M^\kappa (R_0; V_0, n_0) $ are  compact, respectively,  under pointed and unpointed Gromov-Hausdorff convergence. 
	Moreover there are only finitely many   $M^\kappa$-complexes of diameter $\leq \Delta$ in $M^\kappa (R_0, N_0)$, up to simplicial homeomorphisms.
\end{theorem}
\noindent All the assumptions in this result are necessary. Indeed we will see how dropping the bounds on the valency or on the size of the simplices we do not have neither finiteness nor compactness (see Example \ref{complex-examples}).

\noindent We think that Theorems \ref{intro-compactness}, \ref{intro-compactness2} and  \ref{intro-complexes}  mark quite well the  advantage of the synthetic condition of curvature $\leq \kappa$ over       sectional curvature bounds by  identifying   classes which are closed under Gromov-Hausdorff convergence, in contrast with the  the classical convergence theorems of Riemannian geometry.

\vspace{1mm}

%Once proved the importance of this class of metric spaces our next step is to find a criteria to establish when a locally compact, locally CAT$(\kappa)$, locally geodesically complete, length metric space $X$ satisfies a packing condition.

The Appendix is devoted to recall, for the reader's convenience,  some basics of ultrafilters and ultraconvergence of metric spaces, which is a tool heavily used all along the paper. 
\vspace{1mm}

\small
\noindent {\sc Acknowledgments.} {\em We would like to thank S. Gallot, G. Besson and A. Lytchak for the useful and stimulating discussions during the preparation of this work.}
\normalsize

\enlargethispage{5mm}

\vspace{3mm}
\section{Preliminaries on GCBA-spaces}\label{sec-preliminaries}
First of all we fix the notation.  The open and the closed ball of radius $R$ centered at $x$ in a metric space $X$ will be denoted by $B_X(x,R)$ and $\overline{B}_X(x,R)$ respectively; if the metric space is  clear from the context we will simply write $B(x,R)$ and $\overline{B}(x,R)$. The closed annulus with center at $x$ and radii $r_1 < r_2$ will be denoted by $A(x,r_1,r_2)$. If $(X,d)$ is a metric space and $\lambda$ is a positive real number we denote by $\lambda X$ the metric space $(X,\lambda d)$, where $(\lambda d) (x,y) = \lambda d(x,y)$ for any $x,y\in X$, i.e. the rescaled metric space. \linebreak 
We denote with $B_{\lambda X}(x,r)$ the ball of center $x$ and radius $r$ with respect to the metric $\lambda d$.
The identity map from $(X,d)$ to $(X,\lambda d)$ is denoted by $\text{dil}_\lambda$. \\
A geodesic is a curve $\gamma\colon I \to X$, where $I$ is an interval of $\mathbb{R}$, such that for any $t \leq s\in I$ it holds $d(\gamma(t),\gamma(s)) = \vert t - s \vert$. If $I = [a,b]$ we say that $\gamma$ is a geodesic joining $x=\gamma(a)$ to $y=\gamma(b)$. A generic geodesic joining two points $x,y\in X$ will be denoted by $[x,y]$, even if there are more geodesics joining $x$ and $y$. A curve is a local geodesic if it is a geodesic around any point in its interval of definition.

\noindent  Finally we stress the fact that we consider  pointed Gromov-Hausdorff convergence  \emph{only for complete metric spaces}: so every time we write  $(X_n, x_n) \rightarrow (X, x)$ in the pointed Gromov-Hausdorff sense we mean that $X_n$ and $X$ are complete. This condition is not restrictive; indeed if $(X_n, x_n)$ converges to $(X, x)$ then it converges also to the completion $(\hat{X}, \hat{x})$. As a consequence if $(X_n, x_n)$ is a sequence of proper metric spaces converging to $(X,x)$ then $X$ is proper (see Corollary 3.10 of \cite{Her16}).

\subsection{CAT$(\kappa)$  and GCBA-spaces}\label{subsec-buseman,cat}
We recall the definition of locally CAT$(\kappa)$ metric space.
We fix $\kappa\in \mathbb{R}$. We denote by $M^\kappa_2$ the unique simply connected, complete, $2$-dimensional Riemannian manifold of constant sectional curvature equal to $\kappa$ and by $D_\kappa$ the diameter of $M_2^\kappa$. So $D_\kappa = +\infty$ if $\kappa\leq 0$ and $D_\kappa = \frac{\pi}{\sqrt{\kappa}}$ if $\kappa > 0$.  

A metric space $X$ is CAT$(\kappa)$ if any two points at distance less than $D_\kappa$ can be connected by a geodesic and if the geodesic triangles with perimeter less than $2D_\kappa$ are thinner than their comparison triangles in the model space $M^\kappa_2$. \linebreak
This means the following. For any three points $x,y,z \in X$ such that \linebreak $d(x,y) + d(y,z) + d(z,x) < 2D_\kappa$ a geodesic triangle with vertices $x,y,z$ is the choice of three geodesics $[x,y]$, $[y,z]$ and $[x,z]$,  denoted by $\Delta(x,y,z)$. For any such triangle there exists a unique triangle $\overline{\Delta}^\kappa(\bar{x},\bar{y},\bar{z})$ in $M^\kappa_2$, up to isometry, with vertices $\bar{x}$, $\bar{y}$ and $\bar{z}$ satisfying $d(\bar{x},\bar{y})=d(x,y)$, $d(\bar{y},\bar{z})= d(y,z)$ and $d(\bar{x},\bar{z})=d(x,z)$; such a triangle is called the {\em $\kappa$-comparison triangle} of $\Delta(x,y,z)$. The comparison point of $p \in [x,y]$ is the point $\bar{p} \in [\bar{x},\bar{y}]$ such that $d(x,p)=d(\bar{x},\bar{p})$. The triangle $\Delta(x,y,z)$ is thinner than $\overline{\Delta}^\kappa(\bar{x},\bar{y},\bar{z})$ if for any couple of points $p\in [x,y]$ and $q\in [x,z]$ we have $d(p,q)\leq d(\bar{p},\bar{q})$.

A metric space $X$ is called {\em locally} CAT$(\kappa)$ if for any $x\in X$ there exists $r > 0$ such that $B(x,r)$ is a CAT$(\kappa)$ metric space. The supremum among the radii $r < \frac{D_\kappa}{2} $ satisfying this property is called the CAT$(\kappa)$-radius at $x$ and it is denoted by $\rho_{\text{cat}}(x)$. The infimum of $\rho_{\text{cat}}(x)$ among the points $x\in X$ is called the CAT$(\kappa)$-radius of $X$ and it is denoted by $\rho_{\text{cat}}(X)$; therefore,  by definition, $\rho_{\text{cat}}(X) \leq \frac{D_\kappa}{2}$. 
%Notice that the definition of  CAT$(\kappa)$-radius depends on the chosen bound $\kappa$, but it is omitted in the notation.

A metric space $X$ is {\em GCBA if there exists a $\kappa$ such that $X$ is locally CAT$(\kappa)$, locally compact, separable and locally geodesically complete. }\linebreak The last property means that any local geodesic in $X$ defined on an interval $[a,b]$ can be extended, as a local geodesic, to a bigger interval $[a- \varepsilon, b + \varepsilon]$. In some case we will write GCBA$^\kappa$, if we want to emphasize the role of $\kappa$. \linebreak This class of metric spaces is the one studied in \cite{LN19}. A metric space is geodesically complete if any local geodesic can be extended, as a local geodesic, to the whole $\mathbb{R}$.
We recall a well known fact: any complete, locally geodesically complete metric space is geodesically complete.

A {\em tiny ball}, according to \cite{LN19},   is a metric ball $B(x,r)$ such that \linebreak $r<\min \lbrace 1, \frac{D_\kappa}{100} \rbrace$ and $\overline{B}(x,10r)$ is compact. 
%In \cite{LN19} are studied the properties of tiny balls in GCBA metric spaces.

\subsection{Contraction maps and almost-convexity radius }
We suppose $X$ is a complete, locally geodesically complete, locally CAT$(\kappa)$, geodesic metric space. 
If $x,y \in X$ satisfy $d(x,y) < \rho_\text{cat}(x)$ then there exists a unique geodesic joining them. Hence for any $x\in X$ and $0<r\leq R < \rho_\text{cat}(x)$ it is well defined the {\em contraction map}:
\vspace{-3mm}

$$\varphi^R_r\colon \overline{B}(x,R) \to \overline{B}(x,r)$$
by sending a point $y \in \overline{B}(x,R)$ to the unique point $y'$ along the geodesic $[x,y]$ satisfying  $ d(x,y')/r =  d(x,y) /R$.
Moreover any local geodesic starting at $x$ which is contained in $B(x, \rho_\text{cat}(x))$ is a geodesic. This fact, together with the locally geodesically completeness and the completeness of $X$, shows that the map $\varphi^R_r$ is surjective. It is also $\frac{2r}{R}$-Lipschitz as stated in \cite{LN19}. \linebreak We skecth here the computation.
\begin{lemma}\label{lemma-lip}
	Any contraction map is $\frac{2r}{R}$-Lipschitz.
\end{lemma}
\begin{proof}
	By the CAT$(\kappa)$ condition it is enough to prove the thesis on the model space $M^\kappa_2$. The result is clearly true when $\kappa \leq 0$, so we can assume $\kappa = 1$. In this case $M^\kappa_2$ is the standard sphere $\mathbb{S}^2$. 
	\vspace{1mm}
	
	\noindent 	 {\em Step 1.} For any $x\in \mathbb{S}^2$ and for any $0\leq R \leq \frac{\pi}{2}$ the inverse of the exponential map, the logarithmic map $\log_x \colon B(x,R) \to B_{T_x\mathbb{S}^2}(O,R)$, is $\frac{R}{\sin R}$-Lipschitz. 
	So for any $R$ in our range we have that the logarithmic map is $2$-Lipschitz. Thus we can conclude that, for any $y,z\in B(x,\frac{\pi}{2})$,	
	$$d(y,z) \leq d(\log_x(y), \log_x(z)) \leq 2d(y,z)$$	
	\noindent  where the first inequality follows by standard comparison results.
	\vspace{1mm}
	
	\noindent 	 {\em Step 2.} We fix $0<r \leq R \leq \frac{\pi}{2}$ and $y,z\in B(x,R)$. Let   $y'$ and $z'$ be the contractions of $y$ and $z$. We observe that the contraction of $\log_x(y)$, on the tangent space,  from the radius $R$ to $r$   coincides with the point $\log_x(y')$ and the same holds for $z$; this contraction map is a dilation of factor $\frac{r}{R}$. Therefore
	\vspace{-2mm}	
	$$d(y',z')\leq d(\log_x(y'), \log_x(z')) = \frac{r}{R}d(\log_x(y), \log_x(z)) \leq \frac{2r}{R}d(y,z).$$
	
	\vspace{-6mm}		
\end{proof}

\vspace{1mm}
The natural set of scales where the contraction map is defined is not bounded from above by the CAT$(\kappa)$-radius but rather from the almost-convexity radius.
The \emph{almost-convexity radius at a point $x\in X$} is defined as the supremum of the radii $r$ such that for any two geodesics $[x,y], [x,z]$ of length at most $r$ and any $t\in [0,1]$ it holds: 
$$d(y_t,z_t) \leq 2td(y,z),$$ 
where $y_t, z_t$ are respectively the  points along $[x,y]$ and $[x,z]$ satisfying \linebreak $d(x,y_t)=td(x,y)$ and $d(x,z_t)=td(x,z)$. The almost-convexity radius at $x$   does not depend on $\kappa$ and is denoted by $\rho_{\text{ac}}(x)$. Then, by definition,  for any point $y\in B(x,\rho_\text{ac}(x))$ there exists a unique geodesic joining $x$ to $y$ (the existence follows from  the assumptions on $X$), so the contraction map is well defined for any $0<r\leq R < \rho_\text{ac}(x)$. A straightforward modification of Corollary 8.2.3 of \cite{Pap04} shows that any local geodesic joining $x$ to a point $y$ at distance $d(x,y)<\rho_\text{ac}(x)$  is actually a geodesic. This fact and the geodesic completeness of $X$ imply again that any contraction map within the almost-convexity radius is surjective and  $\frac{2r}{R}$-Lipschitz, by definition.

\enlargethispage{5mm}

\noindent The {\em (global) almost-convexity radius} of the space $X$, denoted by $\rho_\text{ac}(X)$,  is correspondingly defined as the infimum over $x$ of the almost-convexity radius at $x$.
Clearly we always have $\rho_\text{ac}(X) \geq \rho_\text{cat}(X)$.  
The inequality can be partially reversed when $X$ is proper: indeed in this case it holds
\vspace{-3mm}	

\begin{equation}
	\label{ac-radius,cat-radius}
	\rho_{\text{cat}}(X) \geq \min \left\lbrace \frac{D_\kappa}{2}, \rho_\text{ac}(X) \right\rbrace,
\end{equation}
\vspace{-3mm}	

\noindent therefore a lower bound on the almost-convexity radius and the knowledge of the upper bound $\kappa$ yield  a  lower bound on the CAT$(\kappa)$-radius. The proof of \eqref{ac-radius,cat-radius} follows directly from Corollary II.4.12 of \cite{BH09} once observed that any two points of $X$ at distance less than $\rho_\text{ac}(X)$ are joined by a unique geodesic.

%{\color{red} We give here an example where the almost-convexity radius at a point is much larger than the corresponding CAT$(\kappa)$-radius:
%}
%{\color{blue} 
%\begin{ex}
%Consider the space $Y = \mathbb{R}^2 \setminus B$, where $B$ is an open disk of radius $1$.
%The space $Y$ is locally CAT$(0)$ and that the boundary of the disk is locally convex in $Y$.
%Let $X=Y \cup_{\partial B}  \mathbb{S}^{2,+}$ be the metric space obtained by glueing to $Y$, along the boundary of the disk,  an emisphere   of  a $2$-dimensional, round sphere of radius $1$.  
% It is easy to show that $X$  is a geodesically complete, proper, CAT(1) metric space.
%We denote by $N$ the pole of the emisphere $\mathbb{S}^{2,+}$:  by definition $\rho_{\text{cat}}(N) = \frac{\pi}{2}$ while $\rho_{\text{ac}}(N) = + \infty$. 
%\end{ex}
%}

\subsection{Tangent cone and the logarithmic map}\label{subsec-log}
We fix a complete, geodesic, GCBA-space $X$. \\
Given two local geodesics $\gamma$, $\eta$ starting at the same point $x \in X$ we can consider the geodesic triangle $\Delta(x, \gamma(t),\eta(t))$ for any small enough $t>0$. \linebreak
The comparison triangle $\overline{\Delta}^\kappa(\bar{x},\overline{\gamma(t)}, \overline{\eta(t)})$ has an angle $\alpha_t$ at $\bar{x}$. By the CAT$(\kappa)$ condition the  angle  $\alpha_t$ is decreasing when $t \rightarrow 0$, see \cite{BH09}.  Hence it is possible to define the angle between $\gamma$ and $\eta$ at $x$ as $\lim_{t\to 0} \alpha_t  $:   it is denoted by $\angle_x(\gamma,\eta)$ and it takes values in $[0,\pi]$.\\
For any $x\in X$  the {\em space of directions of $X$ at $x$}  is defined as
\vspace{-3mm}

$$\Sigma_x X = \lbrace \gamma \text{ local geodesic s.t. } \gamma(0)=x \rbrace/_\sim$$
where $\sim$ is the equivalence relation $\gamma \sim \eta$ if and only if $\angle_x(\gamma, \eta)=0$.  \linebreak
The function $\angle_x(\cdot,\cdot)$ defines a distance which makes of   $\Sigma_x X$  a compact, geodesically complete, CAT(1) metric space with diameter $\pi$ (see \cite{LN19}).
The {\em tangent cone of $X$ at the point $x$} is the metric space
$$T_x X = \Sigma_x X \times [0,+\infty)$$
\noindent up to the equivalence relation $(v,0)\sim (w,0)$ for every $v,w \in \Sigma_x X$. \linebreak
The point corrisponding to $t=0$ is called the vertex of the tangent cone, denoted by $O$. The metric on $T_x X$ is given by the following formula: given two points $V=(v,t)$ and $W=(w,s)$ of $T_x X$ we define $d_T(V,W)$ as the unique positive real number satisfying:
\begin{equation}
	\label{euclidean-cone-formula}
	d_T(V,W)^2 = t^2 + s^2 - 2ts\cos(\angle_x(v,w)).
\end{equation}
In other words $T_x X$ is the euclidean cone over $\Sigma_x X$.
With this metric $T_x X$ is a proper, geodesically complete, CAT(0) metric space (\cite{LN19}). 
\begin{obs}
	\label{cono_sfera}
	Let $Y = \mathbb{S}^{n-1}$ be the euclidean standard sphere of radius $1$. Then the euclidean cone over $Y$ is isometric to $\mathbb{R}^n$.
\end{obs}

\noindent For any point $x\in X$   the {\em logarithmic map} at $x$   is defined as:
$$\log_x \colon B(x,\rho_\text{ac}(x)) \to T_x X, \qquad y \mapsto ([x,y], d(x,y)),$$
where $[x,y]$ is the unique geodesic from $x$ to $y$ (uniqueness is due to the definition of almost-convexity radius).
\vspace{2mm}

The logarithmic map can be recovered by the contraction maps as follows.\linebreak
First notice that if $X$ is a GCBA-space and $\lambda > 0$ then the space $\lambda X$ is GCBA.  
Now, let the logarithmic map on the space $\lambda X$  at $\text{dil}_\lambda(x)$ be denoted by  
$$\log_{\text{dil}_\lambda(x)} \colon B_{\lambda X}(\text{dil}_\lambda(x), \lambda \rho_{\text{ac}}(x)) \to T_{\text{dil}_\lambda(x)} (\lambda X).$$
The spaces $T_{\text{dil}_\lambda(x)} (\lambda X)$ and $T_x X$ are canonically isometric since the respective space of directions  are canonically isometric.
Let $R < \rho_{\text{ac}}(x)$: we consider a sequence of real numbers $r_n \to 0$, we set $\lambda_n= \frac{R}{r_n}$ and  we define the maps 
$$g_n = \log_{\text{dil}_{\lambda_n}(x)} \circ\, \text{dil}_{\lambda_n} \circ \varphi^R_{r_n} \colon \overline{B}_X(x, R) \to T_x X$$
where we are using the natural identification  $T_{\text{dil}_{\lambda_n} (x)}(\lambda_n X) \cong T_xX$. \linebreak
By the CAT$(\kappa)$ condition the map $\log_{\text{dil}_{\lambda_n}(x)}$ is $(1+\varepsilon_n)$-Lipschitz with $\varepsilon_n \to 0$  for  $r_n \to 0$. So, by Lemma \ref{lemma-lip}, the map $g_n$ is $2(1+\varepsilon_n)$-Lipschitz and for any non-principal ultrafilter $\omega$ this sequence defines a ultralimit map $g_\omega$ between the ultralimit spaces (cp. Proposition \ref{limitmap} in the Appendix). \linebreak
Since $T_x X$ is proper we can apply Proposition \ref{ultralimit-proper} and find that the target space of $g_\omega$ is $T_xX$, i.e.
$$g_\omega : \omega \text{-}\lim\overline{B}_X(x,R) \rightarrow T_x X.$$ 
Using the definition of the logarithmic map and the natural identification $T_{\text{dil}_{\lambda_n}(x)} (\lambda_n X) \cong T_xX$ as metric spaces, it is straightforward to check that $g_\omega$, restricted to the standard isometric copy of $\overline{B}_X(x,R)$ in $\omega \text{-}\lim\overline{B}_X(x,R)$ given by Proposition \ref{ultralimit-proper}, coincides with $\log_x$. 
\vspace{1mm}

In general the logarithmic map of a  \textup{GCBA} space  is not injective, due to the possible branching of geodesics. 
We summarize its properties in the following lemma:
\begin{lemma}
	\label{logarithmic-map-surjective}
	Let $x \in X$ be a point of a complete, geodesic, \textup{GCBA}   space. Then the logarithmic map $\log_x$ has the following properties:	
	\begin{itemize}
		\item[(a)]  $\log_x(\overline{B}(x,r)) = \overline{B}(O,r)$ for any $r<\rho_\textup{ac}(x)$;
		\item[(b)]  $d(O, \log_x(y))= d(x, y)$ for any $y\in B(x,\rho_\textup{ac}(x))$;
		\item[(c)] it is $2$-Lipschitz on $B(x,\rho_{\text{ac}}(x))$. 
	\end{itemize} 
\end{lemma}
\begin{proof}
	Let $y \in B(x,\rho_\text{ac}(x))$. By definition  we have $\log_x(y)=([x,y],d(x,y))$, where $[x,y]$ is the unique geodesic from $x$ to $y$. From \eqref{euclidean-cone-formula} we  immediately infer that $d_T(\log_x(y),O)= d(y,x)$. This proves (b) and that $\log_x(\overline{B}(x,r))$ is included in $\overline{B}(O,r)$ for any $r<\rho_\text{ac}(x)$. 
	Now let $V=(v,t) \in \overline{B}(O,r)$, for $r<\rho_\text{ac}(x)$. We take a geodesic $\gamma$ in the class of $v$. Since $X$ is locally geodesically complete, there exists an extension of $\gamma$ as a geodesic to the interval $[0,r]$ (this follows from the completeness of $X$ and the fact that any local geodesic is a geodesic if it is contained in a ball of radius smaller than the almost-convexity radius).	
	Then, using the definition of the logarithmic map, we deduce that  $\log_x (\gamma (r)) = V$. Now $d(x, \gamma(r)) = r$, which concludes the proof of (a). 
	Finally  we have seen that the logarithmic map is obtained as the restriction of the limit map $g_\omega: \omega\text{-}\lim \overline{B}_X(x,R) \rightarrow T_x X$ to $\overline{B}_X(x,R)$. It is $2$-Lipschitz for all $R \leq \rho_{\text{ac}}$, therefore it is $2$-Lipschitz on $B(x,\rho_{\text{ac}}(x))$. 
\end{proof}

The logarithmic map gives a good local approximation of $X$ by the tangent cone, as expressed in the following result.
\begin{lemma}[\cite{LN19}, Lemma 5.5]
	\label{logarithm-approximation}
	Let $x \in X$ be a point of a complete, geodesic, \textup{GCBA}   space. For any $\varepsilon > 0$ there exists $\delta > 0$ such that for all $r<\delta$ and for every $y_1,y_2 \in B(x,r)$ it holds
	\vspace{-3mm}
	
	$$\vert d(y_1,y_2) - d_T(\log_x(y_1), \log_x(y_2))\vert \leq \varepsilon r.$$
\end{lemma}

As a consequence of this fact, Lytchak and Nagano proved that the tangent cone at $x$ can be seen as the Gromov-Hausdorff limit of a rescaled tiny ball around $x$. We explicit the proof of this fact because in the following we will need to write who are the maps realizing the Gromov-Hausdorff approximations.

\begin{lemma}[\cite{LN19}, Corollary 5.7]
	\label{lemma-tangentcone=limit}
	Let $x \in X$ be a point of a complete, geodesic, \textup{GCBA}   space. For any sequence $\lambda_n \to \infty$ consider the sequence of  \textup{CAT}$(\kappa)$, pointed spaces $Y_n =  \left( \lambda_n B(x,r),   x \right)$, for  any $r<\rho_\textup{cat}(x)$. Then: 
	\begin{itemize}
		\item[(a)]  $Y_n \rightarrow (T_xX, d_T, O)$  in the pointed Gromov-Hausdorff convergence;
		\item[(b)]   the approximating maps $f_n \colon Y_n \to T_xX$ are given by $f_n = \log_{\textup{dil}_{\lambda_n}(x)}$
		(using again the natural identification  $T_{\textup{dil}_{\lambda_n}(x)}(\lambda_n X) \cong T_x X$)
		
	\end{itemize}  
\end{lemma}

\begin{proof} Fix $R>0$ and   any $\varepsilon > 0$. Let $\delta$ be as in Lemma \ref{logarithm-approximation} and set $r_n=1/ \lambda_n$. We may assume that  $r_n\cdot R < \delta$. Then   for all $y_1,y_2 \in B_{Y_n}(x,R)$ we have $y_1,y_2 \in B_X(x,r_n R)$ and we can apply the  Lemma \ref{logarithm-approximation}, which yields
	\vspace{-3mm}
	
	$$\vert d(y_1,y_2) - d_T(\log_x(y_1), \log_x(y_2))\vert \leq \varepsilon r_n R.$$
	We have $d_{Y_n}(y_1,y_2) = \frac{d(y_1,y_2)}{r_n}$ and,  by \eqref{euclidean-cone-formula} and  by the definition of the logarithmic map, 
	\vspace{-3mm}
	
	$$d_T(f_n(y_1), f_n(y_2)) = \frac{1}{r_n}d_T(\log_x(y_1), \log_x(y_2)).$$
	In conclusion we get
	\vspace{-3mm}
	
	$$\vert d_{Y_n}(y_1,y_2) - d_T(f_n(y_1), f_n(y_2))\vert \leq \varepsilon R.$$
	Since this is true for any $\varepsilon >0$ the thesis follows from   Lemma \ref{logarithmic-map-surjective}.
\end{proof}

\vspace{1mm}

\enlargethispage{5mm}
Finally we observe that this characterization of $T_x X$ has another consequence. 
Fix any $v\in \Sigma_xX$,  which can be naturally seen as an element of $T_xX$, and  take any geodesic $\gamma$ starting at $x$ defining $v$: then for any sequence $r_n \to 0$ we have that the sequence  $\gamma(r_n) \in Y_n$ defines $v$ in the limit \linebreak (indeed, $f_n(\gamma(r_n)) = v$ for any $n$).

\subsection{Dimension and natural measure}\label{subsec-structure}
We recall  some fundamental properties of GCBA-spaces proved in \cite{LN19}.  
For any point $x\in X$ there exists an integer number $k\in \mathbb{N}$ such that any sufficiently small ball around $x$ has Hausdorff dimension $k$. This number is called {\em the dimension of $X$ at the point $x$} and it is denoted by $\text{dim}(x)$. 
It is possible to show that $\text{dim}(x)$ is equal to the geometric dimension of the tangent cone to $X$ at $x$ as defined in \cite{Kl99}. The {\em dimension of $X$} is the (possibly infinite) quantity $\text{dim}(X) = \sup_{x\in X} \text{dim}(x) \in [0,+\infty]$.\\
There exists a natural stratification of $X$ into disjoint subsets $X^k$, where $X^k$ is the set of points of dimension $k$, for $k\in \mathbb{N}$. In other words $X = \bigsqcup_{k\in \mathbb{N}} X^k$. \linebreak 
Moreover the $k$-dimensional Hausdorff measure $\mathcal{H}^k$ is locally positive and locally finite on $X^k$. Hence it is defined a measure on $X$ as
$$\mu_X = \sum_{k\in \mathbb{N}} \mathcal{H}^k \llcorner X^k.$$
The measure $\mu_X$ is locally positive and locally finite: we call it the {\em natural measure of $X$}.
\begin{ex}
	\label{Riemannian}
	If $X$ is a $n$-dimensional Riemannian manifold with sectional curvatures $\leq \kappa$ then $X$ is a locally geodesically complete, locally compact, separable, locally CAT$(\kappa)$ metric space. In this case $\mu_X$ is the $n$-dimensional Hausdorff measure and it coincides with the Riemannian volume measure, up to a multiplicative constant.
\end{ex}
This stratification of $X$ has good local properties, as shown in \cite{LN19}.  \linebreak
%Let $U = B(x,r)$ be a tiny ball in $X$. 
For any   $k\in \mathbb{N}$ it is possible to define the set of {\em regular points} Reg$^k(X)$ of the $k$-dimensional part $X^k$ of $X$. 
We do not present here the definition of regular points (they are those points that are $(k,\delta)$-strained for a suitable small $\delta$, according to \cite{LN19}, Sec. 11.4). Instead we recall the main properties of the set of $k$-dimensional and regular $k$-dimensional points we will need. \linebreak
For every $S \subset X$ we will denote $S^k = S \cap X^k$ and Reg$^k(S)= S^k \cap $Reg$^k(X)$.\\
Then:

\setlength{\leftmargini}{20pt}
\begin{itemize}
	\item the set Reg$^k(X)$ is open in $X$ and dense in $X^k$ (Cor. 11.8 of \cite{LN19});
	\item for any tiny ball $B(x,r)$     there exists   $k$ such that $B(x,r)$   does not contain   points of dimension  $>k$  (Corollary 5.4 of \cite{LN19});
	\item for any tiny ball $B(x,r)$ there exists a constant  $C$, {\em only  depending  on the maximal number of $r$-separated points in $\overline{B}(x,10r)$}, such that:
	\begin{equation}
		\label{k-inequality1}
		\mathcal{H}^k \left( B(x,r)^k \right) \leq C\cdot r^k  
	\end{equation}
	
\vspace{-5mm}
	
	\begin{equation}
		\label{k-inequality2}
		\mathcal{H}^{k-1} \left( \bar B (x,r)  ^k \setminus \text{Reg}^k (B(x,r))  \right)\leq C\cdot r^{k-1}
	\end{equation}
 
\end{itemize}
\noindent 	(Corollary 11.8 of \cite{LN19}; see Sec.\ref{sec-packingcovering} for  the definition of $r$-separated \nolinebreak points).
\vspace{1mm}

 \enlargethispage{6mm}

\subsection{Gromov-Hausdorff convergence}\label{subsec-GH}
 
We recall   here some facts about the behaviour of the natural measures and  the dimension under pointed Gromov-Hausdorff convergence. \\
Consider a proper GCBA-space $X$ and   its natural measure $\mu_X \! = \! \sum_{k=0}^{n} \mathcal{H}^k \llcorner X^k$, where $n = \text{dim}(X)$ is assumed to be finite. The $k$-dimensional Hausdorff measure $\mathcal{H}^k$ restricted to the $k$-dimensional part is a Radon measure  \linebreak
(indeed it is Borel regular and locally finite on the proper metric space $X$), so it is $\mu_X$. In particular for any open subset $U\subset X$ it holds:
$$\mu_X(U) = \sup \lbrace \mu_X(K) \text{ s.t. } K \text{ is a compact subset of } U\rbrace.$$
Now suppose to have a sequence of proper GCBA-spaces $X_n$ converging in the pointed Gromov-Hausdorff sense to some (proper) GCBA-space $X$. \linebreak
Arguing as in the first part of the proof of Theorem 1.5 of \cite{LN19} we deduce that  the natural measures $\mu_{X_n}$ converge in the weak sense to the natural measure of the limit, $\mu_X$. This means that for any compact subsets $K_n\subset X_n$ converging to a compact subset $K\subset X$ it holds:
% \small
\begin{equation}
	\label{convergence-measure}
	\lim_{\varepsilon \to 0} \liminf_{n\to +\infty} \mu_{X_n}(B(K_n,\varepsilon)) = \lim_{\varepsilon \to 0} \limsup_{n\to +\infty} \mu_{X_n}(B(K_n,\varepsilon)) = \mu_X(K)
\end{equation}
\normalsize
where we denote by $B(K_n,\varepsilon)$ the $\varepsilon$-neighbourhood of $K_n$.
As a consequence:
\begin{lemma}
	\label{volume-convergence}
	Let $X_n$ be a sequence of proper, \textup{GCBA}-spaces converging in the pointed Gromov-Hausdorff sense to a proper, \textup{GCBA}-space $X$. 
	Let $x_n \in X_n$ be a sequence of points converging to $x\in X$. Then for any $R > 0$ it holds:
	\begin{equation}
		\label{eqballs}
		\mu_X(B(x,R)) \leq \limsup_{n \to + \infty}\mu_{X_n}(B(x_n,R)).
	\end{equation}
\end{lemma}
\begin{proof}
	The natural measure $\mu_X$ is Radon and any compact subset contained in $B(x,R)$ is contained in $\overline{B}(x,R - 2\eta)$ for some $\eta > 0$, therefore
	$$\mu_X(B(x,R)) = \sup_{\eta > 0} \mu_X(\overline{B}(x,R-2\eta)).$$
	
	\vspace{-1mm}
	\noindent 	On the other hand for any $\eta > 0$ we have by \eqref{convergence-measure}
	\vspace{-5mm}
	
	$$\mu_X(\overline{B}(x,R \!-\!2\eta))\!\leq \limsup_{n \to + \infty}\mu_{X_n}(\overline{B}(x_n,R\!-\!\eta))\leq \limsup_{n \to + \infty}\mu_{X_n}(B(x_n,R)).$$
	\nopagebreak
	
	\vspace{-10mm}
\end{proof}

\noindent The equality in (\ref{eqballs}) would follow from a uniform estimate on the volumes of the annulii of a given thickness. Indeed this is the case when the metric spaces satisfy a uniform doubling condition, as we will see in Section \ref{sec-doubling}.
\vspace{1mm}

We end this preliminary section  recalling some facts  about  the  stability of the dimension under Gromov-Hausdorff convergence.  In \cite{LN19} (Def. 5.12), Lytchak and Nagano introduce the notion of {\em standard setting of convergence}.
This means  considering a sequence of tiny balls $$    B(x_n,r_0) \subset   \overline  B(x_n,10 r_0)$$
in a sequence of  GCBA-spaces $X_n$ satisfying the following assumptions:

\setlength{\leftmargini}{15pt}
\begin{itemize}
	%[leftmargin=*]
	\vspace{1mm}
	\item   the closed balls  $ \overline B(x_n,10r_0)$ have uniformly bounded $\frac{r_0}{2}$-covering number
	(i.e.  $\exists \, C_0$ such that the ball  $ \overline B(x_n,10r_0)$ can be covered by  $C_0$  closed balls of radius $\frac{r_0}{2}$ with centers in  $ \overline B(x_n,10r_0)$  for all $n$ , cp.  Sec.\ref{sec-packingcovering})
		\vspace{1mm}
	\item   the balls $ \overline  B(x_n,10 r_0)$ converge  to a compact ball $\overline B(x,10r_0)$ of   a GCBA-space $X$ in the Gromov-Hausdorff sense;
		\vspace{1mm}
	\item  the closures $ \overline  B(x_n,  r_0)$ converge to the closure $\overline B(x, r_0)$ of a tiny ball    in $X$.
\end{itemize} 
	\vspace{1mm}
We then have: 

\setlength{\leftmargini}{20pt}
\begin{lemma}[Lemma 11.5 \& Lemma 11.7 of \cite{LN19}] ${}$
	\label{regularity-convergence}
	
	\noindent 	Let   $B(x_n,r_0)$ be   a sequence of tiny balls   in the standard setting of convergence. 	Let $y_n \in  B(x_n,r_0)$ be a sequence converging to $y\in B(x,r_0)$. Then:

	\begin{itemize}
		\item[(a)] \textup{dim}$(y) \geq \limsup_{n \to + \infty}$ \textup{dim}$(y_n)$;

		\item[(b)] if $y$ is $k$-regular then \textup{dim}$(y) =$ \textup{dim}$(y_n)$ for all $n$ large enough.
	\end{itemize}
	
\end{lemma}

For non-compact spaces the following general result is known:

\begin{lemma}[Lemma 2.1  of \cite{Nag18}] ${}$
	\label{regularity-pointedconvergence}
	
	\noindent 	Let   $(X_n,x_n)$ be   a sequence of pointed, proper, geodesically complete \textup{CAT}($\kappa$) spaces converging to some  $(X,x)$ in the pointed Gromov-Hausdorff sense. 
	Then \textup{dim}$(X) \leq \liminf_{n \to + \infty}$ \textup{dim}$(X_n)$.	
\end{lemma}

\vspace{3mm}
\section{Estimate of volume of balls from below}\label{sec-volume}
We fix again a complete, geodesic, GCBA-space $X$. \\  
From \eqref{k-inequality1} \& \eqref{k-inequality2} it follows that  there exists an upper bound 
%$$\mu_X( B(x,r)) \leq C\cdot r^k $$ 
for the measure of any tiny ball $ B(x,r)$; moreover one can find a uniform upper bound of the measure of {\em all} balls, independently of the center $x$,  provided that  $X$ satisfies a uniform packing condition at some scale (see Theorem \ref{char_pack} in Section \ref{sec-packingcovering} for a precise statement). 
It is less clear if there exists a lower bound on the measure, and in particular if this lower bound depends only on some universal constant.
Indeed in general the $\mu_X$-volume of balls of a given radius is not uniformly bounded below independently of the space $X$. \linebreak
For instance  consider the balls of radius $\frac{1}{2}$ inside $\mathbb{R}^n$: when $n$ grows the measure of these balls tends to $0$. The next theorem shows that if the dimension is bounded from above then there is a uniform bound from below to the measure of balls of a given (sufficiently small) radius:
\begin{theo}
	\label{boundbelowvolume}
	Let $X$ be a complete, geodesic, \textup{GCBA} metric space.\linebreak If $\textup{dim}(X)\leq n_0$ then for any $x\in X$ and any  $r < \min \lbrace 1,\rho_\textup{ac}(x) \rbrace$ it holds
	$$\mu_X(\overline{B}(x,r))\geq c_{n_0} \cdot r^{n_0},$$
	where $c_{n_0}$ is a constant only depending  on $n_0$.
\end{theo}
The proof of this fact is based on  ideas most of which are already present in \cite{LN19}.
First of all we have:
\begin{prop}
	Let $X$ be a complete, geodesic, \textup{GCBA} metric space and $x\in X$ be a point of dimension $n$. 
	Then there exists a $1$-Lipschitz, surjective map $P\colon T_x X \to \mathbb{R}^n$ such that: 
	\begin{itemize}
		\item[(a)] $P(O)=0$;
		\item[(b)] $P(\overline{B}(O,r)) = \overline{B}(0,r)$ for any $r>0$;
		\item[(c)] $d_T(V,O)=d_{\mathbb{R}^n}(P(V),0)$ for any $V\in T_xX$.
	\end{itemize}
\end{prop}
\begin{proof}
	As the point $x$ has dimension $n$ then the geometric dimension of $T_x X$ is $n$. This implies that $\Sigma_x X$ is a space of dimension $n-1$ satisfying the assumptions of Proposition 11.3 of \cite{LN19}. So there exists a $1$-Lipschitz surjective map $P'\colon \Sigma_x X \to \mathbb{S}^{n-1}$. We extend the map $P'$ to a map $P$ over the tangent cones by sending the point $V=(v,t)$  to the point $(P'(v),t)$. \linebreak 
	It is immediate to check that $P$ is surjective and that $P(0)=0$. \\
	Moreover the tangent cone over $\mathbb{S}^{n-1}$ is $\mathbb{R}^n$, as said in Example \ref{cono_sfera}; therefore the equality $P(\overline{B}(O,R)) = \overline{B}(0,R)$ follows directly from \eqref{euclidean-cone-formula}. 
	Always by \eqref{euclidean-cone-formula} we have $d_T(V,O)=d_{\mathbb{R}^n}(P(V),0)$ for any $V\in C_xX$. Finally the $1$-Lipschitz property of $P$ follows from the same property of $P'$ and from the properties of the cosine function.
\end{proof}
Combining this result with the properties of the logarithmic map explained in Section \ref{subsec-log} we deduce  the following: 
\begin{prop}
	\label{PSI}
	Let $X$ be a complete, geodesic, \textup{GCBA} metric space and $x\in X$ be a point of dimension $n$. 
	Then there exists a $2$-Lipschitz, surjective map $\Psi_x \colon B(x,\rho_\textup{ac}(x)) \to \mathbb{R}^n$ such that 
	\begin{itemize}
		\item[(a)] $\Psi_x(x)=0$;
		\item[(b)] $\Psi_x(\overline{B}(x,r)) = \overline{B}(0,r)$ for any $0<r<\rho_\textup{ac}(x)$;
		\item[(c)] $d(x,y)=d(0,\Psi_x(y))$ for any $y\in B(x,\rho_\textup{ac}(x))$.
	\end{itemize}
\end{prop}
\begin{proof}
	Define $\Psi_x = P \circ \log_x$, where $P$ is the map of the previous proposition and $\log_x$ is the logarithmic map at $x$. Then $\Psi$ satisfies the thesis.
\end{proof}
Using the map $\Psi_x$ we can transport metric and measure properties from $\mathbb{R}^n$ to $X$. We denote by $\omega_n$ the $\mathcal{H}^n$-volume of the ball of radius $1$ of $\mathbb{R}^n$.
\begin{cor}\label{boundbelowvolume2}
	Let $X$ be a complete, geodesic, \textup{GCBA} metric space and $x\in X$ be a point of dimension $n$. Then
	\vspace{-3mm}
	
	$$\mathcal{H}^n(B(x,r))\geq \frac{1}{2^n}\omega_n r^n$$
	for any $0<r<\rho_\textup{ac}(x)$.
\end{cor}
\begin{proof}
	It follows directly from the properties of the map $\Psi_x$ and the behaviour of the Hausdorff measure under Lipschitz maps.
\end{proof}

\begin{proof}[Proof of Theorem \ref{boundbelowvolume}]
	We fix $x\in X$, $0<r <\min \lbrace 1,\rho_\text{ac}(x) \rbrace$ and $\varepsilon = \frac{r}{2n_0}$. We call $d_0$ the dimension of $x$. 
	We look for the biggest ball around $x$ of Hausdorff dimension exactly $d_0$. In order to do that we define
	$$r_1 = \sup\lbrace \rho>0 \text{ s.t. HD}(B(x,\rho)) = d_0\rbrace.$$
	(where  HD denotes the Hausdorff dimension). Notice that HD$(B(x,\rho))$ is monotone increasing in $\rho$.
	If $r_1\geq r$ we stop and we redefine $r_1 = r$. \linebreak
	Otherwise there exists a point $x_1$ such that $d(x,x_1)\leq r_1 + \varepsilon$ and the dimension of $x_1$ is $d_1 > d_0$, by definition of $r_1$. Now we look for the biggest ball around $x_1$ of Hausdorff dimension $d_1$. We define
	$$r_2 = \sup\lbrace \rho>0 \text{ s.t. HD}(B(x_1,\rho)) = d_1\rbrace.$$
	Arguing as before, if $r_1 + \varepsilon + r_2\geq r$ we stop the algorithm and we redefine $r_2$ as $r = r_1 + \varepsilon + r_2$. Otherwise we can find again a point $x_2$ such that $d(x_2,x_1)\leq r_2 + \varepsilon$ and whose dimension is $d_2>d_1$.
	We continue the algorithm until $r_1+ \varepsilon +\ldots + r_k = r$. It happens in at most $n_0$ steps. At the end we have points $x = x_0, x_1, \ldots x_k$ with $k\leq n_0$ such that $d(x_i,x_j)\leq r_j + \varepsilon$, $r_1+ \varepsilon +\ldots +r_k = r$ and such that the dimension of $x_j$ is $d_j$, with $d_i> d_j$ if $i>j$.
	We observe that the $d_j$-dimensional parts of the balls $B(x_j,r_j)$, denoted by $B^{d_j}(x_j,r_j)$, are disjoint and contained in $\overline{B}(x,r)$, by construction. Moreover the open ball $B(x_j,r_j)$ has no point of dimension greater than $d_j$.
	So 
	$$\mu_X(\overline{B}(x,r)) = \sum_{k=0}^{n_0}\mathcal{H}^k\llcorner \overline{B}^k(x,r) \geq \sum_{j}\mathcal{H}^{d_j} (B^{d_j}(x_j,r_j)).$$
	The last step is to estimate the last term of the sum. 
	Since $k\leq n_0$ and $r_1+ \varepsilon + \ldots +r_k = r$ then $r_1+\ldots + r_k = r - (k-1)\varepsilon \geq \frac{r}{2}$. Hence there exists an index $j$ such that $r_j\geq \frac{r}{2n_0}$.
	By definition any point of the ball $B(x_j,r_j)$ is of dimension $\leq d_j$. Hence by the properties of the Hausdorff measure we get
	$$\mathcal{H}^{d_j}(B^{d_j}(x_j,r_j)) =  \mathcal{H}^{d_j}(B(x_j,r_j)) \geq \frac{1}{2^{d_j}}\omega_{d_j} r_j^{d_j} \geq c_{n_0} r^{n_0},$$
	where the first inequality follows directly from the previous corollary and the last one holds since $r\leq 1$. So we can choose 
	$$c_{n_0} = \bigg(\frac{1}{4n_0}\bigg)^{n_0}\min_{k=0,\ldots,n_0}\omega_k$$
	that is a constant depending only on $n_0$. This concludes the proof.
\end{proof}

\vspace{3mm}
\section{Packing in GCBA-spaces}\label{sec-packingcovering}
Let $Y \subset X$ be any subset of a  metric space:\\
-- a subset $S$ of $Y$ is called  {\em $r$-dense}   if   $\forall y \in Y$  $\exists z\in S$ such that $d(y,z)\leq r$; \\
-- a subset $S$ of $Y$ is called  {\em $r$-separated} if   $\forall y,z \in S$ it holds $d(y,z)> r$.\\
The {\em $r$-packing number of $Y$} is the maximal cardinality of a $2r$-separated subset of $Y$ and   is denoted by $\text{Pack}(Y,r)$. 
The {\em $r$-covering number of $Y$}  is the minimal cardinality of a $r$-dense subset of $Y$ and   is denoted by $\text{Cov}(Y,r)$. These two quantities are classically related by the following relations:
\begin{equation}
	\label{packing-covering}
	\text{Pack}(Y,2r) \leq \text{Cov}(Y,2r) \leq \text{Pack}(Y,r).
\end{equation}
%\begin{proof}
%	First of all we prove Pack$(Y,2r) \leq$ Cov$(Y,2r)$. Let $x_1,\ldots,x_n$ be any $4r$-separated subset of $Y$ and $y_1,\ldots,y_m$ be a $2r$-dense subset of $Y$ with minimal cardinality. In particular $Y \subset \bigcup_{i=1}^n \overline{B}(y_i,2r)$. If $n>m$ then there are two points $x_i,x_j$ that belong to the same ball $\overline{B}(y_l,2r)$, hence $d(x_i,x_j)\leq 4r$ and this is not possible.
%	
%	Now it remains to prove Cov$(Y,2r) \leq$ Pack$(Y,r)$. Let $x_1,\ldots,x_n$ be a $2r$-separated subset of $Y$ of maximal cardinality. Then for any $y\in Y$ there exists $x_i$ such that $d(y,x_i)\leq 2r$ because of the maximality of $x_1,\ldots,x_n$. So $x_1,\ldots,x_n$ is a $2r$-dense subset of $Y$, hence Cov$(Y,2r)\leq$ Pack$(Y,r)$.
%\end{proof}
On a given space $X$  the numbers Pack$(\overline{B}(x,R),r)$ and Cov$(\overline{B}(x,R),r)$,  for $0 < r \leq R$,  depend in general on the chosen point $x$. We are interested in the case where these numbers can be bounded independently of $x\in X$. 
Therefore consider the functions
$$\text{Pack}(R,r)=\sup_{x\in X}\text{Pack}(\overline{B}(x,R),r), \qquad \text{Cov}(R,r)=\sup_{x\in X}\text{Cov}(\overline{B}(x,R),r)$$
called, respectively, the {\em packing and covering functions} of $X$.
They take values on $[0,+\infty]$; moreover, as an immediate consequence of \eqref{packing-covering}, we have
\begin{equation}
	\label{packing-covering-2}
	\text{Pack}(R,2r)\leq \text{Cov}(R,2r)\leq \text{Pack}(R,r).
\end{equation}

\begin{defin}\label{def-packing}
	Let $X$ be a metric space and let $C_0,P_0,r_0>0$.\\
	We say that {\em $X$ is $P_0$-packed at scale $r_0$} if Pack$(3r_0,r_0)\leq P_0$, that is every ball of radius $3r_0$ contains no more than $P_0$ points that are $2r_0$-separated.\\
	Analogously we say that $X$ is {\em $C_0$-covered at scale $r_0$} if Cov$(3r_0,r_0)\leq C_0$, i.e. every ball of radius $3 r_0$ can be covered by at most $C_0$ balls of radius $r_0$.
\end{defin}

The next theorem affirms that the packing functions can be well controlled  for complete, locally CAT$(\kappa)$-spaces which are locally geodesically complete (notice that no local compactness is assumed, since it will {\em follow} from the packing condition):

\begin{theo}
	\label{CAT-packing}
	Let $X$ be a complete, locally \textup{CAT}$(\kappa)$, locally geodesically complete, geodesic metric space with $\rho_{\textup{ac}}(X) > 0$. Suppose that $X$ satisfies 
	$$\textup{Pack}\left(3r_0,\frac{r_0}{2}\right) \leq P_0  \hspace{5mm} \text{ for } \hspace{1mm}   0<r_0 < \rho_\textup{ac}(X)/3.$$ 
	Then  $X$ is proper and geodesically complete; so it is a \textup{GCBA} metric space. \linebreak
	Moreover for any $0<r\leq R$ it holds:
	$$\textup{Pack}(R,r)\leq P_0(1+P_0)^{\frac{R}{r} - 1} \text{, if } r\leq r_0;$$
	$$\textup{Pack}(R,r)\leq P_0(1+P_0)^{\frac{R}{r_0} - 1} \text{, if } r > r_0.$$
\end{theo}

We want to remark that, in general, a control of the packing function at some fixed scale does not imply any control at smaller scales, as shown in the following example.
\begin{ex}
	\label{example-packing}
	Let $\mathbb{D}^n \subset \mathbb{R}^n$ be the closed Euclidean disk of radius $1$. \linebreak
	Let $X_n$ be the space obtained gluing a Euclidean ray $[0,+\infty)$ to a point of the boundary of $\mathbb{D}^n$. Fix $r_0 = 1$. Any $2r_0$-separated subset $S$ of $X_n$ contains at most one point of $\mathbb{D}^n$. Hence Pack$(3r_0,r_0)\leq 2$, in other words $X_n$ is $2$-packed at scale $1$ for every $n$. However at smaller scales, for example at scale $r = \frac{1}{4}$, we can easily show that Pack$(3r,r) \to +\infty$ when $n\to+\infty$. \linebreak
	Notice that the spaces $X_n$ in this example are complete and CAT(0) but they fail to be {\em geodesically complete}.
	
\end{ex}
We also remark  that a packing condition to scales bigger than the almost-convexity radius does not propagate to smaller scales: 
\begin{ex}
	Let $X_n$ be the graph with one vertex and $n$ loops of length $1$. \linebreak
	For any $n$ we glue an half-line to the vertex obtaining a complete, GCBA, geodesic metric space $Y_n$. As in Example \ref{example-packing} it is easy to show that at big scales the spaces $Y_n$ satisfy a uniform packing condition, while at small scales they do not.
\end{ex}

The proof of Theorem \ref{CAT-packing} is based on some preliminary lemmas.
\begin{lemma}
	\label{packingextension}
	Let $X$ be a space satisfying   the assumptions of Theorem \ref{CAT-packing}. \linebreak
	Then $X$ is $P_0$-packed at scale $r$ for any $r\leq r_0$.
\end{lemma}
\begin{proof}
	We fix $x\in X$ and $r\leq r_0$. We take a $2r$-separated subset $\lbrace x_1,\ldots,x_N\rbrace$ of $\overline{B}(x,3r)$. We consider the contraction map $\varphi^{3r_0}_{3r}$ which is surjective and $\frac{2r}{r_0}$-Lipschitz. For any $i$ we fix a preimage $y_i$ of $x_i$ under $\varphi^{3r_0}_{3r}$. We have \linebreak
 $2r < d(x_i,x_j) \leq \frac{2r}{r_0}d(y_i,y_j)$  for any $i\neq j$. This means that the set $\lbrace y_1,\ldots, y_N\rbrace$ is $r_0$-separated in $\overline{B}(x, 3r_0)$, hence $N\leq P_0$.
\end{proof}

\enlargethispage{6mm}

\begin{cor}
	Let $X$ be as in Theorem \ref{CAT-packing}. Then $X$ is locally compact. 
\end{cor}
\begin{proof}
	We fix a point $x\in X$. The ball $\overline{B}(x,3r_0)$ is complete since it is closed and $X$ is complete. Moreover for any $\varepsilon > 0$ the maximal cardinality of a $\varepsilon$-separated subset of $\overline{B}(x,3r_0)$ is finite, hence this ball is totally bounded. We can conclude it is compact.
\end{proof}

 As a consequence, since   $X$ is a locally compact, complete, geodesic metric space, then by Hopf-Rinow theorem it is proper. Moreover since it is complete and locally geodesically complete then it is also geodesically complete. This proves the first assertion of Theorem \ref{CAT-packing}.\\
We will now prove that the $P_0$-packing condition at every scale $r\leq r_0$ implies the announced estimate of   $\textup{Pack}(R,r)$ for every $R$. First, we show:

\begin{lemma}
	\label{packingbigscales}
	Let $X$ be a geodesic metric space that is $P_0$-packed at scale $r_0$. Then for any $R\geq 3r_0$ it holds:
\vspace{-4mm}	
	
	$$\textup{Pack}(R,r_0)\leq P_0(1+P_0)^{\frac{R}{r_0} - 1}.$$
\end{lemma}
%First of all we observe two simple consequences of the definition of packing number.
%\begin{obs}
%	Let $Y,Z\subset X$ be two subsets. Then 
%	\begin{itemize}
%		\item \textup{Pack}$(Y\cup Z, r) \leq$ \textup{Pack}$(Y,r) +$ \textup{Pack}$(Z,r);$
%		\item if $Y\subset Z$ then \textup{Pack}$(Y,r)\leq$ \textup{Pack}$(Z,r)$.
%	\end{itemize}
%\end{obs}
%The proof of these properties is trivial.

\begin{proof}
	We prove the thesis by induction on $k$, where $k$ is the smallest integer such that $R \leq 3r_0 + kr_0$.
	The case $k=0$ clearly holds as for $R = 3r_0$ we have  Pack$(R,r_0) \leq P_0 \leq P_0(1+P_0)^2$. 
	Let now $k \geq 1$ and $R\geq 3r_0$ such that $R \leq 3r_0 + kr_0$.
	We consider the sphere $S(x,R-r_0)$ of points at distance exactly $R- r_0$ from $x$. We observe that $R-r_0 \leq 3r_0 + (k-1)r_0$, so by induction we can find a $2r_0$-separated subset $y_1,\ldots,y_n$ of $S(x,R-r_0)$ of maximal cardinality, where $n\leq P_0(1+P_0)^{\frac{R-r_0}{r_0} -1}$. 
	Moreover 
	\vspace{-4mm}	
	
	$$\bigcup_{i=1}^n \overline{B}(y_i,3r_0) \supset A(x, R - r_0, R).$$ 
	Indeed for any $y\in A(x,R - r_0,R)$ we take a geodesic $[x,y]$ and we call $y'$ the point on the geodesic $[x,y]$ at distance $R-r_0$ from $x$. Then $y\in \overline{B}(y',r_0)$. Moreover there exists $y_i$ such that $d(y',y_i)\leq 2r_0$, because of the maximality of the set $\lbrace y_1,\ldots,y_n \rbrace$. Hence $d(y,y_i)\leq 3r_0$.
	Therefore we get:
	\begin{equation*}
		\begin{aligned}
			\text{Pack}(\overline{B}(x,R),r_0) &\leq \text{Pack}(\overline{B}(x,R - r_0),r_0) + \text{Pack}(A(x,R - r_0,R),r_0) \\
			&\leq\text{Pack}(\overline{B}(x,R - r_0),r_0) + \sum_{i=1}^n\text{Pack}(\overline{B}(y_i,3r_0),r_0).
		\end{aligned}
	\end{equation*}
	Since $\text{Pack}(\overline{B}(y_i,3r_0),r_0)\leq P_0$, we obtain 
	\begin{equation*}
		\begin{aligned}
			\text{Pack}(\overline{B}(x,R),r_0) &\leq \text{Pack}(\overline{B}(x,R - r_0),r_0) + P_0\cdot n\\
			&\leq\text{Pack}(\overline{B}(x,R - r_0),r_0) + P_0\cdot \text{Pack}(\overline{B}(x,R - r_0),r_0)\\
			&\leq (1+P_0)P_0(1+P_0)^{\frac{R - r_0}{r_0} - 1}= P_0(1+P_0)^{\frac{R}{r_0} - 1}.
		\end{aligned}
	\end{equation*}
	
	\vspace{-7mm}	
\end{proof}
\noindent We can now prove Theorem \ref{CAT-packing}.

\begin{proof}[Proof of Theorem \ref{CAT-packing}.]
	We have already shown that $X$ is proper and geodesically complete and that it is  $P_0$-packed at every scale $0<r \leq r_0$.  Therefore,  for these values of $r$,  Lemma \ref{packingbigscales} yields 
	$$\text{Pack}(R,r) \leq P_0(1+P_0)^{\frac{R}{r} - 1}$$
	$ \forall R \geq 3r$; but this also holds for $R\leq 3r$ since then $\text{Pack}(R,r) \leq \text{Pack}(3r,r)$.
	On the other hand if   $r\geq r_0$ the thesis follows directly from Lemma \ref{packingbigscales}.
	Indeed  when $R\geq 3r_0$ then $\text{Pack}(R,r) \leq \text{Pack}(R,r_0)$ and  Lemma \ref{packingbigscales} concludes.  
	If $R < 3r_0$ we get
	$$\text{Pack}(R,r)\leq \text{Pack}(R,r_0) \leq \text{Pack}(3r_0,r_0)\leq P_0$$
	and $P_0(1+P_0)^{\frac{R}{r} - 1} \geq P_0.$
\end{proof}

\vspace{2mm}
We can read this result in terms of the covering functions instead of the packing functions using \eqref{packing-covering-2}.
\begin{cor}
	\label{covering}
	Let $X$ be a complete, locally \textup{CAT}$(\kappa)$, locally geodesically complete, geodesic metric space with $\rho_\textup{ac}(X) > 0$. Suppose that $X$ satisfies  
	\vspace{-2mm}
	
	$$\textup{Cov}\left(3r_0,\frac{r_0}{2}\right) \leq C_0 \hspace{5mm} \text{ for } \hspace{1mm}   r_0 < \rho_\textup{ac}(X)/3 .$$ Then for any $0<r\leq R$ it holds:
	\vspace{-3mm}	
	
	$$\textup{Cov}(R,r)\leq C_0(1+C_0)^{\frac{2R}{r} - 1} \text{, if } r\leq 2r_0;$$
	\vspace{-8mm}
	
	$$\textup{Cov}(R,r)\leq C_0(1+C_0)^{\frac{2R}{r_0} - 1} \text{, if } r > 2r_0.$$
\end{cor}
\begin{proof}
	By \eqref{packing-covering-2} we have that $X$ satisfies $\textup{Pack}(3r_0,\frac{r_0}{2}) \leq C_0$. Hence we can apply the previous proposition to get:
	\vspace{-3mm}	
	
	$$\text{Cov}(R,r) \leq \text{Pack}\left(R,\frac{r}{2}\right) \leq C_0(1+C_0)^{\frac{2R}{r} - 1}, \hspace{5mm} \text{ if  } \hspace{1mm}  \frac{r}{2} \leq r_0$$  
	\vspace{-8mm}
	
	$$\text{Cov}(R,r) \leq C_0(1+C_0)^{\frac{2R}{r_0} - 1}, \hspace{5mm} \text{ if  } \hspace{1mm}   \frac{r}{2} > r_0.$$
	\vspace{-10mm}
	
\end{proof}

%{\color{blue}
%We notice that the above facts have the following consequence (the first of which is implicit in \cite{LN19},  cp. Proposition 5.1 therein):
%
%\begin{cor}\label{cor-upperbound}
%	Let $X$ be as in Theorem \ref{CAT-packing}. Then,  there exists a constant $C$, \linebreak only depending on $r_0, P_0$ such that every tiny ball of radius $r\leq r_0$ satisfies:
%\vspace{-3mm}
%	
%	 $$\mathcal{H}^k(B(x,r)^k) \leq C\cdot r^k,$$ 
%          $$\mu_X (B(x,r)) \leq C\cdot (r+\cdots+r^n)$$ 
%\end{cor}
%\begin{proof} By Cor.11.8 of \cite{LN19}, the constant $C$ in (\ref{k-inequality}) only depends on the packing constant of $\bar B(x,10r)$  which,  in turns,   by Lemma \ref{packingextension}., only depends on $P_0$ and $r_0$. \linebreak
% \end{proof}
%}

We are ready to characterize the packing condition in terms of dimension and measure of a GCBA metric space.
\begin{theo}
	\label{char_pack}
	Let $X$ be a complete, geodesic \textup{GCBA}$^\kappa$  metric space with $\rho_\textup{ac}(X) \geq \rho_0 > 0$. The following facts are equivalent.
	\begin{itemize}
			\vspace{1mm}
		\item[(a)] There exist $P_0 > 0$ and $0<r_0< \frac{\rho_0}{3}$ such that $\textup{Pack}(3r_0, \frac{r_0}{2}) \leq P_0$;
			\vspace{1mm}
		\item[(b)] There exist $n_0,V_0,R_0>0$ such that $\textup{dim}(X) \! \leq \! n_0$ and $\mu_X(B(x,R_0)) \!\leq\! V_0$ \linebreak for any $x\in X$;
			\vspace{1mm}
		\item[(c)] There exists a measure $\mu$ on $X$ and there exist two functions $c(r), C(r)$ such that for any $x\in X$ and for any $0<r<\rho_0$:
		\vspace{-4mm}	
		
		$$0< c(r) \leq \mu(B(x,r)) \leq C(r) < +\infty.$$
	\end{itemize}	
	
	\noindent Moreover the set of constants $(n_0,V_0,R_0, \rho_0,\kappa)$ can be expressed only in terms of the set of constants $(P_0,r_0,\rho_0,\kappa)$ and viceversa. \\
	Finally if any of the above conditions holds then the natural measure $\mu_X$ satisfies condition (c) and $X$ is proper and geodesically complete.
\end{theo}

\begin{proof}
	Assume first that $X$ satisfies Pack$(3r_0,\frac{r_0}{2}) \leq P_0$. 	
	First of all it follows that the dimension of $X$ is bounded. Indeed we fix any point $x\in X$ and we denote by $n$ its dimension. We consider the map $\Psi_x \colon B(x,2r_0) \to \mathbb{R}^{n}$ given by Proposition \ref{PSI}. 
	Let $x_1,\ldots, x_k$ be a $2r_0$-separated subset of $\overline{B}_{\mathbb{R}^n}(0,2r_0)$. Since $\Psi_x$ is surjective we can take preimages $y_i$ of $x_i$ under $\Psi_x$.
	Moreover $d(y_i,x)=d(\Psi_x(y_i),0)$, hence $y_i \in \overline{B}(x,2r_0)$. As $\Psi_x$ is $2$-Lipschitz the set $\lbrace y_1,\ldots,y_k\rbrace$ is a $r_0$-separated subset of $\overline{B}(x,2r_0)$. Then 
	$$k\leq  \textup{Pack} \left(2r_0,\frac{r_0}{2}\right) \leq  \textup{Pack} \left(3r_0,\frac{r_0}{2}\right) \leq P_0$$
	by Theorem \ref{CAT-packing}.
	But it is easy to show that $k \geq 2n$. Therefore $2n\leq P_0$ is the bound on the dimension we were looking for. We observe that this bound is expressed only in terms of $P_0$.
	We fix now $x\in X$ and any $R > 0$. Let $r=\min \lbrace 1, R, \frac{1}{10}r_0, \frac{1}{100}D_\kappa \rbrace$. We take a covering of $\overline{B}(x,R)$ with balls of radius $r$. By Theorem \ref{CAT-packing} it is possible to do that with $k$ balls, where $k$ can be estimated in the following way:
	$$k = \text{Cov}(\overline{B}(x,R),r) \leq \text{Pack}\bigg(\overline{B}(x,R),\frac{r}{2}\bigg) \leq P_0(1+P_0)^{\frac{2R}{r} - 1}.$$
	We call $y_1,\ldots, y_k$ the centers of these balls. By Theorem \ref{CAT-packing} the space $X$ is proper, then from the choice of $r$ we get that $B(y_i,r)$ is a tiny ball for any $i$, as follows from \eqref{ac-radius,cat-radius}. Moreover the maximal number of $r$-separated points inside $\overline{B}(y_i, 10r)$ is bounded by $\text{Pack}(10r,\frac{r}{2}) \leq P_0(1+P_0)^{19}$, as follows again by Theorem \ref{CAT-packing}. 
	Hence by \eqref{k-inequality1} we have
	$$\mathcal{H}^j(\overline{B} (y_i,r)^j)\leq C(P_0)r^j,$$
	where $C(P_0)$ is a constant depending only on $P_0$.
	Therefore, using the fact that the dimension of $X$ is bounded above by $n_0 = \frac{P_0}{2}$ and  $r\leq 1$, we get:
	$$\mu_X(\overline{B}(y_i,r)) = \sum_{j=0}^{n_0}\mathcal{H}^j( \overline{B} (y_i,r)^j )\leq \frac{P_0}{2} \cdot C(P_0)$$
	
\vspace{-3mm}
\noindent	for any $i$.
	Finally
\vspace{-4mm}	
	
	\begin{equation}
		\label{volumeestimate}
		\mu_X(\overline{B}(x,R))\leq P_0(1+P_0)^{\frac{2R}{r} - 1}\cdot \frac{P_0}{2} \cdot C(P_0) = V(P_0, r_0, R, \kappa).
	\end{equation}
	This shows that for any $x\in X$ and any $R_0$ we can find the desired uniform bound on the volume of the ball $B(x,R_0)$. This ends the proof of the implication (a) $\Rightarrow$ (b).
	Moreover this part of the proof, together with Theorem \ref{boundbelowvolume}, shows that if (a) holds then the measure $\mu_X$ is a measure that satisfies condition (c) of the theorem.
	
	\noindent 	Assume now that has dimension bounded above by $n_0$ and that the volume of the balls of radius $R_0$ are uniformly bounded above by $V_0$. We set \linebreak $r_0 = \min\lbrace \frac{R_0}{6},1, \frac{\rho_0}{6}\rbrace$. The claim is that $X$ satisfies $\text{Pack}(3r_0,\frac{r_0}{2}) \leq P_0$ for some $P_0$ depending only on $V_0, R_0, n_0$ and $\rho_0$.
	We consider the ball of radius $\frac{R_0}{2}$ centered at a point $x\in X$. We take a $r_0$-separated subset of $\overline{B}\left(x,\frac{R_0}{2}\right)$ and we suppose its cardinality is bigger than some $k$. It means that there are $k$ points $y_1,\ldots,y_k \in \overline{B}\left(x, \frac{R_0}{2}\right)$ such that $d(y_i,y_j)>r_0$ for any $i\neq j$. Hence the balls centered at $y_i$ of radius $\frac{r_0}{2}$ are pairwise disjoint and satisfy $\overline{B}\left(y_i,\frac{r_0}{2}\right)\subset \overline{B}\left(x, \frac{R_0}{2} + \frac{r_0}{2}\right)\subset B(x, R_0)$, since $\frac{R_0}{2} + \frac{r_0}{2} \leq \frac{R_0}{2} + \frac{R_0}{6} < R_0$. \\
	We can apply Theorem \ref{boundbelowvolume} to get $\mu_X\left(\overline{B}(y_i,\frac{r_0}{2})\right)\geq c_{n_0}\left(\frac{r_0}{2}\right)^{n_0}$ for any $i$. Thus 
	$$V_0 \geq \mu_X(B(x,R_0))\geq \sum_{i=1}^k \mu_X\bigg(\overline{B}\bigg(y_i,\frac{r_0}{2}\bigg)\bigg) \geq k\cdot c_{n_0}\bigg(\frac{r_0}{2}\bigg)^{n_0},$$
	then 
	$$k\leq \frac{2^{n_0}V_0}{c_{n_0}r_0^{n_0}} = \frac{2^{n_0}V_0}{c_{n_0}}\cdot \max\left\lbrace 1, \left(\frac{6}{\rho_0}\right)^{n_0}, \left(\frac{6}{R_0}\right)^{n_0}\right\rbrace = P_0.$$
	It means that Pack$(\overline{B}(x,\frac{R_0}{2}),\frac{r_0}{2})\leq P_0$.
	Since $R_0\geq 6r_0$ we can conclude that Pack$(3r_0,\frac{r_0}{2})\leq P_0$ that is what claimed.

	\noindent Finally assume that there exists a measure $\mu$ such that for any $x\in X$ and for any $0<r<\rho_0$ it holds
	$$0< c(r) \leq \mu(B(x,r)) \leq C(r) < +\infty.$$
	We take any $r_0 < \frac{\rho_0}{3}$ and we fix any point $x\in X$. Let $k$ be the maximal cardinality of a $r_0$-separated subset of $\overline{B}(x,3r_0)$. Then, arguing as before, we can find $k$ disjoint balls of radius $\frac{r_0}{2}$ contained in $B(x, 4r_0)$. Since \linebreak $C(4r_0)\geq \mu(B(x,4r_0)) \geq k\cdot c(\frac{r_0}{2})$ then $k\leq \frac{C(4r_0)}{c(\frac{r_0}{2})} = P_0$.
	This shows that $X$ satisfies (a) with these choices of $r_0$ and $P_0$.
\end{proof}

\vspace{3mm}
\section{The doubling condition in GCBA-spaces}
%\section{The doubling property for the natural measure}
\label{sec-doubling}
In this section $X$ will be a complete, geodesic GCBA-space. \\
We say that $X$ is {\em purely $n$-dimensional} if dim$(x)=n$ for any $x\in X$.\linebreak
Moreover we say that a measure $\mu$ on $X$ is:
\begin{itemize}
	\item {\em D-doubling up to the scale $t$ at  $x\in X$} if there exists a constant $D>0$ such that for any $0<t'\leq t$ it holds
	$$\frac{\mu(B(x,2t'))}{\mu(B(x,t'))}\leq D ;$$
	\item {\em $D$-doubling up to scale $t$} if it is $D$-doubling up to scale $t$ at any point $x\in X$ (for a uniform doubling constant $D$).
\end{itemize}
When uniformity of the constant and of the scale is not an issue we will simply say that   $\mu$ is {\em locally doubling} on $X$: that is  if for any $x\in X$ there exist $t_x>0$ and $D_x > 0$ such that $\mu$ is $D_x$-doubling up to scale $t_x$ at  any point of $B(x,t_x)$. 
\vspace{1mm}

\begin{obs}
	\label{rem-doubpack}
	{\em 
		Notice that any metric measured space $(X,\mu)$ satisfying a \linebreak $D_0$-doubling condition up to scale $t_0$ is $P_0$-packed at scale $r_0=\frac{t_0}{4}$  for  $P_0=D_0^4$ \linebreak (provided that the measure gives   positive mass to the balls of positive radius).\linebreak
		Actually let $x \in X$  and take any   $r_0$-separated subset $\lbrace y_1,\ldots, y_k\rbrace$ of  $\overline{B}(x,3r_0)$. So the balls $B(y_i,\frac{r_0}{2})$ are pairwise disjoint.
		From the doubling property  we get:
		\vspace{-7mm}
		
		$$\mu_X \left(B(x,3r_0)\right) \geq \sum_{i=1}^k \mu \left( B ( y_i,  r_0/2 ) \right) 
		\geq   \sum_{i=1}^k  \frac{1}{D_0^4}\mu (B(y_i,8r_0))$$	
		%	$ \geq \frac{1}{D_0^4}\mu (B(x,3r_0)).$ 
		and since $ B(y_i, 8r_0) \supset B(x,3r_0)$ we deduce that $k \leq  D_0^4$. 
		
		%In other words $\text{Pack}(3r_0, \frac{r_0}{2}) \leq P_0$.
	}
\end{obs}		

The next result characterizes GCBA-spaces whose natural measure is locally doubling:
\begin{theo}
	\label{pure-dimensional}
	Let $X$ be a proper, geodesic \textup{GCBA} metric space. Suppose $\mu_X$ is locally doubling: then $X$ is purely $n$-dimensional for some $n$.
\end{theo}

\noindent We begin the proof of Theorem \ref{pure-dimensional} with the following two preliminary results.
\begin{lemma}
	Let $X$ be a proper, geodesic \textup{GCBA} metric space and $x\in X$.
	Let $v\in \Sigma_xX$ and assume that every point of $\overline{B}((v,1),\varepsilon)$ is a $k$-regular point of $T_xX$, for some $\varepsilon >0$. 
	Then there exists $r>0$ such that all points of the set \linebreak
	\vspace{-3mm}
	%Let $v\in \Sigma_xX$ such that $(v,1)$ is a $k$-regular point of $T_xX$. Let $\varepsilon > 0$ be such that any point of $\overline{B}((v,1),\varepsilon)$ is a $k$-regular point of $T_xX$. Then there exists $r > 0$ such that for any $r' < r$ any point of the set
	$$A_{v,\varepsilon}(r) = \lbrace y \in X \text{ s.t. } d_T(\log_x(y), (v, d(x,y)))  \leq \varepsilon d(x,y)\rbrace \cap B(x,r)$$
	have  dimension $k$. 
\end{lemma}
We recall that, since $T_xX$ is a GCBA-space and since the set of $k$-regular points is open in $T_xX$,   if $(v,1)$ is $k$-regular point  in $T_xX$  then it is always possible to find $\varepsilon$ satisfying the assumptions of the lemma. The set $A_{v,\varepsilon}$, or better its projection on the tangent cone through the logarithm map at $x$, should be thought as a part of angular sector around $v$. So the statement of the lemma says that any possible geodesic segment starting at $x$ with direction close to $v$ stay in the $k$-dimensional part of $X$ for a uniform time $r$.
\begin{proof}
	Suppose the thesis is false. Then, there exists a sequence of points $y_n$  of dimension different from $k$ at distance $r_n \to 0$ from $x$ such that 
	$$ d_T(\log_x(y_n), (v, r_n)) \leq \varepsilon r_n.$$
	We consider   rescaled tiny balls $Y_n = \frac{1}{r_n}  B(x,r_0)$ as in Lemma \ref{lemma-tangentcone=limit}, together with the approximating maps $f_n$; so for all $n$ we have:
	$$d_T(f_n(y_n), (v,1)) \leq \varepsilon.$$
	Moreover we are in the standard setting of convergence. 
	Indeed the GCBA-space $X$ is geodesic and complete, so the contraction maps $\varphi^R_r$ are well-defined for any $R < \rho_{\text{cat}} (x)$ and they are surjective and $\frac{2r}{R}$-Lipschitz; therefore, by applying   the same proof as in Lemma \ref{packingextension},  we conclude that the rescaled balls are uniformly packed (the other properties follow from the discussion in Section \ref{sec-preliminaries}).
	Moreover the sequence  $y_n \in Y_n$ converges to some point $y_\infty \in \overline{B}((v,1),\varepsilon)$. So $y_\infty$ is $k$-regular by assumption. But, by Lemma \ref{regularity-convergence},  the points $y_n$ must be $k$-dimensional for $n$ large enough, which is a contradiction.
\end{proof}

\begin{lemma}
	Let $v\in \Sigma_x X$ and let $\gamma$ be a geodesic starting at $x$ defining $v$. For any $ 0 < \varepsilon < 1$ %there exists $\delta > 0$ such that 
	%$$B\bigg(\gamma(r), \frac{\varepsilon r}{12}\bigg) \subset A_{v,r}$$
	%	for any $r<\delta$.
	we have, for all $r>0$ small enough:
	$$B\bigg(\gamma \left( \frac{r}{2} \right), \frac{\varepsilon r}{8}\bigg) \subset A_{v,\varepsilon} (r)$$
	
\end{lemma}
\begin{proof}
	As the logarithm map is $2$-Lipschitz, for every $y\in B\left(\gamma \left( \frac{r}{2} \right), \frac{\varepsilon r}{8}\right)$ we have
	
	\begin{equation*}
		\begin{aligned}
			d_T(\log_x(y), (v,d(x,y))) &\leq d_T \left(\log_x(y), \log_x \left(\gamma \left(\frac{r}{2}\right)\right)\right) + d_T \left(\left(v,\frac{r}{2}\right), (v, d(x,y))\right) \\
			&\leq 2d\left(y,\gamma\left(\frac{r}{2}\right)\right) + \left| \frac{r}{2} - d(x,y) \right|  \\
			& \leq 3d \left(y,\gamma\left(\frac{r}{2}\right)\right)   \leq \frac{3 \varepsilon r}{8}
			\leq \varepsilon d(x,y)
		\end{aligned}
	\end{equation*}
	since $d(x,y) \geq \frac{r}{2} - \frac{\varepsilon r}{8}$.
	%\noindent (in fact this implies: \\
	%$r < \frac{4d(x,y)}{(4-\epsilon)}$, so $\frac{3\epsilon r}{4} < \frac{ 3\epsilon d(x,y) }{(4-\epsilon)} < \epsilon d(x,y)$ as $4-\epsilon >3$)
	On the other hand if $y \in B  (\gamma ( \frac{r}{2} ), \frac{\varepsilon r}{8})$  we have $d(x,y) \leq \frac{r}{2}+\frac{\varepsilon r}{8} < r$, so the ball $B  (\gamma ( \frac{r}{2} ), \frac{\varepsilon r}{8})$ is included in $A_{v,\varepsilon} (r)$.
\end{proof}

\begin{proof}[Proof of Theorem \ref{pure-dimensional}]
	Let us suppose $X$ is not pure dimensional. 
	We take a point $x_0\in X$ of minimal dimension $d_0$. Then we have by assumption 
	$$r_0 = \sup\lbrace \rho > 0 \text{ s.t. HD}(B(x_0,\rho)) = d_0\rbrace <+\infty.$$
	We can find a point $x  \in X$ with dimension $d > d_0$ such that $d(x_0,x) = r_0$.
	%	 In order to do that we will use the properness of $X$. 
	Indeed for any $n$ we can find a point $x_{n}$ such that $d(x_0, x_n) < r_0 + \frac{1}{n}$ and dim$(x_n) > d_0$. The sequence of points $x_n$ converge, as the space is proper,  to a point $x$ at distance exactly $r_0$ from $x_0$.
	Assume that dim$(x) = d_0$: then  there would exist  a small radius $\rho$ such that the Hausdorff dimensions of $B(x, \rho)$ is exactly $d_0$. But  $x_n$ belongs to  $B(x, \rho )$ for  $n \gg 0$, and   any open ball around $x_n$ has Hausdorff dimension strictly greater than $d_0$; therefore  HD$(B(x, \rho))>d_0$, a contradiction.  \\
	Now, the tangent cone  $T_{x } X$ at $x $  has dimension $d$. Hence there exists a point $v\in \Sigma_{x} X$ and $\varepsilon > 0$ such that any point of the ball $\overline{B}((v,1),\varepsilon)$ is regular and of dimension  $d$. We take any geodesic $\gamma$ starting at $x$ and defining $v$ and we set $y_r= \gamma(\frac{r}{2})$. Applying the two lemmas above we have that, for all $r$ small enough, any point of the ball $B(y_r, \frac{\varepsilon r}{8})$ is $d$-dimensional. \linebreak
	Since $X$ satisfies a doubling condition around $x$ we know by Remark \ref{rem-doubpack} that a ball $B(x,r_0)$ is $P_0$-packed, for some $r_0, P_0$  depending on $x$.  So, by Theorem \ref{CAT-packing} and by the properties of the natural measure recalled in Section \ref{subsec-structure}, there exists a constant $C$, only depending on  $r_0$ and $P_0$, such that for all sufficiently small $r$ we have: 
	$$\mu_X \left(B \left(y_r, \frac{\varepsilon r}{8}\right) \right) \leq C \cdot \left(\frac{\varepsilon r}{8}\right)^{d}.$$
	%	 where $C$ is a constant that does not depend on $r$ but just on $x$ (namely, it only depends on the packing constant of a fixed tiny  ball centred in $x$).\linebreak
	%as follows by Proposition 5.1 of \cite{LN19}.	
	Consider now the ball $B(y_r, r)$: notice that   there exists a ball of radius at least $\frac{r}{2}$ contained in $B(y_r, r) \cap B(x_0,r_0)$, so made only of $d_0$-dimensional points. \linebreak 
	In particular by Corollary \eqref{boundbelowvolume2} we have $\mu_X(B(y_r, r)) \geq c_{d_0}(\frac{r}{2})^{d_0}$, where $c_{d_0}$ is a constant depending only on $d_0$. Thus 
	$$\frac{\mu_X(B(y_r, r))}{\mu_X (B(y_r, \frac{\varepsilon r}{4}))} \geq C' r^{d_0 - d},$$
	where $C'$ is a constant that does not depend on $r$.
	Since this is true for any $r$ small enough and $d_0 < d$, this inequality contradicts  the doubling assumption at $y_r$, when $r$ goes to $0$.
\end{proof}

As a consequence of what proved in Section \ref{sec-packingcovering} we obtain the following:
\begin{cor}
	\label{doubling}
	Let $X$ be a complete, geodesic \textup{GCBA}$^\kappa$  metric space with $\rho_{\textup{ac}}(X) \geq \rho_0 > 0$. The following facts are equivalent:
		\vspace{1mm}
	\begin{itemize}
		\item[(a)] there exist $D_0 > 0$ and $t_0 > 0$ such that the natural measure $\mu_X$ is $D_0$-doubling up to scale $t_0$; 
			\vspace{1mm}
		\item[(b)] $X$ is purely dimensional and there exist $P_0 > 0$ and $0 < r_0 <\rho_0/3$ such that $\textup{Pack}(3r_0, \frac{r_0}{2})\leq P_0$; 
			\vspace{1mm}
		\item[(c)] there exist $n_0, V_0, R_0 > 0$ such that $X$ is purely $n_0$-dimensional and $\mu_X(B(x,R_0)) \leq V_0$ for any $x\in X$.
	\end{itemize}
	Moreover each of the three sets of constants $(D_0,t_0, \rho_0, \kappa)$, $(P_0,r_0, \rho_0, \kappa)$, $(n_0,V_0,R_0, \rho_0,\kappa)$ can be expressed in terms of the others. \\
	Finally if the conditions hold then $X$ is proper and geodesically complete.
\end{cor}
\begin{proof}[Proof of Corollary \ref{doubling}] ${}$ The implication (a) $\Rightarrow$  (b) follows from Theorem \ref{pure-dimensional} and  from Remark  \ref{rem-doubpack}together with  Theorem \ref{CAT-packing}.\\
	%We assume the natural measure on $X$ is $D_0$-doubling up to scale $t_0$. \linebreak 
	%	We set $r_0 = \min\lbrace \frac{t_0}{4}, \frac{\rho_0}{4} \rbrace$ and $P_0 = D_0^4$. We fix $x\in X$ and we take a set $\lbrace y_1,\ldots, y_k\rbrace$ of $r_0$-separated points inside $\overline{B}(x,3r_0)$. This means that the balls $B(y_i,\frac{r_0}{2})$ are pairwise disjoint. Moreover $B(x,3r_0)\subset B(y_i, 6r_0)$. \linebreak 
	%	For any $i$ we get the following estimate using the doubling property of $\mu_X$:
	%	$$\mu_X\bigg(B\bigg(y_i,\frac{r_0}{2}\bigg)\bigg) \geq \frac{1}{D_0^4}\mu_X(B(y_i,8r_0)) \geq \frac{1}{D_0^4}\mu_X(B(x,3r_0)).$$
	%	Therefore
	%	$$\mu_X(B(x,3r_0)) \geq \frac{k}{D_0^4}\mu_X(B(x,3r_0)).$$
	%	In other words $\text{Pack}(3r_0, \frac{r_0}{2}) \leq P_0$.
	%	By Theorem \ref{CAT-packing} we know that $X$ is proper. We conclude that $X$ is purely dimensional using Theorem \ref{pure-dimensional}. This ends the proof of the implication (a) $\Rightarrow$ (b).\\
	Assume now $X$  purely $n$-dimensional and $\text{Pack}(3r_0,\frac{r_0}{2})\leq P_0$.
	We recall that by Theorem \ref{char_pack} $n$ can be bounded from above in terms of $P_0$.
	We fix $t_0<\min \lbrace 1, R, \frac{1}{10}r_0, \frac{1}{100}D_\kappa \rbrace$ as in the proof of Theorem \ref{char_pack}. By Theorem \ref{CAT-packing} we know $X$ is proper, so it is easy to check that $\rho_\text{cat}(X) \geq t_0$ by \eqref{ac-radius,cat-radius}. Therefore by Theorem \ref{boundbelowvolume} we have 
	\vspace{-3mm}
	
	$$\mu_X(B(x,t)) \geq c_nt^n = c(P_0)t^n$$
	for any $t\leq t_0$. Moreover, by the same estimate used in the proof of Theorem \ref{char_pack}, and using the fact that $\mu_X$ is just the $n$-dimensional Hausdorff measure, we get
	\vspace{-5mm}
	
	$$\mu_X(B(x,2t)) \leq P_0(1+P_0)^3\cdot \frac{P_0}{2} \cdot  C(P_0)t^n$$
	for any $t\leq t_0$. Hence
	\vspace{-3mm}
	
	$$\frac{\mu_X(B(x,2t))}{\mu_X(B(x,t))} \leq \frac{P_0(1+P_0)^3\cdot \frac{P_0}{2} \cdot C(P_0)}{c(P_0)} =D_0$$
	which  shows the implication (b) $\Rightarrow$ (a). \\
	The equivalence between (b) and (c) is proved in Theorem \ref{char_pack}.
\end{proof}

Finally the doubling condition also implies   the  uniform continuity  of  the natural measure  of annuli: 

\begin{lemma}
	Let $X$ be a complete, geodesic, \textup{GCBA}$^\kappa$-space which is $D_0$-doubling up to scale $t_0$ and satisfies $\rho_{\textup{ac}}(X)\geq \rho_0$. There exists $\beta > 0$, only depending on $D_0$, such that for every $R>0$ and for every positive $\varepsilon  <  \min \left\{ \frac{t_0}{24R}, \frac19 \right\} $ 
	%there exists $\delta > 0$,  depending only on $D_0,t_0,\rho_0$ and $R$ such that for any $x\in X$ \linebreak
	it holds:
	$$\mu_X(A(x,R,(1-\varepsilon)R))\leq   \left( \max \left\{  \frac{24R}{t_0}, 9  \right\} \right)^\beta \cdot  \varepsilon^\beta \cdot \mu_X(B(x,R)).$$
\end{lemma}

\begin{proof}
	The proof is exactly the same as in  Proposition 11.5.3 of \cite{HKNS15}, with a minor modification  due to the fact that we assume that $\mu_X$ is  doubling only up to  scale $t_0$. 
	Actually, arguing as in the first part of the proof of Proposition 11.5.3 of \cite{HKNS15}, one deduces that
	\begin{equation}\label{eqHKNS}
		\mu_X(A(x, R, R-t)) \leq D_0^4\cdot \mu_X(A(x, R-t, R-3t))
	\end{equation}
	for all $x\in X$ and all positive $t\leq \min\left\{ \frac{t_0}{8}, \frac{R}{3}\right\}=:t_R$.
	From (\ref{eqHKNS}) we deduce that   for all $t \leq t_R$ it holds
	$$\mu_X\left(A(x, R, R -t )\right)
	\leq D_0^4 \Big( \mu_X(B(x,R)) - \mu_X \left( A (x, R, R - t)\right) \Big)
	$$
	hence  
	$$\mu_X(A(x,R, R - t)) \leq \left(\frac{D_0^4}{1+D_0^4}\right) \cdot \mu_X (B(x,R))$$
	Setting $t_m=\frac{1}{2\cdot 3^m}$ one then shows by induction as in  \cite{HKNS15} that 
	$$\mu_X \Big( A(x,R, (1-t_m)R) \Big) \leq \left(\frac{D_0^4}{1+D_0^4}\right)^{m+1-m_0} \cdot \mu_X (B(x,R))$$
	for all $m \geq m_0 = \left \lceil{ \log_3 (\frac{R}{2t_R} ) } \right \rceil $.
	Our claim then  follows for $\varepsilon \leq \min \left\{ \frac{t_0}{24R}, \frac19 \right\}$ choosing $\beta = \log_3\left(\frac{1+D_0^4}{D_0^4}\right)$. Indeed for every such $\varepsilon$ we choose the unique integer $m\geq m_0$ such that $t_{m+1} \leq \varepsilon \leq t_m.$
	Therefore we have
	\begin{equation*}
		\begin{aligned}
			\mu_X(A(x,R,(1-\varepsilon)R))&\leq \mu_X(A(x,R,(1-t_m)R)) \\ 
			&\leq \left(\frac{D_0^4}{1+D_0^4}\right)^{m+1-m_0} \cdot \mu_X (B(x,R)).
		\end{aligned}
	\end{equation*}
	Using the fact that $m+1 \geq -\log_3 2\varepsilon$ we get
	$$\mu_X(A(x,R,(1-\varepsilon)R))\leq  (2\cdot 3^{m_0})^\beta \cdot \varepsilon^\beta\cdot \mu_X(B(x,R)).$$
	Since $m_0 \leq \log_3\left(\frac{R}{2t_R}\right) + 1$ the thesis follows.
\end{proof}

%%, where $m_0$ is the smallest integer $\geq \log_3(\frac{4R}{t_0})$. 
%Indeed for $t=m_0$ we have $(1-t_{m_0})R \geq R - \frac{t_0}{8}$, therefore
%$$  \mu_X(A(x,R, (1-t_{m_0})R)) \leq \mu_X\left(A\left(x, R, R - \frac{t_0}{8}\right)\right).$$
%Moreover
%\begin{equation*}
%\begin{aligned}
%\mu_X\left(A\left(x, R, R - \frac{t_0}{8}\right)\right) &\leq D_0^4 \left( \mu_X\left(A\left(x, R- \frac{t_0}{8}, R - \frac{3t_0}{8}\right)\right) \right) \leq \\
%& D_0^4 \left( \mu_X(B(x,R)) - \mu_X\left(A\left(x, R, R - \frac{t_0}{8}\right)\right) \right)
%\end{aligned}
%\end{equation*}
%that implies
%$$\mu_X\left(A\left(x, R, R - \frac{t_0}{8}\right)\right) \leq \frac{D_0^4}{1+D_0^4} \cdot V_0,$$
%i.e. the thesis for $m=m_0$. The inductive step and the conclusion follows directly by the same argument used in Proposition 11.5.3 of \cite{HKNS15}.

As a consequence we deduce  that for $D$-doubling GCBA-spaces the  measure  of balls is continuous under the Gromov-Hausdorff convergence, which sharpens Lemma \ref{volume-convergence}:
\begin{cor}
	\label{volume-convergence-uniform}
	Let $X_n$ be a sequence of  geodesic, \textup{GCBA}{$^\kappa$}-spaces which  are $D_0$-doubling up to scale $t_0$ and satisfying $\rho_{\textup{ac}}(X)\geq \rho_0$. 
	Assume that  the $X_n$ converge in the pointed Gromov-Hausdorff sense to some \textup{GCBA}-space $X$ and let  $x_n\in X_n$ be  a sequence of points converging to $x\in X$. Then for any $R\geq 0$ it holds
	$$\mu_X(B(x,R)) = \lim_{n\to +\infty}\mu_{X_n}(B(x_n,R)).$$
\end{cor}

\begin{proof}
	By Remark \ref{rem-doubpack} and Theorem \ref{CAT-packing} the space $X$  is $P_0$-packed at some scale $r_0 \leq \rho_0/3$ for  $P_0, r_0$ only depending on $D_0,t_0,\rho_0$ and $\kappa$. By Theorem \ref{char_pack}, precisely by \eqref{volumeestimate}, the balls of radius $R$ in $X$ have  uniformly bounded volume, that is  
	$$\mu_X(B(x,R))\leq  C(R)$$
	for a universal  function $C(R)$ only depending on $D_0,t_0,\rho_0$ and $R$.\linebreak
	By the above Corollary for all  $R> 0$ and  $\varepsilon > 0$ there exists $\delta > 0$, \linebreak depending only on $D_0,t_0$ and $R$ such that for any $x_n  \in X_n$  it holds \linebreak 
	$\mu_{X_n} (A(x_n,R+\delta,R))\leq \varepsilon.$ 
	The proof then follows directly from $\eqref{convergence-measure}$.
	% and the uniform estimate on the volume of the annulii given by the previous Corollary.
\end{proof}

\vspace{3mm}	
\section{Compactness of packed and doubling GCBA-spaces}\label{sec-compactness}
The aim of this section is to study properties that are stable under Gromov-Hausdorff convergence and the relations between ultralimits and Gromov-Hausdorff convergence. 
\vspace{1mm}

Throughout the section we fix $P_0,r_0,\rho_0 > 0$ with $r_0<\rho_0/3$ and $\kappa\in \mathbb{R}$.\linebreak  
We denote by GCBA$^\kappa_\text{pack}(P_0, r_0; \rho_0)$ the class of complete, geodesic, GCBA$^\kappa$  metric spaces $X$ with $\rho_\text{ac}(X)\geq \rho_0$ and Pack$(3r_0,\frac{r_0}{2}) \leq P_0$. Then we have the following result which is strictly related to Gromov's Precompactness Theorem, see \cite{Gr81}:
\begin{theo}
	\label{compactness-packing}
	The class $\textup{GCBA}^\kappa_{\textup{pack}}(P_0,r_0;\rho_0)$ is closed under ultralimits and compact under pointed Gromov-Hausdorff convergence. 
\end{theo}

\begin{proof}
	Any space $X\in \text{GCBA}^\kappa_{\text{pack}}(P_0,r_0;\rho_0)$ is proper by Theorem \ref{CAT-packing}, geodesic and geodesically complete. Consider any sequence $(X_n,x_n)$ of elements of $\text{GCBA}^\kappa_{\text{pack}}(P_0,r_0;\rho_0)$ and any non-principal ultrafilter $\omega$.
	For any $n$ we have $\rho_\text{cat}(X_n) \geq \min\lbrace \frac{D_\kappa}{2}, \rho_0 \rbrace = \rho'_0 > 0$ from \eqref{ac-radius,cat-radius}.
	Then by Corollary \ref{ultralimit-stability} we have that $X_\omega$ is a complete, locally geodesically complete, locally CAT$(\kappa)$, geodesic metric space with again $\rho_\text{cat}(X_\omega) \geq \rho'_0$. \\
	We want to prove now that Pack$(3r_0, \frac{r_0}{2})\leq P_0$ holds on $X_\omega$.
	We fix a point $y=(y_n) \in X_\omega$: by  Lemma \ref{ultralimitballs} we have   $\overline{B}(y,3r_0) = \omega$-$\lim \overline{B}(y_n,3r_0)$. Let $z^i = (z^i_n)$, $i=1,\ldots,N$ be a $r_0$-separated subset of $\overline{B}(y,3r_0)$, that is  $d(z^i, z^j) > r_0$ for all $i\neq j$. For any couple $i\neq j$ we have $d(z^i_n, z^j_n)>r_0$, $\omega$-a.s. Since there are a finite number of couples,  $d(z^i_n, z^j_n)>r_0$ for any $i \neq j$, $\omega$-a.s. \linebreak 
	Moreover the points $z^i_n$ belong to $\overline{B}(y_n,3r_0)$ for any $i$. So, $\omega$-a.s., there is a $r_0$-separated subset of $\overline{B}(y_n,3r_0)$ of cardinality $N$. Therefore $N\leq P_0$ and in particular Pack$(3r_0, \frac{r_0}{2})\leq P_0$ on $X_\omega$. We can now apply again Theorem \ref{CAT-packing} to conclude that $X_\omega$ is proper, hence a GCBA$^\kappa$ metric space. \\
	To finish the first part of the proof we need to show that $\rho_\text{ac}(X_\omega)\geq \rho_0$. \linebreak
	This is the object of the following:
	
	\begin{prop}
		\label{prop-semicontinuity}
		Let $(X_n,x_n)$ be  \textup{GCBA}$^\kappa$-spaces   converging    to $(X,x)$ with respect to the  pointed Gromov-Hausdorff topology. Then: $$\rho_{\textup{ac}} (X) \geq \limsup_{n \rightarrow \infty} \rho_{\textup{ac}} (X_n)$$
	\end{prop}
	
	\noindent We postpone the proof of this proposition	to end  the proof of Theorem \ref{compactness-packing}. \linebreak
	In order to prove the compactness under pointed Gromov-Hausdorff convergence we take a sequence of spaces $(X_n,x_n)\subseteq \text{GCBA}^\kappa_\text{pack}(P_0,r_0;\rho_0)$ and we fix any non-principal ultrafilter $\omega$. Let $(X_\omega, x_\omega) \in \text{GCBA}^\kappa_\text{pack}(P_0,r_0;\rho_0)$ be the ultralimit. Since the limit is proper we can apply Proposition \ref{Jansen} to find a subsequence $(X_{n_k},x_{n_k})$ that converges in the pointed Gromov-Hausdorff sense to $(X_\omega, x_\omega)$, showing the compactness part of the statement.
\end{proof}
\vspace{2mm}

\noindent {\em Proof	of Proposition \ref{prop-semicontinuity}}. Assume that  $\rho_{\text{ac}}(X_n) \geq \rho_0>0$ for infinitely many $n$. 
Take any non-principal ultrafilter  $\omega$: since by definition $X$ is proper then  by Proposition \ref{Jansen} we have    $X=\omega\text{-}\lim X_n$. 
If $\rho_0 \leq \frac{D_\kappa}{2}$  we have $\rho_{\text{cat}}(X_n) \geq \rho_0$  for all $n$, so by Corollary \ref{ultralimit-stability} we conclude immediately that $\rho_{\text{ac}}(X_\omega) \geq \rho_{\text{cat}}(X_\omega)\geq \rho_0$. \\
Assume now that $\rho_0 > \frac{D_\kappa}{2}$; in particular  as before we deduce  $\rho_{\text{cat}}(X_\omega) = \frac{D_\kappa}{2}$. \\
The strategy is the following: we claim that for any $y=(y_n) \in X_\omega$ and for any point $z=(z_n)$ at distance $<\rho_0$ from $y$ there exists a unique geodesic joining $y$ to $z$. In particular this geodesic must coincide with the ultralimit of the geodesics $[y_n,z_n]$ of length $<\rho_0$. If this is true then for any two points $z=(z_n), w=(w_n)$ of $X_\omega$ at distance $<\rho_0$ from $y$ and any $t\in [0,1]$  we get 
$$d(z_t, w_t) = \omega\text{-}\lim d((z_n)_t, (w_n)_t) \leq \omega\text{-}\lim 2td(z_n,w_n) = 2td(z,w)$$
which implies  that   $\rho_{\text{ac}}(y) \geq \rho_0$ for any $y\in X_\omega$.\\
So suppose our claim is not true: that is 
assume that there exists a point \linebreak $y=(y_n)  \in X_\omega$,  a radius $  \rho_1 \in ( \frac{D_\kappa}{2}, \rho_0)$  such that any point at distance $<\rho_1$ from $y$ is joined to $y$ by a unique geodesic, while for arbitrarily small values $\varepsilon>0$  there exist   two different geodesics $\gamma_\varepsilon, \gamma_\varepsilon'$ joining $y$ to the same point $z_\varepsilon = (z_{\varepsilon,n})$ with  $d(y, z_\varepsilon)=\rho_1 + \varepsilon$.
\\
We consider the points $w_\varepsilon = \gamma_\varepsilon (\rho_1 - \varepsilon)$, $w_\varepsilon' = \gamma_\varepsilon' (\rho_1 - \varepsilon)$ and set $\ell = d(w_\varepsilon, w_\varepsilon')$. We observe we have $\ell\leq 4\varepsilon$ and $\ell > 0$ since the ball of radius $\frac{D_\kappa}{2}$ around $z_\varepsilon$ is CAT$(\kappa)$ by assumption, so uniquely geodesic. Similarly, we   consider the points $u_\varepsilon = \gamma_\varepsilon (\rho_1 + \varepsilon - \frac{D_\kappa}{2})$, $u_\varepsilon' = \gamma_\varepsilon' (\rho_1 + \varepsilon - \frac{D_\kappa}{2})$ and we set $L = d(u_\varepsilon, u_\varepsilon')$.
Our first step is to prove that
\vspace{-3mm}	

\begin{equation}
	\label{L_bigger_l}
	L = d(u_\varepsilon, u_\varepsilon') \geq \frac{D_\kappa}{8} \cdot \frac{\ell}{2\varepsilon} =: \delta.
\end{equation}		
\noindent 	So suppose by contradiction that \eqref{L_bigger_l} does not hold.
First of all we remark that $\delta \leq \frac{D_\kappa}{2}$, since $\ell \leq 4\varepsilon$. 
Then, as the ball $B(z_\varepsilon, \frac{D_\kappa}{2})$ is CAT$(\kappa)$, we can consider the $\kappa$-comparison triangle $\overline{\Delta}^\kappa(\overline{z_\varepsilon}, \overline{u_\varepsilon}, \overline{u_\varepsilon'})$.  
As usual we denote by $\overline{w_\varepsilon}, \overline{w_\varepsilon'}$ the comparison points of $w_\varepsilon$ and $w_\varepsilon'$, respectively.  
By definition the edges of $\overline{\Delta}^\kappa(\overline{z_\varepsilon}, \overline{u_\varepsilon}, \overline{u_\varepsilon'})$ have length $\frac{D_\kappa}{2}, \frac{D_\kappa}{2}, L$.  
We consider another triangle $\Delta(Z,V,V')$ on $M^\kappa_2$ with edges $[Z,V], [Z,V'], [V,V']$ of length respectively $\frac{D_\kappa}{2}, \frac{D_\kappa}{2}, \delta$. We denote by $W, W'$ the points along $[Z,V]$ and $[Z,V']$ at distance $2 \varepsilon$ from $Z$.  Since the contraction map $\varphi^R_r$ towards $Z$ is $\frac{2r}{R}$-Lipschitz and $d(W,Z)=d(W',Z)=2 \varepsilon$ we deduce
$$d(W,W') \leq 2\cdot \frac{  2 \varepsilon}{(D_\kappa /2)} d(V,V') = \frac{8 \varepsilon}{ D_\kappa } \delta = \frac{\ell}{2}.$$
\noindent	Since we are assuming by contradiction that $L < \delta$,  we have by comparison that  $  d(\overline{w_\varepsilon}, \overline{w_\varepsilon'}) < d(W,W')$. So, applying the CAT$(\kappa)$ condition, we obtain 
\vspace{-3mm}	

$$\ell = d(w_\varepsilon, w_\varepsilon') \leq d(\overline{w_\varepsilon}, \overline{w_\varepsilon'}) < d(W,W') \leq \frac{\ell}{2}$$

\vspace{-1mm}	
\noindent	a contradiction. Therefore \eqref{L_bigger_l} holds.\\
Now,   by assumption there exists a unique geodesic from $y$ to any point in $B(y,\rho_1)$. 
Since $d(y, w_\varepsilon) < \rho_1$ by construction,   if $w_\varepsilon = \omega$-$\lim w_{\varepsilon,n} $  then the ultralimit of the geodesics $\gamma_{\varepsilon,n} = [y_n, w_{\varepsilon,n}]$ is the unique geodesic joining $y$ to $w_\varepsilon$, that is $\gamma_\varepsilon =  \omega\text{-}\lim \gamma_{\varepsilon,n} $. 
Analogously,  if $w_\varepsilon' = \omega$-$\lim w'_{\varepsilon,n}$, we have $\gamma'_\varepsilon =  \omega\text{-}\lim \gamma_{\varepsilon,n} $ where  $\gamma'_{\varepsilon,n} = [y_n, w'_{\varepsilon,n}]$.
Applying the contraction property on $X_n$ from   $R= \rho_1 - \varepsilon$ to $r=\rho_1 + \varepsilon -  D_\kappa/2 $ we get
\begin{equation}
	\label{L_smaller}
	\begin{aligned}
		L= d(u_\varepsilon, u_\varepsilon') &= \omega\text{-}\lim \,\, d\left(\gamma_{\varepsilon,n}(\rho_1 + \varepsilon -  D_\kappa/2 ), \, \gamma'_{\varepsilon,n}(\rho_1 + \varepsilon -  D_\kappa/2 )\right) \\
		& \leq \omega\text{-}\lim  \,\,    \frac{ 2 ( \rho_1 + \varepsilon -  D_\kappa/2) }{ \rho_1 - \varepsilon }     \cdot d(w_{\varepsilon,n}, w'_{\varepsilon, n}) \\
		%&=  \frac{ 2 (\rho_1 + \varepsilon - \frac{D_\kappa}{2})}{\rho_1 - \varepsilon} \cdot d(w_\varepsilon, w'_\varepsilon)\\
		&=  \frac{2 (\rho_1 + \varepsilon -  D_\kappa/2 )}{\rho_1 - \varepsilon}\cdot \ell.
	\end{aligned}
\end{equation}
As   $\rho_1 > \frac{D_\kappa}{2}$, combining \eqref{L_bigger_l} and \eqref{L_smaller} gives a contradiction for $\varepsilon \rightarrow 0$. \linebreak
We have therefore proved that $\rho_{\text{ac}} (X) \geq \rho_0$.  This implies the upper semi-continuity of the almost-convexity radius since we can apply the same argument to any subsequence.  \qed
\vspace{1mm}

\begin{obs} {\em
		In particular, for any sequence of metric spaces $X_n$ in \linebreak GCBA$^\kappa_\text{pack}(P_0,r_0;\rho_0)$ and for any non-principal ultrafilter $\omega$ the ultralimit $X_\omega$ is a proper space.
		Notice that,   in general, the ultralimit of a sequence of proper spaces is not proper, even if the spaces are really mild.  
		\linebreak For instance, let $(X_n,x_n) = (\mathbb{R}^n,0)$ and $\omega$ be any non-principal ultrafilter. Then $X_\omega$ is isometric to $\ell^2(\mathbb{R})$,  the spaces of sequences  $\lbrace a_n \rbrace$ of real numbers such that $\sum a_n^2 < +\infty$.  This is   a non-proper space of infinite dimension.
		
	}
\end{obs}

\vspace{2mm}
The compactness of a class of proper metric spaces $\mathcal{C}$ is hard to achieve since  properness and  dimension are in general not stable under limits. \linebreak 

\noindent In the next theorem we precisely characterize the classes of proper, GCBA$^\kappa$, geodesic metric spaces with almost-convexity radius uniformly boun\-ded from below that are precompact and compact under pointed Gromov-Hausdorff convergence. For this, we need the following slight refinement of the packing condition a scale $r_0$. 
 Given a function $P\colon [0,+\infty) \to \mathbb{N}$, we say that a pointed metric space $(X,x)$ belongs to the class 
 $$\textup{GCBA}_{\text{pack}}^\kappa(P(\cdot),r_0;\rho_0)$$ if $X$ is a complete, geodesic, GCBA$^\kappa$  metric space with $\rho_\text{ac}(X)\geq \rho_0$ and Pack$(\overline{B}(x,R),\frac{r_0}{2}) \leq P(R)$ for all $R>0$. This is equivalent to asking that  the packing costant $P$ of Definition \ref{def-packing} possibly depends also on the distance of the center of the balls from $x$.  The same argument used in Theorem \ref{compactness-packing} shows that the class GCBA$_{\text{pack}}^\kappa(P(\cdot),r_0;\rho_0)$ is closed under ultralimits and therefore compact under pointed Gromov-Hausdorff convergence. \linebreak 
 Moreover we have: 

\begin{theo}
	\label{precompactness}
	Let $\mathcal{C}$ be a class of proper, \textup{GCBA}$^\kappa$, geodesic metric spaces $X$ with $\rho_\textup{ac}(X)\geq \rho_0 > 0$. Then $\mathcal{C}$ is precompact under the pointed Gromov-Hausdorff convergence if and only if there exist   $P(\cdot)$ and  $r_0 > 0$ such that   $$\mathcal{C} \subseteq \textup{GCBA}_{\textup{pack}}^\kappa(P(\cdot),r_0;\rho_0).$$ 
	Moreover $\mathcal{C}$ 
	%{\narrowstyle 
	is compact if and only if it is precompact and closed under ultralimits.
	%}
	\linebreak 
	%Moreover $\mathcal{C}$ is precompact (resp.compact) with respect to the pointed Gromov-Hausdorff convergence if and only if it is precompact (resp.compact) with respect to the pointed measured Gromov-Hausdorff convergence.
\end{theo}
\vspace{-5mm}

%\normalstyle
\noindent We stress the ``only if'' part in the above statement: for \textup{GCBA}$^\kappa$ spaces, a uniform packing assumption  (depending only on the distance from the basepoint $x$) at some \emph{fixed} scale is a  {\em necessary} and  sufficient  condition in order to have precompactness 
(we recall that, in the general Gromov's Precompactness Theorem, one needs to have a uniform control of the packing function at every scale in order to achieve precompactness).

\begin{proof}[Proof of Theorem \ref{precompactness}.]
	Let $\mathcal{C}$ be a class of proper, GCBA$^\kappa$, geodesic spaces $X$ with $\rho_\text{ac}(X) \geq \rho_0 > 0$. Let us prove the first equivalence stated  in \ref{precompactness}. \linebreak
	So assume that it is precompact in the pointed Gromov-Hausdorff sense, i.e. the closure $\overline{\mathcal{C}}$ is compact under pointed Gromov-Hausdorff convergence. \linebreak
	Suppose $\mathcal{C}$ is not contained in $\text{GCBA}^\kappa_\text{pack}(P(\cdot),r_0;\rho_0)$ for any choice of $P(\cdot)$ and $r_0$. Hence there exists $r_0 < \frac{\rho_0}{3}$ and $R>0$ such that for any $n$ there is a space $(X_n,x_n) \in \mathcal{C}$ with a set of $r_0$-separated points inside $\overline{B}(x_n, R)$ of cardinality at least $n$. 
	%We consider the sequence $(X_n, x_n)\in \mathcal{C}$. 
	By assumption there exists a subsequence,  denoted again $(X_n,x_n)$, converging in the pointed Gromov-Hausdorff sense to $(X,x)$. The space $X$ is proper,  see Section \ref{sec-preliminaries}.  Fix now any non-principal ultrafilter $\omega$. Then $(X_\omega, x_\omega)$ is isometric to $(X,x)$ by Proposition \ref{Jansen}, and in particular it is proper.
	We are going to prove that inside $\overline{B}(x_\omega,R)$ there are infinitely many points that are at distance at least $r_0$ one from the other: therefore $X_\omega$ cannot be proper and this is a contradiction. 
	For any $n$ we denote the set of $r_0$-separated points of cardinality $n$ inside $\overline{B}(x_n, R)$ by $\{ z_n^1, \ldots, z_n^n \}$. Then,  for any fixed $k\in \mathbb{N}$, we consider the admissible sequence $z^k = (z_n^k) \in X_\omega$ (notice that $z^k_n$ is defined only for $n\geq k$, but this suffices to define a point $z^k $ in the ultralimit).  Clearly $ z ^k \in \overline{B}(x_\omega, R)$ for all $k$. \linebreak
	Moreover if $k\neq l$ then $d(z_n^k, z_n^l) > r_0$ for all $n$, hence $d(z^k,z^l) \geq r_0$.  \linebreak
	This shows that $\mathcal{C}$ is a subclass of $\text{GCBA}^\kappa_\text{pack}(P(\cdot),r_0;\rho_0)$ for some $P(\cdot)$ and $r_0$.
	Viceversa if $\mathcal{C}\subseteq \text{GCBA}^\kappa_\text{pack}(P(\cdot),r_0;\rho_0)$ then its closure $\overline{\mathcal{C}}$ is contained in the compact space $\text{GCBA}^\kappa_\text{pack}(P(\cdot),r_0;\rho_0)$ by the analogue of Theorem \ref{compactness-packing}, so $\overline{\mathcal{C}}$ is compact.
	
	\noindent Let us show now the second equivalence.	Suppose that $\mathcal{C}$ is precompact and closed under ultralimits. Applying the same proof of the second part of Theorem \ref{compactness-packing} we get that $\mathcal{C}$ is compact under pointed Gromov-Hausdorff convergence. 
	Viceversa if $\mathcal{C}$ is compact under Gromov-Hausdorff convergence then it is contained in $\text{GCBA}^\kappa_\text{pack}(P(\cdot),r_0;\rho_0)$ for some $P(\cdot),r_0$. In particular for any non-principal ultrafilter $\omega$ and any sequence of spaces $(X_n,x_n) \in \mathcal{C}$ we have that $X_\omega$ is a proper metric space. By Proposition \ref{Jansen} there exists a subsequence that converges in the pointed Gromov-Hausdorff sense to $X_\omega$, hence $X_\omega \in \mathcal{C}$ since $\mathcal{C}$ is compact.
\end{proof}

As a consequence of Theorem \ref{precompactness} and of the estimates on volumes and packing proved in Sections \ref{sec-volume} \& \ref{sec-packingcovering}, we deduce  that the dimension is  {\em almost} stable under pointed Gromov-Hausdorff convergence, in the following sense:

\begin{prop}\label{cor-dim}
	Let $(X_n,x_n)$ be a sequence   of \textup{GCBA}$^\kappa$-spaces with almost convexity radius $\rho_{\textup{ac}} (X_n) \!\geq \!\rho_0\!>\!0$  converging    to $(X,x)$  in the pointed Gromov-Hausdorff sense. 
	Let   $ X_n ^{\textup{max}}$ be the maximal dimensional subspace of $X$. 
	Then  
	$$\textup{dim}(X) \leq \liminf_{n\to +\infty} \textup{dim}(X_n)  $$
	and the equality $\textup{dim}(X) = \lim_{n\to +\infty} \textup{dim}(X_n) $ holds if and only if the distance $d(x_n,  X_n ^{\textup{max}})$ stays uniformly bounded when $n\rightarrow \infty$.
\end{prop}

\begin{proof}
	As the spaces $(X_n,x_n)$   converge to $(X,x)$, they form a  precompact family and  so they   belong  to  $\text{GCBA}^\kappa_{\text{pack}}(P(\cdot), r_0, \rho_0)$, for some $P(\cdot)$ and $r_0$, by Theorem \ref{precompactness}.  
	Let us   first  show that we always have
	\vspace{-3mm}
	
	\begin{equation}\label{naganoineq}
		\textup{dim}(X) \leq \liminf_{n \to +\infty} \textup{dim}(X_n)
	\end{equation}
	% (which generalizes Lemma   2.1  of \cite{Nag18} to {\em locally} CAT$(\kappa)$ spaces).\\
	Actually consider a subsequence, we we still denote $(X_n)$, whose dimensions equal  the limit inferior, denoted $d_0$.
	Now suppose that there exists a point $y\in X$ with dim$(y) = d > d_0$.
	We  may assume that $y$ is $d$-regular, since Reg$^d(X)$ is dense in $X^d$. 
	The point $y$ is the limit of a sequence of points $y_n \in X_n$ and for any $r>0$ the volume of the ball $\overline{B}(y,r)$ is bigger than or equal to the limit of the volumes of the balls $\overline{B}(y_n,\frac{r}{2})$, by \eqref{convergence-measure}. \linebreak
	By Theorem \ref{boundbelowvolume} we have for all $n$:
	\vspace{-3mm}
	
	$$\mu_X\left(\overline{B}\left(y_n,\frac{r}{2}\right)\right) \geq c_{d_0} \cdot \left(\frac{r}{2}\right)^{d_0}$$
	
	\noindent  where $c_{d_0}$ is a constant depending only on $d_0$. Moreover, since $y$ is $d$-regular, then for any  $r$ small enough  the ball $\overline{B}(y,r)$ contains only $d$-dimensional points. We conclude by \eqref{k-inequality1} \& \eqref{k-inequality2} that 
	\vspace{-3mm}
	
	$$\mu_X(\overline{B}(y,r)) \leq C\cdot r^d,$$
	where $C$ is a constant depending only on $y$ and not on $r$.
	Therefore, as $d_0 < d$, we   have a contradiction if  $r$ is small enough, and (\ref{naganoineq}) is proved.\\
	Assume now that $d(x_n, X_n^{\text{max}})<D$ for all $n$. Since the almost convexity radius is bounded below by $\rho_0$ both for $X_n$ and for $X$, also the CAT$(\kappa)$-radius is bounded below by (\ref{ac-radius,cat-radius}). So we can consider   tiny balls $B(y_n, r_0)$ centered at regular points  $y_n$ of maximal dimension,  all with the same radius $r_0$, such that the closed ball $\overline B(y_n, 10r_0)$ converge to some ball $\overline B(y, 10r_0)$ of $X$ and  satisfy the condition $\textup{Pack}(P_0, \frac{r_0}{2})$ for some constant $P_0$ for all $n$,  by  Theorem  \ref{CAT-packing}. 
	We are then in the standard setting of convergence,   which implies by Lemma \ref{regularity-convergence} that  
	$$ \textup{dim}(X)  \geq \textup{dim} (y) \geq \limsup_{n \rightarrow \infty}  \textup{dim} (y_n) = \limsup_{n \rightarrow \infty}  \textup{dim}(X_n).$$ 
	Conversely,   assuming   $\textup{dim}(X) = \lim_{n\to +\infty} \textup{dim}(X_n)$, then in particular  \linebreak$\textup{dim}(X_n)$ is constant for $n \gg 0$ and equal to $d_0=\textup{dim}(X)$. Consider  a regular point  $y=(y_n) \in X$ of dimension $d_0$: then  the points $y_n$  are admissible by definition (that is, $d(x_n,y_n)$ stays uniformly bounded); moreover,   we can choose as before uniformly packed tiny balls with $B(y_n,10r_0)$ converging to $B(y,10r_0)$, so the points $y_n$ belong to $X_n^{\text{max}}$, again by Lemma  \ref{regularity-convergence} (b).
\end{proof}
%\vspace{-3mm}

\begin{ex} ${}$ Let $(X, x) \in \text{GCBA}^\kappa_\text{pack}(P_0,r_0,\rho_0)$ be any space. We consider the space $Y$ obtained by gluing the half-line $[0,+\infty)$ to $X$ at the point $x$. Clearly $Y$ belongs to $\text{GCBA}^\kappa_\text{pack}(P_0',r_0',\rho_0')$. The pointed Gromov-Hausdorff limit of the sequence $(Y, n)$, where $n\in [0,+\infty)$, is the real line. This is an example where the maximal dimension part escapes to infinity and the dimension is not preserved.
\end{ex}

We are going now to explore   some variations of Theorem \ref{precompactness}.\\
We fix constants $\kappa\in \mathbb{R}$ and $P_0, r_0, V_0,  R_0, D_0, t_0,  \rho_0, n_0 > 0$, with  $r_0 \leq \rho_0/3$, and consider the following  classes  of complete, geodesic GCBA$^\kappa$  spaces $X$:
\vspace{1mm}

\noindent -- the class $\text{GCBA}^\kappa_\text{pack}(P_0,r_0;\rho_0, n_0)$ of spaces   which  are $P_0$-packed at scale \linebreak 
${}$\hspace{2mm} $r_0$, with almost-convexity radius   $\rho_\text{ac}(X)\geq \rho_0$ and dimension $\leq n_0$; 
\vspace{1mm}

\noindent -- the class $\text{GCBA}^\kappa_\text{pack}(P_0,r_0;\rho_0, n_0^{^{\text{pure}}})$ of spaces $P_0$-packed at scale $r_0$, \linebreak 
${}$\hspace{2mm} with almost-convexity radius   $\rho_\text{ac}(X) \!\geq \!\rho_0$  and of {\em pure dimension} $n_0$; 
\vspace{1mm}

\noindent -- the classes $\text{GCBA}^\kappa_\text{vol}(V_0, R_0; \rho_0, n_0)$,   $\text{GCBA}^\kappa_\text{vol}(V_0, R_0; \rho_0, n_0^{^{\text{pure}}})$ of those \linebreak
${}$\hspace{1mm} satisfying   $\mu_X(B(x,R_0))\leq V_0$, $\rho_\text{ac}(X) \geq \rho_0$ and  which have,   respectively,  \linebreak
${}$ \hspace{1mm}  dimension $\leq n_0$ and pure dimension $n_0$;  
\vspace{1mm}

\noindent --  the class  $\text{GCBA}^\kappa_\text{doub}(D_0,t_0;\rho_0)$   of spaces   $D_0$-doubled up to scale $t_0$, \linebreak
${}$\hspace{2mm}  with $\rho_\text{ac}(X)\geq \rho_0$. \\
Then:

\begin{cor}
	\label{compactness-pointed}
	All the above   classes  are compact with respect to the pointed Gromov-Hausdorff convergence. 
\end{cor}
\begin{proof}
	By Theorem \ref{char_pack}   and Corollary \ref{doubling},  the above  are all  subclasses of  \linebreak $\text{GCBA}^\kappa_\text{pack}(P_0,r_0;\rho_0)$, for suitable $P_0$ and  $r_0$. By the compactness Theorem  \ref{precompactness}, the proof then reduces to show that the additional conditions on the dimension, on the measure of balls of given radius or on the doubling constant are stable under Gromov-Hausdorff limits.
	By Lemma \ref{volume-convergence}, if a sequence $X_n$ in  $\text{GCBA}^\kappa_\text{vol}(V_0,R_0;\rho_0, n_0)$  converges to $X$, then 
	$\mu_X(B(y,R_0)) \leq V_0$ for any $y\in X$. 
	On the other hand, from   Corollary \ref{volume-convergence-uniform} it follows  that the doubling condition is preserved to the limit. 
	The stability of the dimension is proved in  Proposition \ref{cor-dim}. To conclude we need  to  show that pure-dimensionality is stable under Gromov-Hausdorff limits: this is the object of the  Proposition which follows.
	%	We fix a sequence $(X_n,x_n)$ of pointed metric spaces contained in $\text{GCBA}^\kappa_\text{vol}(V_0,R_0;\rho_0,n_0)$. By Theorem \ref{char_pack} we know that there exist $P_0,r_0$ such that any $(X_n,x_n)$ is contained in $\text{GCBA}^\kappa_\text{pack}(P_0,r_0;\rho_0)$. Hence by Theorem \ref{precompactness} there exists a subsequence, denoted again by $(X_n,x_n)$, that converges in the pointed Gromov-Hausdorff sense to a space $(X,x) \in \text{GCBA}^\kappa_\text{pack}(P_0,r_0;\rho_0)$. By Lemma \ref{volume-convergence} we  get $\mu_X(B(y,R_0)) \leq V_0$ for any $y\in X$. On the other hand, {\color{blue}by Corollary \ref{cor-dim}, we know  that the dimension of $X$ is not greater than $n_0$.} 
\end{proof}

\begin{prop}
	\label{pure-dimensional-stability}
	Let $(X_n,x_n)$ be a sequence   of \textup{GCBA}$^\kappa$-spaces with almost convexity radius $\rho_{\textup{ac}} (X_n) \geq \rho_0>0$  converging    to $(X,x)$  in the pointed Gromov-Hausdorff sense. 
	Assume that  $X_n$ is pure-dimensional for all $n$: then    $X$ is pure-dimensional of dimension
	$\textup{dim}(X) = \lim_{n\to +\infty}\textup{dim}(X_n).$
\end{prop}
\begin{proof}
	The spaces $(X_n,x_n)$  form a  precompact family and  so,  by Theorem \ref{precompactness}, they   belong  to  $\textup{GCBA}^\kappa_{\textup{pack}}(P(\cdot), r_0, \rho_0)$, for suitable $P(\cdot)$ and $r_0$.
	Then,  using the maps $\Psi_{x_n}$ of Proposition \ref{PSI} as in the first part of Theorem \ref{char_pack}, the numbers dim$(X_n)$ belong to the finite set $[0,n_0]$. Suppose to have two   integers $d_1 \neq d_2$ and two infinite subsequences $X_{n_{i_1}}$, $X_{n_{i_2}}$ such that dim$(X_{n_{i_1}}) = d_1$ for any $i_1$ and dim$(X_{n_{i_2}}) = d_2$ for any $i_2$. 
	%Each subsequence converges in the pointed Gromov-Hausdorff sense to $(X,x)$. 
	We consider the  sequences $x_{n_{i_1}}$ and $x_{n_{i_2}}$: for any $r>0$ we have by \eqref{convergence-measure} and Lemma \ref{volume-convergence}
	$$\limsup_{n\to +\infty}\mu_{X_n}\left(B\left(x_n,\frac{r}{2}\right)\right)\leq\mu_X(B(x,r))\leq \limsup_{n\to +\infty}\mu_{X_n}(B(x_n,r)).$$
%	the volumes of the balls of radius $r$ around these points converge to the volume of the ball of radius $r$ around $x$, by {\color{blue} the analogue of} Corollary \ref{volume-convergence-uniform} and Corollary \ref{doubling}. 
	By \eqref{k-inequality1} and Theorem \ref{boundbelowvolume} we have
	$$\frac{1}{C}\left(\frac{r}{2}\right)^{d_1}\leq \mu_{X_{n_{i_1}}}\left(B\left(x_{n_{i_1}}, \frac{r}{2}\right)\right) \leq \mu_{X_{n_{i_1}}}(B(x_{n_{i_1}}, r)) \leq C r^{d_1},$$
	$$\frac{1}{C}\left(\frac{r}{2}\right)^{d_2}\leq \mu_{X_{n_{i_2}}}\left(B\left(x_{n_{i_2}}, \frac{r}{2}\right)\right) \leq \mu_{X_{n_{i_2}}}(B(x_{n_{i_2}}, r)) \leq C r^{d_2},$$
	where $C$ is a constant depending only on $P(\cdot)$ and $r_0$.
	Since this is true for any arbitrarily small $r$  we deduce that $d_1 = d_2$. Therefore $\lim_{n\to +\infty} \text{dim}(X_n)$ exists and we denote it by $d_0$.
	We again apply the same estimate as before to conclude that for any $y\in X$ and for any small $r>0$ we have
	$$\frac{1}{C}r^{d_0}\leq \mu_{X}(B(y, r)) \leq C r^{d_0},$$
	where $C$ is a constant depending on $P(\cdot)$ and $r_0$.
	Therefore the dimension of $y$ is $d_0$, which  concludes the proof.
\end{proof}

Finally we can specialize   these theorems  to subclasses of  compact  spaces. Clearly the subclasses of the above classes  made of spaces   with diameter  less than or equal to some constant $\Delta$ will be  compact with respect to the usual Gromov-Hausdorff distance.
We state here just two particularly interesting cases, which are reminiscent of  the classical finiteness theorems of Riemannian geometry. 
%Weinstein, Chegeer, Yamaguchi  {\color{red},  Grove-Petersen-Wu etc
Consider the  classes:
$$\text{GCBA}^\kappa_\text{vol}(V_0;\rho_0, n_0^=), \hspace{5mm}\text{GCBA}^\kappa_\text{vol}(V_0;\rho_0, n_0^{^{\text{pure}}})$$  
of complete, geodesic  $\text{GCBA}^\kappa$  with total measure $\mu(X) \leq V_0$,  almost convexity radius $\rho_\text{ac}(X) \geq \rho_0$ and which are, respectively, precisely $n_0$-dimensional  and purely $n_0$-dimensional.

\begin{cor}
	\label{compactness-unpointed}
	${}$ \hspace{-3mm}	The classes $\textup{GCBA}^\kappa_\textup{vol}(V_0;\rho_0, n_0^=)$    and $\textup{GCBA}^\kappa_\textup{vol}(V_0;\rho_0, n_0^{^{\textup{pure}}})$  are compact under Gromov-Hausdorff convergence and contain only finitely many homotopy types. 
\end{cor}

\begin{proof}	
	%It is evident that   $\text{GCBA}^\kappa_\text{pack}(P_0,r_0;\rho_0, \Delta_0)$ is a closed subset of \linebreak $\text{GCBA}^\kappa_\text{pack}(P_0,r_0;\rho_0)$. Therefore, by Theorem \ref{compactness-packing},  it is  compact for  the  pointed Gromov-Hausdorff topology, hence for the usual Gromov-Hausdorff disance (as it is made of compact spaces with uniformly bounded diameter).\\
	First we  show   that the diameter is uniformly bounded in both classes. 
	Actually consider   $X \in \text{GCBA}^\kappa_\text{vol}(V_0;\rho_0, n_0 ^=)$
	and   take any two points $y,y' \in X$  
	such that $d(y,y') = \Delta > \rho := \min\lbrace\rho_0, 2 \rbrace$. 
	%Our aim is to bound $\Delta$ from above uniformly. 
	Let $\gamma$ be a geodesic joining $y$ to $y'$. Along $\gamma$ we take points at distance $\rho$ one from the other: they are at least $\frac{\Delta}{\rho} - 1$ and the balls of radius $\frac{\rho}{2}$ around these points are disjoint. Then by Theorem \ref{boundbelowvolume} we get
	\vspace{-3mm}	
	
	$$V_0 \geq \mu_{X}(X)\geq \frac{c_{n_0}}{2^{n_0}}\rho^{n_0} \left(\frac{\Delta}{\rho}-1\right)$$
	so  the diameter of $X$  is bounded from above in terms of  $n_0, \rho_0$ and $V_0$ only. \\
	Let $R_0$ such un upper bound. Then these classes are included in \linebreak  $\text{GCBA}^\kappa_\text{vol}(V_0, R_0;\rho_0, n_0)$, whose compactness    we have just proved. The conclusion follows from Propositions \ref{cor-dim} and \ref{pure-dimensional-stability}.
	% Moreover, any limit of spaces in   $\text{GCBA}^\kappa_\text{vol}(V_0;\rho_0, n_0)$ has again dimension equal to $n_0$, by Proposition \ref{cor-dim}, and volume less than $V_0$, by the same proof as in Corollary \ref{compactness-volume}, so this class is actually compact.
	
	\noindent 	Finally  notice that any element of both classes has local geometric contractibility function LGC$(r)=r$ for $r\leq \rho_0$  (see \cite{Pet90} for the definition). Moreover the covering dimension of any space in both classes coincides with the Hausdorff dimension, so it   is uniformly bounded from above. We can then apply Corollary B of \cite{Pet90} to conclude that there are only finitely many homotopy types inside any of the two classes.
	%{\color{red} non serve richiamare anche l'erratum di GPW?}{\color{blue} Non occorre perchè il teorema di Petersen è corretto, infatti è della forma: se una classe che soddisfa LGC + upper bound sulla dimensione di Lebesgue è precompatta allora contiene un numero finito di tipi di omotopia. L'erratum GPW2 riguarda l'articolo GPW nel quale affermano che le ipotesi LGC + upper bound sulla dimensione di Lebesgue bastano per avere precompattezza, chiaramente falso. Per noi funziona perchè sappiamo già che la classe è precompatta.}
\end{proof}

\enlargethispage{6mm}

\vspace{3mm}
\section{Examples: $M^\kappa$-complexes}\label{sec-simplicial}
%We want to provide examples of classes of GCBA-spaces satisfying a uniform packing condition. 
Beyond Riemannian manifolds with uniform upper  bounds on the sectional curvature and injectivity radius bounded below, an important class of GCBA$^\kappa$ spaces is provided by $M^\kappa$-complexes with {\em bounded geometry}, in a sense we are going to explain. We will prove that the metric spaces in this class are uniformly packed and we will show that this class is compact under pointed Gromov-Hausdorff convergence. Other finiteness results will be presented.

%We start with Riemannian manifolds.
%
%\begin{ex}
%	Let $\mathcal{M}(n,a,b,\iota)$ be the class of complete Riemannian manifolds of dimension $n$ with all sectional curvatures belonging to $[a,b]$ and injectivity radius $\geq \iota$. Then there exist $P_0(n,a,b,\iota), r_0(n,a,b,\iota)$ such that for any $X\in \mathcal{M}(n,a,b,\iota)$ it holds $\textup{Pack}(3r_0,\frac{r_0}{2})\leq P_0$.
%	Indeed any $X \in \mathcal{M}(n,a,b,\iota)$ is a space satisfying the assumptions of Theorem \ref{char_pack} with $\kappa=b$. Moreover dim$(X)=n$. By Bishop-Gromov Theorem we have $\mu_X(B(x,r)) \leq \mu_{M_a^n}(B(o,r))$ for any $x\in X$ and any $0 \leq r \leq \iota$. Here $M_a^n$ is the $n$-dimensional space-form of constant curvature $a$ and $o$ is any point of $M_a^n$. The desired estimate follows from Theorem \ref{char_pack}.
%	
%	More generally we can consider the class of $n$-dimensional Riemannian manifolds with sectional curvature bounded above by $\kappa$, injectivity radius bounded below by $\rho_0$ and volumes of balls of radius $R_0$ bounded above by $V_0$. This class satisfies a uniform packing condition by Theorem \ref{char_pack} and contains the class $\mathcal{M}(n,a,b,\iota)$ for good choices of the parameters.
%\end{ex}
%
%Another interesting example}

\subsection{Geometry of $M^\kappa$-complexes}
\label{subsec-basics}

First of all we recall briefly the definitions and the properties of the class of simplicial complexes we are interested in. 
A {\em $\kappa$-simplex}  $S$ is the convex set generated by $n+1$ points $v_0,\ldots, v_{n}$ of $M^\kappa_n$ in general position, where $M^\kappa_n$ is the unique $n$-dimensional space-form with constant sectional curvature $\kappa$. If $\kappa >0$ the points $v_0,\ldots,v_{n}$ are required to belong to an open emisphere.
We say that $S$ has dimension $n$. Each $v_i$ is called a {\em vertex}.
A {\em d-dimensional face} $T$ of $S$ is the convex hull of a   subset   $\lbrace v_{i_0},\ldots,v_{i_d}\rbrace$ of   $(d+1)$ vertices. The {\em interior} of $S$, denoted  $\dot{S}$,  is defined as $S$ minus the union of its lower dimensional faces; the {\em boundary} $\partial S$   is    the union of its codimension $1$ faces. \\
Let $\Lambda$ be any set and $E = \bigsqcup_{\lambda \in \Lambda} S_\lambda$, where any $S_\lambda$ is a $\kappa$-simplex. Let $\sim$ be an equivalence relation on $E$ satisfying:

\begin{itemize}
	\item[(i)] for any $\lambda \in \Lambda$ the projection map $ p  \colon S_\lambda \to E/_\sim$ is injective;
		\vspace{1mm}
	\item[(ii)] for any $\lambda, \lambda' \in \Lambda$ such that $p (S_\lambda)\cap p (S_{\lambda'}) \neq \emptyset$ there exists an isometry $h_{\lambda,\lambda'}$ from a face $T \subset S_\lambda$ onto a face $T'\subset S_{\lambda'}$ such that   $p  (x) = p (x')$,  for $x \in S_\lambda$ and $x' \in S_{\lambda'}$,   if and only if $x' = h_{\lambda,\lambda'}(x)$.
\end{itemize}
	\vspace{1mm}

\noindent The quotient space $K = E/_\sim$ is called a $M^\kappa$-simplicial complex or simply {\em $M^\kappa$-complex};  the set $E$ is  the {\em total space}.  A subset $S\subset K$ is called an $m$-simplex of $K$ if it is the image under  $p$ of an $m$-dimensional face of some $S_\lambda$; its {\em interior}  and its {\em boundary} are, respectively,   the image under $p$ of the interior and the boundary of $S_\lambda$.
The {\em support of a point} $x\in K$, denoted  supp$(x)$, is   the unique simplex containing $x$ in its interior (notice that  supp$(v) = v$ when $v$ is a vertex).
The \emph{open star} around a vertex $v$ is the union of the interior of all simplices having $v$ as a vertex.
\\
Metrically $K$ is equipped with the quotient pseudometric. By Lemma I.7.5 of \cite{BH09} the pseudometric can be expressed using strings.
A {\em $m$-string} in $K$ from $x$ to $y$ is a sequence $\Sigma = (x_0,\ldots, x_m)$ of points of $K$ such that $x=x_0$, $y=x_m$ and for each $i=0,\ldots,m-1$ there exists a simplex $S_i$ containing $x_i$ and $x_{i+1}$. 
Moreover a  $m$-string $\Sigma = (x_0,\ldots,x_m)$ from $x$ to $y$ is {\em taut} if
\begin{itemize}
	\item there is no simplex containing $\lbrace x_{i-1}, x_i, x_{i+1}\rbrace $;
	\item if $x_{i-1},x_i \in S_i$ and $x_i, x_{i+1} \in S_{i+1}$ then the concatenation of the segments $[x_{i-1}, x_i]$ and $[x_i, x_{i+1}]$ is  geodesic  in the subcomplex  $S_i \cup S_{i+1}$.
\end{itemize}

\noindent The {\em length} of $\Sigma$ is defined as:
\vspace{-3mm}
$$\ell(\Sigma)=\sum_{i=0}^{m-1}d_{S_i}(x_i, x_{i+1})$$

\vspace{-1mm}
\noindent where $d_S$ denotes the standard $M^\kappa$-metric on a geodesic simplex $S$ of $M^\kappa$. \linebreak
Then any   string can be identified to a path in $K$ and the natural quotient pseudometric on $K$ coincides with the following (\cite{BH09}, Lemma I.7.21):
$$d_K(x,y)=\inf\lbrace \ell(\Sigma) \text{ s.t. } \Sigma \text{ is a taut string from } x \text{ to }y\rbrace.$$

%\begin{lemma}
%	\label{taut-string-distance}
%	Let $K$ be a connected $M^\kappa$-complex with valency at most $N$ and size bounded by $R$. Then for all $x,y \in K$ it holds 
%	$$d(x,y) = \inf\lbrace \ell(\Sigma) \text{ s.t. } \Sigma \text{ is a taut string from } x \text{ to } y \rbrace.$$
%\end{lemma}
%\begin{proof}
%	The proof follows the line of Proposition I.7.24 of \cite{BH09}. First of all we notice that for all integers $m$ there are at most a finite set of connected subcomplexes of $K$ containing $x$ and $y$ made of at most $m$ simplices. This is clear from the fact that if two simplices share a face then they share a vertex and the number of simplices sharing a vertex is bounded above by $N$. Then the proof of Lemma I.7.27 of \cite{BH09} can be applied to our case. This, together with Lemma I.7.21 of \cite{BH09}, gives us the thesis. 
%\end{proof}

\vspace{-1mm}
\noindent Moreover for any $x\in K$ one can define  the number 
\begin{equation}
	\label{epsilon}
	%\varepsilon(x) = \inf_{x\in S \text{ simplex of } K}\inf_{x\notin T \text{ face of } S} d_S(x,T)
	\varepsilon(x) = 
	\inf_{\scriptsize \begin{array}{c}  S \text{ simplex of } K \\  x\in S \end{array}}   
	\left( \inf_{\scriptsize \begin{array}{c}   T \text{ face of } S \\    x\notin T \end{array}      } d_S(x,T)    \right)
\end{equation}
which has  the following fundamental property:
\begin{lemma}[Lemma I.7.9 and Corollary I.7.10 of \cite{BH09}]
	\label{epsilon-simplex}${}$\\
	If $\varepsilon(x)> 0$ for any $x$ and $K$ is connected then $d_K$ is a metric and $(K,d_K)$ is a length space. Moreover if $y\in K$ satisfies $d_K(x,y) < \varepsilon(x)$ then any simplex $S$ containing $y$ contains also $x$ and $d_K(x,y) = d_S(x,y)$.
\end{lemma} 

\noindent For any vertex $v\in K$ it is possible to define the {\em link}  Lk$(v,K)$ of $K$ at $v$ as follows.
We fix any $\lambda \in \Lambda$ such that $v =p(v_\lambda)$, where $v_\lambda$ is a vertex of $S_\lambda$.\linebreak  The set of unit vectors $w$ of $T_{v_\lambda}M_n^\kappa$ such that the geodesic starting in direction $w$ stays inside $S_\lambda$ for a small time is a geodesic simplex of $M_{n-1}^1 = \mathbb{S}^{n-1}$, \linebreak 
denoted   Lk$(v_\lambda, S_\lambda)$. Consider the equivalence relation on the disjoint union 
$\bigsqcup_{ p(S_\lambda) \ni v} S_\lambda $ given by  $w_\lambda \sim w_{\lambda'}$ if and only if $p (S_\lambda)\cap p (S_{\lambda'}) \neq \emptyset$ and \linebreak $(dh_{\lambda,\lambda'})_{v_\lambda}(w_\lambda) = w_\lambda'$: 
the link Lk$(v,K)$ is  the quotient space under this equivalence relation. It is clearly a $M^1$-complex.
\vspace{2mm}

We  introduce now the class of simplicial complexes we are interested in. \linebreak 
We say that $K$ has {\em valency} at most $N$ if for all $v\in K$ the number of simplices having $v$ as a vertex is bounded above by $N$. Notice that if the valency is at most $N$ then the maximal dimension of a simplex of $K$ is at most $N$ too. \linebreak
We say that a simplex $S$ has {\em size} bounded by $R > 0$ if it contains a ball of radius $\frac{1}{R}$ and it is contained in a ball of radius $R$;  accordingly we say the simplicial complex $K$ has size bounded by $R$ if all the simplices $S_\lambda$ defining $K$ have size bounded by $R$. The bound on the size avoids to have too thin simplices: it should be thought as a quantitative non-collapsing condition, as follows from the next results. Indeed the first one affirms that a bound on the size of a simplex gives uniform bounds on the size of any of its faces: for example its $1$-dimensional faces are not too short (and not too long).
\begin{lemma}
	\label{size-faces}
	Let $S$ be a $M^\kappa$-simplex of dimension $n$ and size bounded by $R$. Then any face of $S$ of dimension $d$ has size bounded by $2^{n-d}R$.
\end{lemma}
\begin{proof}
	We prove the lemma by induction on the dimension $n$. If $n=0, 1$ there is nothing to prove. 
	Assume now that the bounds hold  for all faces of  $M^\kappa$-simplices of dimension $\leq n-1$ and consider   a $n$-dimensional $M^\kappa$-simplex $S=\text{Conv}(v_0,\ldots,v_n)$ of size bounded by $R$.
	Let  $S'=\text{Conv}(v_0,\ldots, v_{n-1})$ be the face of $S$ opposite to $v_n$  and identify  $M^\kappa_{n-1}$   with the $\kappa$-model space containing $S'$.  
	It is clear that   $S'$   is contained in a ball  $B_{M^\kappa_{n-1}}(x,2R)$ of  $M^\kappa_{n-1}$.  
	On the other hand  let   $B_{M^\kappa_{n}}(x,\frac{1}{R})$ be the ball of  $M^\kappa_{n}$ which is contained in $S$. 
	Call $\psi: S \rightarrow S'$  the map sending every point $z$ of $S$ to the intersection of the extension of the geodesic $[v_n,z]$ after $z$  with $S'$ and let $y= \psi(x)$;  
	%meets $\text{Conv}(v_0,\ldots, v_{n-1})$ in a point $y=\varpsi(x)$.
	moreover let $\varphi$ be  the contraction map centered at $v_n$ sending $y$ to $x$.
	Notice that $\psi \circ \varphi (z)=z$ for all $z \in S'$.
	%As $d_S(v_n, x) \leq d_S(v_n, y)$, 
	The map $\varphi$  is $2$-Lipschitz, so any point of $B_{M^\kappa_{n}}(y,\frac{1}{2R})$ is sent to $B_{M^\kappa_{n}}(x,\frac{1}{R})$ under $\varphi$. 
	Therefore  
	$$ \textstyle B_{M^\kappa_{n-1}} \left(y,\frac{1}{2R} \right)  =  B \left(y,\frac{1}{2R} \right) \cap  M^\kappa_{n-1} \subseteq \psi  \left( B_{M^\kappa_{n}}(x,\frac{1}{R}) \right) \subseteq S' $$ which proves the induction step.
	% The extension of the geodesic $[v_n,x]$ after $x$ meets $\text{Conv}(v_0,\ldots, v_{n-1})$ in a point $y=\varpsi(x)$.  
	%	Moreover, the contraction map $\varphi$ centered at $v_n$ that sends $y$ to $x$ is at most $2$-Lipschitz, so any point of $B(y,\frac{1}{2R})$ is sent to $B(x,\frac{1}{R})$. 
	%	By identifying  $M^\kappa_{n-1}$   with the $\kappa$-model space containing $v_0,\ldots, v_{n-1}$, we have
	%	 
	%	Moreover, Arguing as before we get that any point of $B(y,\frac{1}{2R}) \cap M^\kappa_{n-1}$, where $M^\kappa_{n-1}$ is identified with the $\kappa$-model space containing $v_0,\ldots, v_{n-1}$, belongs to $\text{Conv}(v_0,\ldots,v_{n-1})$. This concludes the induction step.
\end{proof}
In the second result we prove the non-collapsing property: limit of $n$-dimensional simplices with uniform bound on the size is again $n$-dimensional (and satisfies the same bound on the size).
\begin{prop}
	\label{simplex-compactness}
	The class of $n$-dimensional $M^\kappa$-simplices of size bounded by $R$  and having a fixed point $o$ as a vertex is compact under the Hausdorff distance on $M^\kappa_n$. Moreover, under this convergence, any face of the limit space is limit of faces of the simplices in the sequence. Finally the same class is closed under ultralimits.
\end{prop}
\begin{proof}
	We take a sequence of simplices $S_l$ as in the assumption.
	We denote by $v_0^l = o, v_1^l \ldots, v_{n}^l$ the vertices of $S_l$. 
	All the sequences $(v_i^l)$ are contained in a compact subset of $M^\kappa_n$, so up to subsequence they converge to $v_i$ for all $i=0,\ldots,n$, in particular $v_0 = o$.
	%	We recall that for any $\kappa$, in an appropriate model of $M^\kappa_{n}$, we can write: 
	%	$$\text{Conv}(v_0^l, \ldots, v_{n}^l) = \left\lbrace \vartheta_0v_0^l + \ldots \vartheta_{n}v_{n}^l \text{ s.t. } \sum_{i=0}^{n} \vartheta_i = 1, \vartheta_i \geq 0 \right\rbrace.$$
	Then the $\varepsilon$-neighbourhood   $\text{Conv}(v_0, \ldots, v_{n})_\varepsilon$ of $\text{Conv}(v_0, \ldots, v_{n})$  is a convex subset of $M^\kappa_n$ which definitely contains $v_0^l = o, v_1^l \ldots, v_{n}^l$,  hence    
	$$ \text{Conv}(v_0, \ldots, v_{n})_\varepsilon  \supseteq \text{Conv}(v_0^l = o, v_1^l \ldots, v_{n}^l).$$ 
	Analogously   $ \text{Conv}(v_0, \ldots, v_{n}) \subseteq \text{Conv}(v_0^l = o, v_1^l \ldots, v_{n}^l)_\varepsilon$ definitely, hence \linebreak
	$ \text{Conv}(v_0, \ldots, v_{n}) \rightarrow \text{Conv}(v_0^l = o, v_1^l \ldots, v_{n}^l)$ for the  Hausdorff distance. 
	% This implies immediately that $S=\text{Conv}(v_0,\ldots,v_n)$ is the limit in the Hausdorff sense of $\text{Conv}(v_0^l,\ldots,v_n^l)$. From this characterization  it follows 
	Similarly any face of $S$ is limit of corresponding faces of $S_l$. 
	We now claim that $v_0,\ldots, v_n$ are in general position. If not then there are three vertices, say $v_0,v_1,v_2$, belonging to the same $1$-dimensional space. This means   the faces $\text{Conv}(v_0^l,v_1^l,v_2^l)$ tend to a 1-dimensional face, therefore thay cannot have size bounded below uniformly, which contradicts Lemma \ref{size-faces}. Therefore $S$ is a $n$-dimensional simplex. Moreover it is clear it is contained in a ball of radius $R$ and it contains a ball of radius $\frac{1}{R}$.
	Fix now any non-principal ultrafilter $\omega$ and a sequence $S_l$ as above. Each $S_l$ is proper and the sequence converges in the Gromov-Hausdorff sense to the proper space $S$. Then by Proposition \ref{Jansen} we get that the ultralimit $S_\omega$ is isometric to $S$.
\end{proof}
Clearly the same conclusion holds for the class of simplices of dimension at most $n$ and size bounded by $R$ since it is the finite union of compact classes.
From this compactness result we get useful uniform estimates.
\begin{lemma}
	\label{distance-vertices}
	Let $K$ be a $M^\kappa$-complex of size bounded  by $R$ and dim$(K) \leq n$. Then there exists a constant $\varepsilon_0(R, n) > 0$ depending only on $R$ and $n$ such that for  all vertices $v,w$ of $K$ it holds 
	$\varepsilon(v) > \varepsilon_0(R,n)$ and  $d_K(v,w)\geq\varepsilon_0(R,n)$. 
\end{lemma}
\begin{proof}
	The class of simplices with size bounded by $2^{n-d}R$ and dimension exactly $d$ is compact with respect to the Hausdorff distance of $M^\kappa_d$ by   \ref{simplex-compactness}. \linebreak Moreover the map $\text{Conv}(v_0,\ldots,v_d) \mapsto d_{M^\kappa_d}(v_0,\text{Conv}(v_1,\ldots, v_d))$ is continuous with respect to the Hausdorff distance and it is  positive. Therefore it attains a  global minimum $\varepsilon_d>0$. 
	Setting $\varepsilon_0(R,n)=\min_{d=0,\ldots,n} \varepsilon_d$, we have $\varepsilon (v) \geq  \varepsilon_0(R,n)>0$   for every vertex $v\in K$. Therefore every two vertices $v,w$ of $K$ satisfy $d_K (v,w) \geq \varepsilon_0(R,n)$ (or, by   Lemma \ref{epsilon-simplex}, there would exist a simplex   $S$ of $ K$ such that $d_K(v,w)=d_S(v,w)< \varepsilon_0(R,n)$, a contradiction). 
\end{proof}

\begin{lemma} \label{epsilond}
	Let $S$ be a $\kappa$-simplex of size bounded by $R$ and $\textup{dim}(S) \leq  n$.\linebreak
	Let $\partial T_\varepsilon$ denote the $\varepsilon$-neighbourhood of the boundary of any face $T$ of $S$.  \linebreak	 
	%We denote by $S^d$ the union of the $d$-dimensional faces of $S$, and by $S^d_\varepsilon$ its \linebreak $\varepsilon$-neighbourhood.
	For any positive  $\tau $ there exists $\varepsilon (R, n, \tau) > 0$ such that for all faces $T$ of $S$,
	for all $x\in T\setminus \partial T_{\tau}$ and  all the faces $T'$ of $S$ which do not contain $x$   it holds: 
	$$d(x,T') \geq \varepsilon (R,n,\tau)$$
	Moreover, for any integer   $d\geq 0$ there exist   $\eta_d=\eta_d(R,n)$, $\varepsilon_d=\varepsilon_d(R,n)>0$, 
	where  $\varepsilon_0 = \varepsilon_0(R,n)$ is the function given by Lemma \ref{distance-vertices} and $\eta_0=\frac{\varepsilon_0}{8(n+1)}$,  satisfying the following conditions: 
	\begin{itemize}
		\item[(a)]  for all $d$-dimensional faces $T$ of $S$, for every  $x\in T\setminus \partial T_{\eta_{d-1}}$  and every face  $T'$ of $S$  not containing $x$ it holds:  $d(x,T') \geq \varepsilon_d$;
		
		%$x\in S^d \setminus  S^{d-1}_{ \varepsilon_{d-1}}$ it holds $d(x,T)\geq \eta_d$ for all the faces $T$ that do not contain $x$;
		\item[(b)]  $\eta_k + \eta_{k+1}+ \cdots \eta_m \leq \frac{\varepsilon_k}{8}$, for all $0 \leq k \leq m \leq n$.
	\end{itemize}
	
\end{lemma}
\begin{proof}
	The proof follows  same arguments  of Lemma \ref{distance-vertices}. 
	Indeed it is sufficient to consider the positive, lower semicontinuous map
	$$h(S) = \min_{T \text{ face of } S} \inf_{x\in T \setminus \partial T_\tau} \min_{\scriptsize \begin{array}{c} T' \text{ face of } S \\ x\notin T'\end{array}} d(x,T')$$
	on the compact set of $\kappa$-simplices of size bounded by $R$ and dimension at most $n$, and take as  $\varepsilon(R,n,\tau)$ its positive minimum.
	
	To prove the second part of the Lemma we define $\varepsilon_1(R,n)$ as $\varepsilon(R,n,\eta_0)$, where this is the number given by the first statement with $\tau = \eta_0$.
	Then we choose $0<\eta_1 = \min\lbrace \frac{\varepsilon_0}{8(n+1)}, \frac{\varepsilon_1}{8(n+1)} \rbrace$ and again we define $\varepsilon_2>0$ as $\varepsilon(R,n,\eta_1)$. \linebreak
	We can continue choosing   $0<\eta_2 = \min\{\frac{\varepsilon_0}{8(n+1)}, \frac{\varepsilon_1}{8(n+1)},\frac{\varepsilon_2}{8(n+1)} \}$  and so on. This process produces the announced $\varepsilon_i, \eta_i$, which   clearly satisfy    (b).
\end{proof}

As a consequence we get the following   useful estimates (the second of which is similar to Lemma I.7.54 of \cite{BH09}):

\begin{lemma}
	\label{length-simplex}
	Let $K$ be a $M^\kappa$-
	%{\narrowstyle 
	complex of size bounded by $R$ and $\textup{dim}
	%}
	(K) \leq  n$.  
	For all $\tau > 0$ there exists   $\varepsilon (R, n, \tau) > 0$ with the following property:  for all $x\in K$ whose support is $S$ satisfying $d_S(x,\partial S)\geq \tau$ we have $\varepsilon(x) \geq \varepsilon (R, n, \tau)$. 
	In particular  if $K$ is connected then $(K,d_K)$ is a length metric space.
\end{lemma}
\begin{proof}
	%	Using the same technique of Lemma \ref{epsilond} we get the following claim: 
	%	\emph{let $S$ be a $\kappa$-simplex of size bounded by $R$ and $\textup{dim}(S) \leq  n$. 		
	%	For any positive  $\tau $ there exists $\eta > 0$ depending only on $R,n$ and $\tau$ such that for all faces $T$ of $S$,
	%	for all $x\in T\setminus \partial T_{\tau}$ it holds $d(x,T') \geq \eta$ for all the faces $T'$ of $S$ that do not contain $x$.}
	
	Let  $x\in K$. Any simplex   containing $x$ must contain supp$(x)$ as a face. It is then enough to apply the first claim of Lemma \ref{epsilond} to get the estimate on $\varepsilon(x)$. The second part  follows immediately from Lemma \ref{epsilon-simplex}.
\end{proof}

\begin{lemma}
	\label{uniform-c}
	Let $K$ be a $M^\kappa$-
	%{\narrowstyle 
	complex of size bounded by $R$ and $\textup{dim}
	%}	
	(K) \leq  n$. 
	Then there exists $\delta=\delta(R,n) > 0$ depending only on $R$ and $n$ such that:
	%\vspace{-2mm}
	\begin{itemize}
		\item[(a)] if two simplices $S,S'$ of $K$ are at distance $\leq \delta(R,N)$, they share a face;
		%\vspace{-2mm}
		\item[(b)] moreover for every $x\in K$ the ball $\overline{B}(x,\delta)$ is contained in the open star of some vertex; 
		%\vspace{-2mm}
		\item[(c)] finally for every $x\in K$ there exists $y\in K$ such that $\overline{B}(x, \delta) \subset \overline{B}(y,\frac{\varepsilon(y)}{4})$ (where $\varepsilon(y)$ is the function defined in \eqref{epsilon}).
	\end{itemize} 	
\end{lemma}
\begin{proof} 
	We start proving (c). Consider the numbers $\varepsilon_d, \eta_d$ given by  Lemma \ref{epsilond}. 
	%We fix the value $\varepsilon(R,n)$ given by Lemma \ref{distance-vertices} and we apply the previous lemma with $\tau = \varepsilon(R,n)$ in order to find the corresponding positive real numbers $\varepsilon_d, \eta_d$ depending only on $R$ and $n$. 
	The claim is that $\delta=\min_{d=0,\ldots,n}\eta_d$ satisfies the thesis of (c).
	Actually take  any  $x\in K$ and   consider the $d$-dimensional simplex $S=\text{supp}(x)$. \linebreak
	There are two possibilities: either $x\in S \setminus \partial S_{\eta_{d-1}}$ or there exists a point $y_1 \in \partial S$ such that $d(x,y_1) \leq \eta_{d-1}$.
	In the first case we observe that any simplex $S'$ containing $x$ must have $S$ has a face and by Lemma \ref{epsilond} we can conclude that $\varepsilon(x)\geq \varepsilon_d$. Hence in this case $\overline{B}(x,\delta)\subset \overline{B}(x, \frac{\varepsilon_d}{8}) \subset \overline{B}(x, \frac{\eta(x)}{4})$ as follows by Lemma \ref{epsilond}.(b).
	Otherwise let    $S_1= \textup{supp}(y_1)$ and  call $0\leq d_1 \leq d-1$ its the dimension.
	Arguing as before we find that either $\eta(y_1) \geq \varepsilon_{d_1}$ or there exists again a point $y_2$ whose support $S_2$ has dimension $0\leq d_2 < d_1$ such that $d(y_1,y_2)\leq \eta_{d_1 - 1}$.
	In the first case we have	$$\overline{B}(x,\delta)\subset \overline{B}(y_1, \eta_{d-1} + \eta_{d_1}) \subset \overline{B}\left(y_1, \frac{\varepsilon_{d_1}}{4}\right) \subset \overline{B}\left(y_1, \frac{\varepsilon(y_1)}{4}\right),$$
	otherwise we  continue the procedure inductively. 
	Then either at some step we have the thesis or    we find a vertex $v$ of $K$ such that $$d(x,v) \leq \eta_{d-1} + \eta_{d-2} + \ldots + \eta_0 \leq \frac{\varepsilon_0}{8}.$$ 
	Therefore  $\overline{B}(x,\delta)\subset \overline{B}\left(v,\frac{\varepsilon_0(R,n)}{4}\right) \subset  \overline{B}\left(v,\frac{\varepsilon(v)}{4}\right)$, which proves (c).

	\noindent In order to prove (b) we fix $x \in K$ and we find the corresponding $y$ given by (c).
	Then for all point $z\in \overline{B}(x,\delta)$ we can apply Lemma \ref{epsilon-simplex} and find that any simplex $S$ containing $z$ must contain also $y$. This means that any such $S$ has the vertices of supp$(y)$ as vertices. This concludes the proof of (b).	
	
	\noindent	Finally the proof of (a) is an easy consequence: suppose to have two points $x$ and $x'$ belonging to two simplices $S,S'$ respectively such that $d(x,x')\leq \delta$; then they belong to the open star of a same vertex by (b). In particular $S$ and $S'$ share a vertex.
\end{proof}

Another straightforward application of compactness and continuity yields the following, whose proof is omitted:
\begin{lemma}
	\label{link-induction}
	Let $K$ be a $M^\kappa$-
	%{\narrowstyle  
	complex of size bounded by $R$ and $\textup{dim}
	%}
	(K) \leq n$. \linebreak 
	Then there exists $R' = R'(R,n)$ depending only on $R$ and $n$ such that for every vertex $v$ of $K$ the $M^1$-complex \textup{Lk}$(v,K)$ has size bounded by $R'$.
\end{lemma} 

We start now considering  $M^\kappa$-complexes with bounded size and valency:
\begin{prop}
	\label{geodesic-simplex}
	Let $K$ be a connected $M^\kappa$-complex of   size bounded by $R$  and valency at most $N$. Then $K$ is locally finite (i.e. for all $x\in K$ there are a finite number of simplices containing $x$)
	%. Moreover, the number $\varepsilon(x)$ introduced in \eqref{epsilon} is positive for all $x\in K$,  and as a consequence 
	and $(K,d_K)$ is a proper, geodesic metric space.
\end{prop}
\begin{proof}
	%	First of all we will show that the number of simplices containing any point $x\in K$ is bounded by $N$. This is true for vertices by assumption. 
	Any simplex $S$ containing a point $x$ must have supp$(x)$ as a face;  in particular, if $v$ is a vertex of  supp$(x)$, then it is also a vertex of $S$.  So the number of simplices containing $x$ is bounded by the number of simplices containing $v$,  which is bounded by $N$ by assumption.
	%	Moreover, since any point is contained in a finite number of simplices it is clear that $\varepsilon(x)$ is strictly positive for any $x\in K$. Therefore by Lemma \ref{epsilon-simplex} we conclude that $(K,d_K)$ is a length metric space. 
	By Lemma \ref{length-simplex} we know that $(K,d_K)$ is a length metric space. 	Finally, by Lemma \ref{uniform-c}, for all $y \in K$  the ball $\overline{B}(y,\delta)$ belongs to the open star of a vertex, which is the union of a finite number of simplices, hence $K$ is locally compact and   complete. 
	%Then, any Cauchy sequence is definitively contained in some ball $\overline{B}(y,\delta)$ which is compact, so it admits a convergent subsequence. 
	Then, as $K$ is a complete, locally compact, length metric space, it  is proper and geodesic  by Hopf-Rinow's Theorem. 
\end{proof}

The following is the analogue of Theorem I.7.28 of \cite{BH09}:
\begin{prop}
	\label{simplicial-geodesics}
	Let $K$ be a connected $M^\kappa$-complex of size bounded by $R$ \linebreak and valency at most $N$. Then for any $\ell > 0$ there exists $m_0 = m_0(\ell,R,N)$ depending only on $\ell,R$ and $N$ such that any $m$-taut string of length $\leq \ell$ satisfies $m \leq m_0$. 
\end{prop}
\begin{proof}
	We  use the same proof of Theorem I.7.28 of \cite{BH09} (which is  for $M^\kappa$-complexes of {\em finite shape}), proceeding   by induction on the dimension of $K$. \linebreak
	%, which is bounded since the valency is bounded. 
	The first  step is to  prove that if a $m$-string $\Sigma$ is included in the open star of a vertex $v$
	then $m$ is bounded by a function $m_0'(\ell, R, N)$. This is clear with $m'_0=3$  if the geodesic associated to $\Sigma$ passes through $v$, otherwise it follows by the inductive hypothesis by projecting radially $\Sigma$ to \textup{Lk}$(v,K)$ (which has lower dimension) using Lemma \ref{link-induction}.\\
	Now, if the bound stated in the proposition did not hold there would exist tout $m$-strings $\Sigma_i$ in $M^\kappa$-complexes $K_i$ with    length $\leq \ell$ and arbitrary large  $m$. \linebreak
	Then there would exist also tout $m'$-substrings  $\Sigma'_i$ of the $\Sigma_i$, with  $m' > m_0'(\ell, R, N)$, included  in  some ball $\bar B(x_i, \delta) \subset K_i$, for $\delta=\delta(R,N)$ defined in Lemma \ref{uniform-c}. By the same Lemma $\Sigma'_i$ would be included in the open star of some vertex which, by step one, implies  that $m' \leq m_0'(\ell, R, N)$, a contradiction.
	%\vspace{5mm}
	%The first inductive step in \cite{BH09} can be done in the same way using Lemma \ref{link-induction} and the  fact that the valency of the link at a vertex is bounded in terms of $N$. 
	%The second inductive step follows from Lemma \ref{uniform-c}.
\end{proof}

\begin{cor}
	Let $K$ be a connected $M^\kappa$-complex of  size bounded by $R$ \linebreak and valency at most $N$. Let $x,y\in K$ such that $d_K(x,y)\leq \ell$. Then there exists a geodesic joining $x$ to $y$ realized as the concatenation of at most $m_0(\ell, R,N)$ geodesic segments, each contained in a simplex of $K$.
\end{cor}
\begin{proof} Immediate from the fact that $K$ is a geodesic space (by   \ref{geodesic-simplex}), the characterization of  $d_K$ in terms of taut strings and Proposition \ref{simplicial-geodesics}.
	%	We know that $K$ is geodesic by Proposition \ref{geodesic-simplex}. Moreover among the geodesics between two points there are at least a taut-string by Lemma \ref{taut-string-distance}. Therefore by Proposition \ref{simplicial-geodesics} for all $x,y \in K$ at distance $\leq\ell$ then there exists a geodesic connecting $x$ and $y$ realized by the concatenation of at most \linebreak $m_0(\ell, R, N)$ geodesic segments, each contained in a simplex.
\end{proof}

In order to establish if a $M^\kappa$-complex is a locally CAT$(\kappa)$ space we  use the following improvement of a well-known criteria. We recall that the {\em injectivity radius} of a complex $K$, denoted $\rho_{\textup{inj}}(K)$, is defined as the supremum of the  $r \geq 0$ such that any two points of $K$ that are at distance at most $r$   are joined by a unique geodesic.
\begin{prop}
	\label{complex-curvature}
	Let $K$ be a connected $M^\kappa$-complex of size bounded by $R$ and valency at most $N$. 
	The following facts are equivalent:
	\begin{itemize}
		\item[(a)] $(K,d_K)$ is locally \textup{CAT}$(\kappa)$;
		\item[(b)] $K$ satisfies the link condition, i.e. the link at any vertex is   \textup{CAT}$(1)$;
		\item[(c)] $(K,d_K)$ is locally uniquely geodesic;
		\item[(d)] $(K, d_K)$ has positive injectivity radius;
		\item[(e)] $\rho_{\textup{inj}}(K)  \geq \delta (R,N)$, where $\delta(R,N)$ is the function defined  in Lemma \ref{uniform-c}.
	\end{itemize}
	Moreover if $K$ satisfies one of the equivalent conditions above then for any $x \in K$ the ball $B(x,\delta(R,N))$ is a \textup{CAT}$(\kappa)$ space, i.e. the \textup{CAT}$(\kappa)$-radius of $K$ is at least $\delta(R,N)$.	
\end{prop}
\noindent The equivalences between (a), (b) and (c) are quite standard. The equivalence of these conditions with (d) is known for simplicial complexes with finite shapes, see \cite{BH09}. The main point of Proposition \ref{complex-curvature} is that the last equivalence continues to hold in our setting and moreover we can bound from below the injectivity radius of $K$ in terms of $R$ and $N$ only. 
\begin{proof}[Proof of Proposition \ref{complex-curvature}]
	The equivalence between (a) and (b) follows from Theorem II.5.2 and Remark II.5.3 of \cite{BH09}, while   (a) $\Rightarrow$ (c)   is straightforward. 
	The implication (c) $\Rightarrow$ (e) follows as in Proposition I.7.55 of \cite{BH09}.
	Actually  by Proposition \ref{geodesic-simplex} we have $\varepsilon(x)>0$ for every $x \in K$, so  the ball $B(x, \frac{\varepsilon(x)}{2})$	is isometric to the open ball $B(O, \frac{\varepsilon(x)}{2})$    of the $\kappa$-cone  $C_\kappa (\text{Lk}(v,K))$ centered at the cone point $O$ (cp. Theorem I.7.39 in \cite{BH09}). 
	%Therefore,  
	Moreover by assumption a neighbourhood of $O$ of the cone    $C_\kappa  (\text{Lk}(v,K))$ is uniquely geodesic, which implies that 
	%\textup{Lk}$(v,K)$ is $\pi$-uniquely geodesic, and hence 
	the whole  $C_\kappa (\text{Lk}(v,K))$  is uniquely geodesic (cp. Corollary I.5.11, \cite{BH09}) and this in turns  implies that 
	$B(x, \frac{\varepsilon(x)}{2})$  is. By  Lemma   \ref{uniform-c}(c) we conclude  that  the injectivity radius is bounded below by $\delta(R,N)$ (recall that the dimension of $K$ is bounded above by $N$). 	
	The implication (e) $\Rightarrow$ (d) is obvious, while (c) $\Rightarrow$ (b) follows  exactly as in Theorem II.5.4 of \cite{BH09}.
	Finally the last remark follows from Theorem I.7.39 \&Theorem II.3.14 of \cite{BH09} together with Lemma \ref{uniform-c}(c).
\end{proof}

We recall that a locally compact, locally CAT$(\kappa)$ $M^\kappa$-complex is locally geodesically complete if and only if it has   no   {\em free faces} (see II.5.9 and II.5.10 of \cite{BH09} for the definition of having free faces and the proof of this fact).
We can finally show that the class of metric spaces we are studying in this section is uniformly packed.
\begin{prop}
	\label{simplical-packing}
	Let $K$ be a connected $M^\kappa$-complex without free faces,  of size bounded by $R$, valency at most $N$ and positive injectivity radius. \linebreak Then $K$ is a proper, geodesic \textup{GCBA}$^\kappa$-space with $\rho_\textup{cat}(K) \geq \rho_0$  and satisfying $\textup{Pack}(3r_0,\frac{r_0}{2}) \leq P_0$, for  constants  $\rho_0, P_0, r_0$ depending only on $R, N$ and $\kappa$,  and $r_0\leq \rho_0/3$.
\end{prop}
\begin{proof}
	By the proof of Proposition \ref{simplicial-geodesics} we know that $K$ is proper and geodesic.	
	Moreover since the injectivity radius is positive then $K$ is locally CAT$(\kappa)$ and  by Proposition \ref{complex-curvature}   the CAT$(\kappa)$-radius is at least $\rho_0 = \delta(N,R)$. \linebreak Since $K$ has no free faces then it is locally geodesically complete. This shows that $K$ is also a GCBA$^\kappa$-space.  We remark that clearly $\mathcal{H}^k(K) = 0$ if $k> N$ since the projection map from a simplex to $K$ is $1$-Lipschitz; this shows that  there are no points of dimension greater than $N$, i.e. dim$(K)\leq N$.\linebreak
	We now use Lemma \ref{uniform-c} to estimate the number of simplices intersecting a ball around any point  $x \in K$. Any simplex $S$ which intersect $\overline{B}(x,\delta(R,N))$ intersects the open star around some vertex $v$, by Lemma \ref{uniform-c}.(b). Therefore $v$ must be a vertex of $S$. If follows that the number of simplices intersecting $\overline{B}(x,\delta(R,N))$ is bounded by $N$.
	Therefore, for any $x\in K$ we have
	\vspace{-3mm}
	
	$$\mu_K(B(x,\delta(R,N))) \leq \sum_{d=0}^{N}N\cdot\mathcal{H}^d(B_{M^\kappa_d}(o,\delta(R,N))) \leq V_0,$$
	where $V_0$ depends just on $R, N$ and $\kappa$ (here is  $o$ is any point of $M^\kappa_d$). \linebreak
	The conclusion follows from Theorem \ref{char_pack}. 
\end{proof}

\subsection{Compactness of $M^\kappa$-complexes}\label{subsec-compactness}
The aim of this section is to provide compactness and finiteness results for simplicial complexes. We denote by $M^\kappa(R,N)$ the class of $M^\kappa$-complexes withour free faces,  of size bounded by $R$,  valency at most $N$  and positive injectivity radius.
\begin{theo}
	\label{simplicial-compactness}
	The class $M^\kappa(R,N)$ is compact under pointed Gromov-Hausdorff convergence.
\end{theo}
\noindent By Proposition \ref{simplical-packing} there exist $P_0,r_0,\rho_0$ such that any $K\in M^\kappa(R,N)$ belongs to  $\text{GCBA}^\kappa_\text{pack}(P_0,r_0;\rho_0)$. So, by Theorem \ref{precompactness},   the class $M^\kappa(R,N)$ is precompact and it is compact if and only if it is closed under ultralimits. We are going now to show  that $M^\kappa(R,N)$ is closed under ultralimits. \\
We fix a non-principal ultrafilter $\omega$ and we take any sequence $(K_n, o_n)$ in $M^\kappa(R,N)$. We denote by $K_\omega$ the ultralimit of this sequence. Our aim is to prove that $K_\omega$ is isometric to a 
$M^\kappa$-complex $\hat K_\omega$ satisfying the same conditions as the $K_n$'s. 
\vspace{3mm}

\noindent {\em Step 1: construction of the simplicial complex $\hat K_\omega$}.

\noindent Let us start definining who are the simplices of $\hat K_\omega$.
Let   $(x_n)$  be any admissible sequence of points, with $x_n \in K_n$, 
% point $x_\omega \in K_\omega$, we take a representative $(x_n)$ of $x_\omega$. 
and   consider the  unique simplex supp$(x_n)$ of $K_n$ containing  $x_n$ in its  interior:
we define $S_{(x_n)} = \omega$-$\lim \text{supp}(x_n)$. \linebreak 
The metric space $S_{(x_n)}$ is a $\kappa$-complex with size bounded by $R$ by   \ref{simplex-compactness}.\linebreak 
Notice that, a priori, if $y_n$ is another sequence defining the same point as $x_n$ in $K_\omega$ then $S_{(y_n)}$ might be different from $S_{(x_n)}$. 

\noindent Now we define $\hat K_\omega$ as  follows. Let $p_n: S \rightarrow K_n$ denote the projection of any simplex of the total space of $K_n$ to $K_n$.  The total space of $\hat K_\omega$ will be 
$$\bigsqcup_{ (x_n) \textup{ admissible} } S_{(x_n)}$$ 
where $(x_n) $ is {\em any} admissible sequence of points  in $K_n$,  and the equivalence \linebreak relation is: 
if $z_\omega =   \omega$-$\lim  z_n  \in S_{(x_n)}$ and $z'_\omega =   \omega$-$\lim  z'_n  \in S_{(x_n')}$ 
(i.e.  $(z_n)$, $(z_n')$ are  admissible sequences of points {\em respectively   in $\text{supp}(x_n)$ and $\text{supp}(x_n')$}), we say that $z_\omega \sim z_\omega'$ if and only if 
%$\omega$-$\lim d_{K_n}(p_{\text{supp}(x_n)}(z_n), p_{\text{supp}(x_n')}(z_n')) = 0$, 
$\omega$-$\lim d_{K_n}(p_n (z_n), p_n (z_n')) = 0$. 
That is,  we compare the points $z_n$ and $z_n'$ in the common space $K_n$ where they live. \linebreak
%{\narrowstyle 
For simplicity we will abbreviate
%} 
$d_{K_n}(p_n(z_n), p_n(z_n'))$ with $d_{K_n}(z_n, z_n')$. \linebreak
First of all we need to check that the relation   is well defined: given other admissible sequences $(w_n), (w'_n)$ with  $w_n \in \text{supp}(x_n)$ and $w_n'\in \text{supp}(x_n')$ such that  $z_\omega=\omega$-$\lim w_n $ and   $z'_\omega=\omega$-$\lim  w'_n $,  we have 
%  $\omega$-$\lim d_{\text{supp}(x_n)}(z_n,w_n)=0$ and $\omega$-$\lim d_{\text{supp}(x_n')}(z_n',w_n')=0$. Then 
\begin{equation*}
	%\begin{aligned}
	d_{K_n}(w_n, w_n') 
	%\leq \omega \text{-} \lim d_{K_n}(w_n, z_n) + \omega \text{-} \lim d_{K_n}(z_n, z_n') +  \omega \text{-} \lim d_{K_n}(z_n', w_n') \\
	\leq  d_{\text{supp}(x_n)}(w_n, z_n) + d_{K_n}(z_n, z_n') +  d_{\text{supp}(x_n')}(z_n', w_n')  
	%\end{aligned}
\end{equation*}
\normalsize
hence $\omega \text{-} \lim d_{K_n}(w_n, w_n') =0$.
Once proved it is well defined it is easy to show it is an equivalence relation. We call $\hat K_\omega$ the quotient space and denote $p_\omega: S_{(x_n)} \rightarrow  \hat K_\omega$   the projections.
\vspace{2mm}

\noindent {\em Step 2:  $\hat K_\omega$ satisfies axiom  (i) of  $M^\kappa$-complexes.}

\noindent
%We need to prove that conditions (a) and (b) of the definition of a simplicial complex are satisfied by this equivalence relation. \\
%{\em Proof of condition (i).}
We fix an admissible sequence $(x_n)$ and the corresponding simplex $S_{(x_n)}$. \linebreak
We need to prove that the map $p_\omega \colon S_{(x_n)} \to  \hat K_\omega$ is injective. 
For this consider points $z_\omega = \omega \text{-} \lim z_n$ and $z'_\omega=\omega \text{-} \lim z'_n$ in $S_{(x_n)}$, with $z_n, z'_n\in \textup{supp}(x_n)$  for all $n$; 
then  there exists $\varepsilon_0 > 0$ such that $\omega$-$\lim d_{\text{supp}(x_n)}(z_n,z'_n) > \varepsilon_0 > 0$.  
In particular $d_{\text{supp}(x_n)}(z_n,z'_n) > \varepsilon_0 \hspace{1mm}$ $\omega$-a.e.$(n)$.
Now, for any point $z$ of a $M^\kappa$-complex define dim$(z)$   as the dimension of supp$(z)$.
The strategy to prove the injectivity is by induction on 
$$d = \max \lbrace \omega\text{-}\lim \text{dim}(z_n), \; \omega\text{-}\lim \text{dim}(z_n') \rbrace.$$
Observe that if $\omega\text{-}\lim \text{dim}(z_n) = k$ then we have $\text{dim}(z_n) = k$  $\omega$-a.e.$(n)$ because the possible dimensions belong to a finite set. For $d=0$ we have that   $z_n,z_n'$ are both vertices of supp$(x_n)$,  $\omega$-a.e.$(n)$.
If $p_\omega$ is not injective then for every $\varepsilon > 0$ we have $d_{K_n}(z_n,z_n')\leq \varepsilon$  $\omega$-a.e.$(n)$. By Lemma \ref{distance-vertices} we know that if $d_{K_n}(z_n, z_n')\leq \varepsilon_0(R,N)$ then $z_n = z_n'$ as points of supp$(x_n)$. %This shows the injectivity of $p_\omega$ in this case. 
\linebreak
We consider now the inductive step. We denote by $T_n, T_n'$ the faces of $S_n$ containing $z_n$ and $z_n'$ in their interior, respectively. We suppose there exists $\tau > 0$ such that for $\omega$-a.e.$(n)$ it holds $z_n \in T_n \setminus (\partial T_n)_\tau$. By Lemma \ref{length-simplex} we have $\varepsilon(z_n)\geq \varepsilon(R,N,\tau)$  $\omega$-a.e.$(n)$, and similarly for $z'_n$. Once again this fact implies the injectivity.
% Clearly the same argument can be applied with $z_n'$ instead of $z_n$. 
Consider now the case where for all $\tau > 0$ the set
$$\lbrace n\in \mathbb{N} \text{ s.t. } d(z_n,\partial T_n) \leq \tau \text{ and } d(z_n', \partial T_n') \leq \tau \rbrace$$
belongs to $\omega$.
Therefore $\omega$-$\lim d(z_n, \partial T_n) = \omega$-$\lim d(z_n', \partial T_n') = 0$. This means that $z_\omega$ belongs to $\partial T_\omega$ and $z_\omega'$ belongs to $\partial T_\omega'$, by Proposition \ref{simplex-compactness}. Hence $z_\omega = \omega$-$\lim w_n$ and $z_\omega' = \omega$-$\lim w_n'$, where $w_n$ and $w_n'$ belong to a lower dimensional face of $T_n$ and $T_n'$ respectively. We then apply the inductive assumption to get the thesis.

\vspace{3mm}

\noindent {\em Step 3:  $\hat K_\omega$ satisfies axiom  (ii) of $M^\kappa$-complexes.}\\
Consider two simplices $S_{(x_n)}$, $S_{(x_n')}$ and  suppose $p_\omega(S_{(x_n)}) \cap p_\omega(S_{(x_n')}) \neq \emptyset$. \linebreak
This means that for any $\varepsilon > 0$ there exist $y_\omega= \omega \text{-} \lim  y_n$ and  $y'_\omega= \omega \text{-} \lim  y'_n$ with  $y_n \in \text{supp}(x_n)$ and $ y_n' \in \text{supp}(x_n')$ such that $d_{K_n}(y_n,y_n') < \varepsilon$, $\omega$-a.e.$(n)$. 
If $\varepsilon < \delta(R,N)$ then by Lemma \ref{uniform-c}.(a) we know that $\text{supp}(x_n)$ and $\text{supp}(x_n')$ share a face in $K_n$. 
Let then  $T_n\subset \text{supp}(x_n)$ and $T_n' \subset \text{supp}(x_n')$ such faces and  $h_n\colon T_n \to T_n'$ an isometry such that $p_n(z) = p_n(z')$ for $z \in T_n$, $z' \in T_n'$ if and only if $z' = h_n(z)$. By assumption this holds  $\omega$-a.e.$(n)$.
By Proposition \ref{simplex-compactness}
it is easy to see that the metric spaces $T_\omega = \omega$-$\lim T_n$ and  $T_\omega' = \omega$-$\lim T_n'$ are, respectively,  faces of $S_{(x_n)}$ and $S_{(x_n')}$. Moreover the sequence of maps $(h_n)$ defines a limit map $h_\omega \colon T_\omega \to T_\omega'$ which is an isometry, by Proposition \ref{limitmap}.
It remains to show that $p_\omega (z_\omega) = p_\omega (z'_\omega)$, for $z_\omega \in T_\omega$ and $z'_\omega \in T'_\omega$, if and only if 
$h_\omega (z_\omega) = z'_\omega$. 
But given $z_\omega= \omega \text{-} \lim  z_n$ and  $z'_\omega= \omega \text{-} \lim  z'_n$  with $z_n \in \text{supp}(x_n), z'_n \in \text{supp}(x'_n)$
%$(h_n(z_n)) = (z_n')$. we have $p_{(x_n)}((z_n)) = p_{(x_n')}((z_n'))$
we have $p_\omega (z_\omega) = p_\omega (z'_\omega)$ by definition if and only if 
$\omega\text{-}\lim d_{K_n}(p_n(z_n),p_n(z_n')) = 0$. 
This happens if and only if for any $\varepsilon > 0$ the inequality
$$d_{K_n}(p_n(z_n),p_n(z_n')) < \varepsilon$$
holds  $\omega$-a.e.$(n)$.
This means that $d_{K_n}(p_n (h_n(z_n)),p_n(z_n')) <\varepsilon$ holds $\omega$-a.e.$(n)$, 
in particular  $p_\omega (h_\omega (z_\omega)) = p_\omega (z'_\omega)$. By the injectivity of the projection map $p_\omega$ we then obtain $ h_\omega(z_\omega) = z'_\omega$, which is the thesis.
\vspace{3mm}

\noindent {\em Step 4:  $\hat K_\omega$ belongs to $M^\kappa(R,N)$.}\\
It is clear that $\hat K_\omega$ has size bounded by $R$ by construction. \\
We want to show it has valency at most $N$. Fix a vertex $v$ of $\hat K_\omega$ and parameterize by $\alpha \in A$  the set of simplices $S_{(x_n(\alpha))}$ of $\hat K_\omega$ having $v$ as a vertex. \linebreak
%Suppose first  that $A$ has finitely many elements, say $a$.
For any fixed $\alpha \in A$  there is a vertex $v_n (\alpha)$ of $\text{supp}(x_n (\alpha))$ such that the sequence $(v_n (\alpha))$ converges  $\omega$-a.e.$(n)$  to $v$, by Proposition \ref{simplex-compactness}. 
In particular  for all $\alpha, \alpha' \in A$ we  get 
$d_{K_n}(v_n (\alpha), v_n (\alpha')) <  \varepsilon(R, N)$ 
$\omega$-a.e.$(n)$, and then $v_n (\alpha)= v_n (\alpha')$   by Lemma \ref{distance-vertices}.\\
Let now  $S_{(x_n(\alpha))} \neq S_{(x_n(\alpha'))} $  be distinct elements of  $\hat K_\omega$, for $\alpha, \alpha' \in A$. 
Then there exists a vertex of the first simplex $u=\! \omega$-$\lim u_n$, with  $u_n \! \in \textup{supp}(x_n(\alpha))$,
which does not belong to the second one. 
So $d_{K_n} (u_n, \textup{supp}(x_n(\alpha')))>0$ \linebreak $\omega$-a.e.$(n)$, hence 
supp$(x_n(\alpha)) \neq \textup{supp}(x_n(\alpha'))$ $\omega$-a.e.$(n)$.
Therefore  if $\hat K_\omega$ has $m$ different simplices $S_{(x_n(\alpha))}$ sharing the vertex $v$, there also exist $m$ different simplices $\textup{supp}(x_n(\alpha))$ of $K_n$ sharing the same vertex $v_n(\alpha)$,  $\omega$-a.e.$(n)$. \linebreak This contradicts our assumptions if $m>N$.

\noindent  Finally, the fact that $\hat K_\omega$ has positive injectivity radius and has not free faces will follow from the last step, where we prove that  $\hat K_\omega$ and $K_\omega$ are isometric. 
In fact  $K_\omega$ is geodesically complete and locally CAT$(\kappa)$, as ultralimit of complete, geodesically complete, locally CAT$(\kappa)$ spaces with  CAT$(\kappa)$-radius uniform   bounded below;  hence $K_\omega$ (and in turns $\hat K_\omega$) has positive injectivity radius and   no free faces, by Proposition \ref{complex-curvature} and  II.5.9\&II.5.10 of \cite{BH09}.
\vspace{3mm}

\noindent {\em Step 5:  $\hat K_\omega$ is isometric to $K_\omega$.}\\
We define a map $\Phi: K_\omega \rightarrow \hat K_\omega$ as follows. 
Let $y_\omega=\omega$-$\lim y_n$ be the $\omega$-limit of be an admissible sequence $(y_n)$ of $K_n$. 
Any $y_n$ belongs to $\text{supp}(y_n)$: we will denote by $(y_n)_{\text{supp}(y_n)}$  the point, in the ultralimit of the sequence of simplices $\text{supp}(y_n)$,  which is defined by the admissible sequence of points $(y_n)$
\footnote{The notation  stresses the fact that we see $(y_n)_{\text{supp}(y_n)}$ as limit of points in the abstract  simplices $\text{supp}(y_n)$ (not in $K_n$). Namely,    $(y_n)_{\text{supp}(y_n)}$  belongs to the total space of $\hat K_\omega$, while  $y_\omega \in K_\omega$.}.
\linebreak
We then define $\Phi$ as
$$\Phi(y_\omega) = p_\omega ((y_n)_{\text{supp}(y_n)}).$$
It is easy to see it is well defined and surjective. \\
It remains to prove it is an isometry.
Let   $y_n,z_n \in K_n$ define admissible sequences. So the distances $d_{K_n}(y_n,z_n)$ are uniformly bounded by some constant $L$. Therefore by Proposition \ref{simplicial-geodesics} for any $n$ there exists a geodesic between $y_n$ and $z_n$ which is the concatenation of at most $m_0(L,R,N)$ segments, each of them contained in a simplex. Since the number of segments is uniformly bounded we can define a path in $\hat K_\omega$ which is the concatenation of geodesic segments, each contained in a simplex of $\hat K_\omega$, and whose length is the limit of the lengths of the segments in $K_n$. This shows that 
$$d_{\hat K_\omega} (p_\omega ((y_n)_{{\text{supp}(y_n)}}), p_\omega ((z_n)_{{\text{supp}(z_n)}}) \leq \omega\text{-}\lim d_{K_n}(y_n,z_n).$$
In order to prove the other inequality we fix two points $y=p_\omega ((y_n)_{{\text{supp}(y_n)}})$ and   $z=p_\omega ((z_n)_{{\text{supp}(z_n)}})$ of $\hat K_\omega$.
Notice that from the  inequality above we deduce that $\hat K_\omega$ is path-connected. Hence, by Proposition \ref{simplicial-geodesics}, we know that there exists a geodesic between $y$ and $z$ which is the concatenation of at most $m_0(\ell,R,N)$ geodesic segments, each of them contained in a simplex, where $\ell = d_{\hat K_\omega}(x,y)$. These segments cross  finitely many  simplices, each of which  can be seen as the $\omega$-limit of a sequence of simplices in $K_n$. Since the number is finite we can see the union of these simplices of $\hat K_\omega$ as the ultralimit of the union of the corresponding simplices in $K_n$. We can therefore approximate the geodesic in $\hat K_\omega$ with paths in $K_n$ between $y_n$ and $z_n$, whose total length tend to $\ell$.
So
$$d_{\hat K_\omega} (p_\omega ({(y_n)_{\text{supp}(y_n)}}), p_\omega ((z_n)_{{\text{supp}(z_n)}}) \geq \omega\text{-}\lim d_{K_n}(y_n,z_n).$$
which  ends the proof of Theorem \ref{simplicial-compactness}. 
%Indeed $K_\omega$ is a simplicial complex of size bounded by $R$ and valency at most $N$. Moreover $K_\omega$ is geodesically complete and locally CAT$(\kappa)$ as ultralimit of complete, geodesically complete, locally CAT$(\kappa)$ metric spaces with uniform lower bound on the CAT$(\kappa)$-radius. Hence by Proposition \ref{complex-curvature} we conclude that $K_\omega$ has positive injectivity radius and furthermore it has no free faces.
\qed
\vspace{5mm}

We can  specialize this compactness theorem to other families of $M^\kappa$-complexes, as done for $\textup{GCBA}^\kappa_{\textup{pack}}(P_0,r_0;\rho_0)$. Namely consider:
\vspace{1mm}

\noindent -- the subclass $M^\kappa(R,N;\Delta) \subset M^{\kappa}(R,N)$ of complexes with diameter $\leq \Delta$; \\
\noindent -- the class  $M^\kappa(R, V, n)$  of $M^\kappa$-complexes without free faces, with size bounded by $R$, total volume $\leq V$, dimension bounded above by $n$ and  positive injectivity radius.

\begin{obs}
	{\em
		We should  specify the measure on the complexes $K$ of the class $M^\kappa(R;V,n)$ under consideration. 
		Any such space is stratified in subspaces of different dimension, so it is naural to consider the measure which is the sum over $k =0,\ldots,n$ of the $k$-dimensional Hausdorff measure on each $k$-dimensional part. 
		This clearly coincides with the natural measure $\mu_K$ of $K$ seen as   \textup{GCBA}-space. 
	}
\end{obs}

\begin{cor}
	\label{finiteness-complex}
	For any choice of $R$, $n$, $V$, $N$ and $\Delta$, the above classes 
	%only contain compact complexes.  and
	are compact under Gromov-Hausdorff convergence and  contain  only finitely many simplicial complexes up to simplicial homeomorphisms.
\end{cor}

\begin{proof}
	The compactness  of $M^\kappa(R,N;\Delta)$ is clear from  the one of $M^\kappa(R,N)$.
	Moreover, by Proposition \ref{simplical-packing}, we know that any $K\in M^\kappa(R,N;\Delta)$ satisfies the condition $\text{Pack}(3r_0, \frac{r_0}{2}) \leq P_0$ for   constants   $P_0 ,  r_0$ only depending on $R$ and $N$. Furthermore, by Lemma \ref{distance-vertices}, any two vertices of $K$ are $\eta(R)$-separated: in particular the number of vertices of $K$ is bounded above by $\text{Pack}(\frac{\Delta}{2}, \frac{\eta(R)}{2})$ which is a number depending only on $R,N,\kappa$ and $\Delta$. \linebreak Since the valency is bounded and the total number of vertices is bounded, we have only  finitely many possible simplicial complexes up to simplicial homeomorphisms. \\
	On the other hand it is straightforward to show that any  $K \in M^\kappa(R;V,n)$ has valency bounded from above by a function depending only on $R,V,n$ and $\kappa$,   because any simplex of locally maximal dimension contributes to the total volume with a quantity greater than a universal function  $v(R,n,\kappa) > 0$.  \\
	This also shows also that the total number of simplices  of $K$ is uniformly bounded in terms of $R,V$ and $n$, hence the combinatorial finiteness of \linebreak $M^\kappa(R;V,n)$. 
	Moreover, since any simplex has uniformly bounded size, also  the diameters of complexes in this class are uniformly bounded. 
	Therefore $M^\kappa(R;V,n) \subset M^\kappa(R,N)$ for a suitable $N$ and, as the class is made of compact metric spaces,   it is actually precompact under (unpointed) Gromov-Hausdorff convergence. 
	It remains to show that $M^\kappa(R;V,n)$ is closed. \linebreak
	By the proof of Theorem \ref{simplicial-compactness} it is clear that the upper bound on the dimension of the simplices is preserved under limits.
	The stability of the upper bound on the total volume is proved as for the class $\text{GCBA}^\kappa_\text{vol}(V_0, R_0; \rho_0, n_0)$ in Corollary \ref{compactness-pointed}.
	%we use the same argument used by \cite{LN19} in the proof of Theorem 1.5, see also the proof of Theorem \ref{compactness-volume}.	 
\end{proof}

Finally, we want to point out  that   the assumptions on size and diameter  in the above compactness results are essential:

\begin{exx}{\em Non-compact families of $M^\kappa$-complexes.}
	\label{complex-examples}
		\vspace{1mm}
	
	\begin{itemize}
		\item[(1)] Let $X_n$  be a wedge of $n$ circles of radius $1$. The family  of $M^0$-complexes   $\{ X_n\} $ has  uniformly bounded size and uniformly bounded diameter, but the valency is not bounded. 
		Notice that this family is neither finite  nor uniformly packed. In particular, it is not precompact.
		\vspace{1mm}
		\item[(2)] Let $X_n$ be obtained  from a circle of radius $1$, then choosing $n$ equidistant points on the circle and gluing $n$ circles of radius $1$ to them. The $X_n$'s admit $M^0$-complex structures with uniformly bounded valency and uniformly bounded diameter, but the size of the simplices is not bounded. 
		Again, this family is neither finite  nor uniformly packed,  hence not precompact.
		\vspace{-2mm}
	\end{itemize}
\end{exx}

\vspace{5mm}
\appendix
\section{Ultralimits}
An {\em ultrafilter} on $\mathbb{N}$ is a subset $\omega$ of $\mathcal{P}(\mathbb{N})$ such that:
\begin{itemize}
	\item[1)] $\emptyset \notin \omega$;
	\item[2)] if $A,B \in \omega$ then $A\cap B \in \omega$;
	\item[3)] if $A\in \omega$ and $A\subset B$ then $B\in \omega$;
	\item[4)] for any $A\subset \mathbb{N}$ then either $A\in \omega$ or $A^c \in \omega$.
\end{itemize}
We recall that there is a one-to-one correspondence between the ultrafilters $\omega$ on $\mathbb{N}$ and the finitely-additive measures defined on the whole $\mathcal{P}(\mathbb{N})$ with values on $\lbrace 0, 1 \rbrace$ such that $\omega(\mathbb{N}) = 1$. 
Indeed given an ultrafilter $\omega$ we define the measure $\omega(A)=1$ if and only if $A\in \omega$; conversely  given a measure $\omega$ as before we define the ultrafilter as the set $\omega = \lbrace A \subset \mathbb{N} \text{ s.t. } \omega(A)=1\rbrace$ \linebreak (it is easy to show it actually is an ultrafilter). In the following $\omega$ will denote both an ultrafilter and the measure that it defines. Therefore we will write that a property $P(n)$ holds $\omega$-a.s. when the set $\lbrace n \in \mathbb{N} \text{ s.t. }P(n)\rbrace \in \omega$. \\
There is an easy example of ultrafilter:   fix $n\in \mathbb{N}$ and   consider the set $\omega$ of subsets of $\mathbb{N}$ containing $n$. An ultrafilter of this type is called {\em principal}. \linebreak 
The interesting ultrafilters are the non-principal ones;  it turns out that an ultrafilter is non-principal if and only if it does not contain any finite set. The existence of non-principal ultrafilters follows from Zorn's lemma.\linebreak
The interest on non-principal ultrafilters is due to the fact that they can define a notion of limit of a bounded sequence of real numbers:

\begin{lemma}
	\label{ultralimit}
	Let $a_n \in [a,b]$ be a bounded sequence of real numbers. \linebreak 
	Let $\omega$ be a non-principal ultrafilter. Then there exists a unique point $x$ in  $[a,b]$   such that for all $\eta > 0$ the set $\lbrace n \in \mathbb{N} \text{ s.t. } \vert a_n - x \vert < \eta\rbrace$ belongs to $\omega$. \linebreak 
	The real number $x$ is said the $\omega$-limit of the sequence $(a_n)$ and it is denoted by $x=\omega$-$\lim a_n$.
	Moreover if $a_n$ and $b_n$ are two bounded sequence of real numbers, it holds:
	\begin{itemize}
		\item[(a)] $\omega\text{-}\lim (a_n + b_n) = \omega\text{-}\lim a_n + \omega \text{-}\lim b_n $;
		\item[(b)] if $\lambda \in \mathbb{R}$ then $\omega\text{-}\lim (\lambda a_n) = \lambda \cdot \omega\text{-}\lim a_n$;
		\item[(c)] if $a_n \leq b_n$ then $\omega\text{-}\lim a_n \leq \omega\text{-}\lim b_n$;
		\item[(d)] if $a\! =\! \omega$-$\lim a_n$ and $f$ is  continuous at $a$ then 
		$\omega   \text{-}    \lim f(a_n)\!  =  f(\omega \text{-}\lim a_n).$
	\end{itemize}
\end{lemma}
\noindent (The proof of the main part can be found in \cite{DK18}, Lemma 7.23, while properties (a)-(d) are trivial.) \\
The ultralimit of unbounded sequences of real numbers can be defined in the following way. Given an unbounded sequence of real numbers $a_n$ the following mutually exclusive situations can occur:
\begin{itemize}
	\item there exists $L > 0$ such that $a_n \in [-L,L]$ for $\omega$-a.e.$(n)$. \\In this case the ultralimit of $(a_n)$ can be defined   using Lemma \ref{ultralimit}.
	\item for any $L > 0$ the set $\lbrace n \in \mathbb{N} \text{ s.t. } a_n \geq L\rbrace$ belongs to $\omega$. \\ In this case we set $\omega$-$\lim a_n = +\infty$.
	\item for any $L < 0$ the set $\lbrace n \in \mathbb{N} \text{ s.t. } a_n \leq -L\rbrace$ belongs to $\omega$.\\ In this case we set $\omega$-$\lim a_n = -\infty$.
\end{itemize}
\noindent We remark that the limit depends strongly on the non-principal ultrafilter $\omega$. \linebreak
The ultralimit of a sequence of metric spaces is defined as follows.

\begin{defin}
	Let $(X_n,x_n)$ be a sequence of pointed metric spaces and $\omega$ be a non-principal ultrafilter. We set:
	$$X = \lbrace (y_n) \, : \, y_n \in X_n \text{ and } \exists L>0 \text{ s.t. } d(y_n,x_n)\leq L \text{ for any } n\rbrace.$$
	and, for $(y_n),(z_n) \in X$, we define the distance as:
	$$d((y_n),(z_n)) = \omega\text{-}\lim d(y_n,z_n).$$
	The space $X_\omega = (X,d)/_{d=0}$ is a metric space and it is called the {\em $\omega$-limit of the sequence of spaces $(X_n,x_n)$}. The fact that $(X,d)$ is a metric space follows immediately from the properties of the ultralimit of a sequence of real numbers and from the fact that $d_n$ is a distance for any $n$.
	In general the limit depends on the non-principal ultrafilter $\omega$ and on the basepoints.
\end{defin}

A basic example is provided by the ultralimit of a constant sequence.
\begin{prop}
	\label{ultralimit-proper}
	Let $(X,x)$ 
	%{\narrowstyle 
	be a metric space and
	%} 
	$\omega$   
	%{\narrowstyle 
	a non-principal ultrafilter.
	%} 
 Consider the constant sequence $(X,x)$ and the corresponding ultralimit $(X_\omega, x_\omega)$, where $x_\omega$ is the constant sequence of points $(x)$. Then
		\vspace{1mm}
	
	\begin{itemize}
		\item[(a)] The map $\iota\colon (X,x) \to (X_\omega, x_\omega)$ that sends $y$ to the constant sequence $(y_n=y)$ is an isometric embedding; 
			\vspace{1mm}	
		
		\item[(b)] if $X$ is proper then $\iota$ is surjective and $(X_\omega, x_\omega)$ is isometric to $(X, x)$.
	\end{itemize}
\end{prop}
\begin{proof}
	The first part is obvious by the definitions. If $X$ is proper and $(y_n)$ is an admissible sequence defining a point of $X_\omega$ then it is contained in a closed ball of $X$, that is compact. By Lemma 7.23 of \cite{DK18} we find $y\in X$ such that for all $\varepsilon > 0$ the set
	$$\lbrace n \in \mathbb{N}\text{ s.t. } d(y,y_n)<\varepsilon \rbrace$$
	belongs to $\omega$. Therefore it is clear that the constant sequence $(y_n=y)$ defines the same point as the sequence $(y_n)$ in $X_\omega$, which proves (b).
\end{proof}

An interesting consequence of the definition is that the ultralimit of pointed metric spaces is always complete (the proof is given in \cite{DK18}, Proposition 7.44):
\begin{prop}
	Let $(X_n, x_n)$ be a sequence of pointed metric spaces and let $\omega$ be a non-principal ultrafilter. Then $X_\omega$ is a complete metric space.
\end{prop}

Once defined the limit of pointed metric spaces it is useful to define limit of maps.
We take two sequences of pointed metric spaces $(X_n,x_n)$ and $(Y_n,y_n)$. A sequence of maps $f_n\colon X_n \to Y_n$ is said {\em admissible} if there exists $M\in \mathbb{R}$ such that $d(f_n(x_n),y_n)\leq M$ for any $n\in \mathbb{N}$. In general an admissible sequence of maps does not define a limit map, but it is the case if the maps are equi-Lipschitz.
A sequence of maps $f_n\colon X_n \to Y_n$ is equi-Lipschitz if there exists $\lambda\geq 0$ such that $f_n$ is $\lambda$-Lipschitz for any $n$.
\begin{prop}
	\label{limitmap}
	Let $(X_n,x_n)$, $(Y_n,y_n)$ be two sequences of pointed metric spaces. Let $f_n\colon X_n \to Y_n$ be an admissible sequence of equi-Lipschitz maps. Let $\omega$ be a non-principal ultrafilter. Let $X_\omega$ and $Y_\omega$ be the $\omega$-limits of $(X_n,x_n)$ and $(Y_n,y_n)$ respectively.
	Define $f=f_\omega \colon X_\omega \to Y_\omega$ as $f((z_n)) = (f_n(z_n))$.
	Then:
	\begin{itemize}
		\item[a)] $f$ is well defined;
		\item[b)] $f$ is Lipschitz with the same constants of the sequence $f_n$.
	\end{itemize}
	In particular if for any $n$ the map $f_n$ is an isometry then $f$ is an isometry, while if $f_n$ is an isometric embedding for any $n$ then $f$ is again an isometric embedding.
\end{prop}
\noindent The map $f=f_\omega$ is called the {\em $\omega$-limit of the sequence of maps $f_n$} and we denote it by $f_\omega = \omega$-$\lim f_n$. The proof in case of isometric embeddings is given in \cite{DK18}, Lemma 7.47; the general case is analogous.
%{\color{red} Commentare dimostrazione e mettere riferimento DK18}
%\begin{proof}
%	First of all $f$ is well defined. Indeed let $(z_n) \in X_\omega$, i.e. $d(z_n,x_n)\leq L$ for any $n$. Since $(f_n)$ is $\lambda$-equi-Lipschitz, then
%	\begin{equation*}
%	\begin{aligned}
%		d(f_n(z_n),y_n) &\leq d(f_n(z_n),f_n(x_n)) + d(f_n(x_n), y_n) \\
%		&\leq \lambda d(z_n,x_n) + M \leq \lambda L + M
%	\end{aligned}
%	\end{equation*}
%	So the sequence of points $(f_n(z_n))$ defines a point of $Y_\omega$.
%	We have to show that the map is well-defined, i.e. $\omega$-$\lim d(f_n(z_n), f_n(w_n))=0$ when $\omega$-$\lim d(z_n, w_n)=0$. This is obvious since
%	$$\omega\text{-}\lim d(f_n(z_n), f_n(w_n)) \leq \omega\text{-}\lim \lambda d(z_n, w_n) = \lambda \omega\text{-}\lim \lambda d(z_n, w_n) = 0.$$
%	This shows also that $f$ is $\lambda$-Lipschitz. 
%	Finally it is easy to check that if $f_n$ is an isometry for any $n$ and $(f_n)$ is a sequence of admissible maps, then $(f_n^{-1})$ is a sequence of admissible maps and $(f_n)^{-1} = (f_n^{-1})$. In the same way it is easy to show that limit of isometric embeddings is an isometric embedding.
%\end{proof}
\vspace{2mm}

This result can be applied to the special case of geodesic segments, since they are isometric embeddings of an interval into a metric space $X$. \linebreak However we first need to explain what is the ultralimit of a sequence of intervals:

\begin{lemma}
	\label{ultralimitintervals}
	Let $I_n = [a_n,b_n]\subset \mathbb{R}$ be a sequence of  intervals   containing $0$ (possibly with $a_n=-\infty$ or $b_n=+\infty$). Let $\omega$ be a non-principal ultrafilter. Then $\omega$-$\lim (I_n,0)$ is isometric to $I$, where $I = [\omega\text{-}\lim a_n, \omega\text{-}\lim b_n]=[a,b]$ \linebreak (possibly with $a=-\infty$ or $b=+\infty$) contains 0.
\end{lemma}
\begin{proof}
	We define a map from $I_\omega$ to $I$ as follows. Given an admissible sequence $(x_n)$ such that $x_n\in I_n$ then $x_n$ is $\omega$-a.s. bounded, so it is defined $\omega$-$\lim x_n$ by Lemma \ref{ultralimit}. We define the map as $(x_n) \mapsto \omega$-$\lim x_n$. It is easy to check it is surjective. Moreover it is an isometry, indeed:
	$$\vert \omega\text{-}\lim x_n - \omega\text{-}\lim y_n \vert = \omega\text{-}\lim\vert x_n - y_n \vert = d((x_n),(y_n)).$$
	
	\vspace{-6mm}
\end{proof}

\noindent In particular the limit of geodesic segments is a geodesic segment.
\begin{lemma}
	\label{ultralimit-geodesic}
	Let $(X_n,x_n)$ be a sequence of pointed metric spaces and let $\omega$ be a non-principal ultrafilter. Let $X_\omega$ be the ultralimit of $(X_n,x_n)$ and let $z=\omega$-$\lim z_n$, $w=\omega$-$\lim w_n \in X_\omega$. Suppose that for all $n$ there exists a geodesic $\gamma_n \colon [0,d(z_n,w_n)] \to X_n$ joining $z_n$ and $w_n$: then there  exists a geodesic   joining $z$ and $w$ in $X_\omega$. In particular  if  $X_n$ is a geodesic space for all $n$, then the ultralimit  $X_\omega$ is a geodesic space.
\end{lemma}
\begin{proof}
	We denote by $I_n$ the interval $[0,d(z_n,w_n)]$. Since $z$ and $w$ belongs to $X_\omega$ then the distance between them is uniformly bounded. Hence from the previous lemma it follows that the ultralimit of the spaces $(I_n,0)$ is $I_\omega = [0, \omega$-$\lim d(z_n,w_n)] = [0,d(z,w)]$.
	The maps $\gamma_n$ define an admissible sequence of isometric embedding, so in particular they define a limit isometric embedding $\gamma_\omega \colon I_\omega \to X$. So $\gamma_\omega$ is a geodesic and clearly $\gamma_\omega(0)=\omega$-$\lim\gamma_n(0)=\omega$-$\lim z_n=z$ and $\gamma_\omega(d(z,w)) = w$.
\end{proof}
In order to prove stability results for classes of metric spaces we also need to establish the convergence of balls under ultralimits:
\begin{lemma}
	\label{ultralimitballs}
	Let $(X_n,x_n)$ be a sequence of geodesic metric spaces and $\omega$ be a non-principal ultrafilter. Let $X_\omega$ be the ultralimit of the sequence $(X_n,x_n)$. Let $y=\omega$-$\lim y_n$ be a point of $X_\omega$. Then for any $R\geq 0$ it holds
	$$\overline{B}(y,R) = \omega\text{-}\lim \overline{B}(y_n,R).$$
\end{lemma}
\begin{proof}
	First of all $\omega\text{-}\lim \overline{B}(y_n,R) \subset \overline{B}(y,R)$. Indeed $z = \omega$-$\lim z_n$ belongs to $\omega\text{-}\lim \overline{B}(y_n,R)$ if and only if $d(z_n,y_n)\leq R$ for $\omega$-a.e.$(n)$. Then $d(z,y)\leq R$, i.e. $z\in\overline{B}(y,R)$. The next step is to show that the set $\omega\text{-}\lim \overline{B}(y_n,R)$ is closed. We take a sequence $z^k = \omega$-$\lim z_n^k$ of points of $\omega$-$\lim \overline{B}(y_n,R)$ that converges to some point $z=\omega$-$\lim z_n$ of $X_\omega$. This implies that $d(y,z)\leq R$. We consider a geodesic segment of $X_n$ between $y_n$ and $z_n$ and we denote by $w_n$ the point along this geodesic at distance exactly $R$ from $y_n$, if it exists. Otherwise $z_n \in \overline{B}(y_n,R)$ and in this case we set $w_n = z_n$. We observe that $w=\omega$-$\lim w_n \in \omega$-$\lim \overline{B}(y_n,R)$ by definition. We claim that $w=z$. In order to prove that we fix $\varepsilon > 0$. Then $\omega$-a.s. $d(y_n,z_n)<R+\varepsilon$. This implies that $d(y_n,w_n)< \varepsilon$. Since it holds $\omega$-a.s. then $d(w,z)<\varepsilon$. From the arbitrariness of $\varepsilon$ the claim is proved.
	The last step is to show that the open ball $B(y,R)$ is contained in $\omega\text{-}\lim \overline{B}(y_n,R)$. Indeed given $z=\omega$-$\lim z_n \in B(y,R)$ then there exists $\varepsilon > 0$ such that $d(z,y)<R-\varepsilon$. The set of indices $n$ such that $d(z_n,y_n)< d(z,y) + \varepsilon < R$ belongs to $\omega$, hence $z\in \omega\text{-}\lim \overline{B}(y_n,R)$.
	Since $X_\omega$ is geodesic and in any length space the closed ball is the closure of the open ball the proof is concluded.
\end{proof}
In general, even if every space $X_n$ is uniquely geodesic, the ultralimit $X_\omega$ may be not uniquely geodesic. This is because, in general, it is not true that all the geodesics of $X_\omega$ are limit of sequences of geodesics of $X_n$. \linebreak The fact that the geodesics of $X_\omega$ are actually limit of geodesics of the spaces $X_n$ is true  when all the   $X_n$ are CAT$(\kappa)$. We recall the following fact which is well known (see \cite{BH09} or \cite{DK18} for instance):
\begin{prop}
	\label{ultralimitCAT}
	Let $(X_n,x_n)$ be a sequence of \textup{CAT}$(\kappa)$ pointed metric spaces and $\omega$ be a non-principal ultrafilter. Then any geodesic of length $<D_\kappa$ in $X_\omega$ is limit of a sequence of geodesics of $X_n$. As a consequence $X_\omega$ is a \textup{CAT}$(\kappa)$ metric space.
\end{prop} 

The main result of the appendix is the following stability property for the \textup{CAT}$(\kappa)$-radius:
\begin{cor}
	\label{ultralimit-stability}
	Let $(X_n, x_n)$ be a sequence of complete, locally geodesically complete, locally \textup{CAT}$(\kappa)$, geodesic metric spaces with $\rho_\textup{cat}(X_n) \geq \rho_0 > 0$. Let $\omega$ be a non-principal ultrafilter. Then $X_\omega$ is a complete, locally geodesically complete, locally \textup{CAT}$(\kappa)$, geodesic metric space with $\rho_\textup{cat}(X_\omega)\geq \rho_0$.
\end{cor}
\begin{proof}
	Let $y=\omega$-$\lim y_n$ be a point of $X_\omega$. For any $r<\rho_0$ and for any $n$ the ball $\overline{B}(y_n,r)$ is a CAT$(\kappa)$ metric space. Moreover by Lemma \ref{ultralimitballs} we have that $\overline{B}(y,r)$ is the ultralimit of a sequence of CAT$(\kappa)$ metric spaces, hence it is CAT$(\kappa)$ by Proposition \ref{ultralimitCAT}.
	This shows that $X_\omega$ is locally CAT$(\kappa)$ and $\rho_\text{cat}(X_\omega)\geq \rho_0$ by the arbitrariness of $r$.
	Moreover $X_\omega$ is geodesic by Corollary \ref{ultralimit-geodesic}.
	We fix now a geodesic segment $\gamma$ of $X_\omega$ defined on $[a,b]$. \linebreak 
	We look at the ball $B(\gamma(a), \rho_0)$, which is CAT$(\kappa)$, and we take a sequence of points $z_n$ such that $\omega$-$\lim z_n =\gamma(a)$. The subsegment of $\gamma$ inside this ball, defined on $[a, a+\rho_0)$ is the limit of a sequence of geodesics $\gamma_n$ inside the corresponding balls $B(z_n,\rho_0)$, by Proposition \ref{ultralimitCAT}. Each $\gamma_n$ can be extended to a geodesic segment $\tilde{\gamma}_n$ on the interval $(a - \rho_0, a+ \rho_0)$ since each $X_n$ is locally geodesically complete and complete. The ultralimit of the maps $\tilde{\gamma}_n$ is a geodesic segment defined on $[a- \rho_0, a + \rho_0]$ which extends $\gamma$. We can do the same around $\gamma(b)$. This proves that $X_\omega$ is locally geodesically complete.
\end{proof}

We conclude the appendix recalling the relations between ultralimits and pointed Gromov-Hausdorff convergence, which we will use in Section \ref{sec-compactness}:

\begin{prop}[see \cite{Jan17}]
	\label{Jansen}
	Let $(X_n,x_n)$ be a sequence of proper, length metric spaces and $\omega$ be a non-principal ultrafilter. Then:
	\begin{itemize}
		\item[(a)] if the ultralimit $(X_\omega, x_\omega)$ is proper then it is the limit of a convergent subsequence in the pointed Gromov-Hausdorff sense;
		\item[(b)] reciprocally, if $(X_n,x_n)$ converges to $(X,x)$ in the pointed Gromov-Hausdorff sense then for any non-principal ultrafilter $\omega$ the ultralimit $X_\omega$ is isometric to $(X,x)$ (we recall that, in this case,  $(X,x)$ is proper by definition of Gromov-Hausdorff convergence).
	\end{itemize} 
\end{prop}

\bibliographystyle{alpha}
\bibliography{Packing_and_doubling_in_metric_spaces_with_curvature_bounded_above.bib}

\end{document}